\documentclass[11pt,reqno]{amsart} 

\usepackage[utf8]{inputenc}
\usepackage{bbm}
\usepackage{hyperref}
\usepackage{enumitem}
\usepackage[margin=2cm]{geometry}
\usepackage{amsfonts,amsmath,amsthm,amssymb}
\usepackage{graphicx}
\usepackage{siunitx}
\usepackage{mathrsfs} 
\usepackage{verbatim}

\numberwithin{equation}{section}
\newtheorem{theorem}{Theorem}[section]
\newtheorem{corollary}[theorem]{Corollary}
\newtheorem{lemma}[theorem]{Lemma}

\theoremstyle{definition}
\newtheorem{definition}[theorem]{Definition}
\newtheorem{remark}[theorem]{Remark}

\newcommand{\ep}{\varepsilon}

\newcommand{\R}{\mathbb{R}}
\newcommand{\C}{\mathbb{C}}
\newcommand{\Z}{\mathbb{Z}}

\title{Desingularization of small moving corners for the Muskat equation}

\author[Eduardo Garc\'ia-Ju\'arez]{Eduardo Garc\'ia-Ju\'arez}
\author[Javier G\'omez-Serrano]{Javier G\'omez-Serrano}
\author[Susanna V. Haziot]{Susanna V. Haziot}
\author[Beno\^it Pausader]{Beno\^it Pausader}

\date{\today}

\begin{document}
	
	\maketitle

 \begin{abstract}

In this paper, we investigate the dynamics of solutions of the Muskat equation with initial interface consisting of multiple corners allowing for linear growth at infinity. Specifically, we prove that if the initial data contains a finite set of small corners then we can find a precise description of the solution showing how these corners desingularize and move at the same time.

At the analytical level, we are solving a small data critical problem which requires renormalization. This is accomplished using a nonlinear change of variables which serves as a logarithmic correction and accurately describes the motion of the corners during the evolution. 
  
\end{abstract}

 \setcounter{tocdepth}{1}
	\tableofcontents

\section{Introduction}
The Muskat problem models the interactions between two immiscible, incompressible fluids propagating through porous media. Their motion is governed by the experimental Darcy's law under the restoring force of gravity. Since the interface between the two fluid regions is an unknown which needs to be solved for as part of any solution, this is a free boundary problem. Its formulation can be reduced to an integral evolution equation on the interface and in the case when the viscosities of the two fluids are equal and the profile of the interface is graphical, this formulation has a particularly compact form.  In the last two decades, the Muskat problem has generated a rapidly growing interest and has been studied extensively. A great majority of the work carried out on this problem has been done in the sub-critical setting, which is well-understood by now.       

In this paper we study the small-data critical theory for the Muskat equation and more precisely, we seek to understand the behavior of an interface whose initial data consists of a superposition of a finite number of small corners. Despite the fact that the problem is quasi-linear, its parabolic nature enables us to construct our solutions using semi-linear methods such as a renormalization process and a fixed-point argument. This allows us to get a good understanding of the behavior and shape of the corners as well as of the spaces in which they live. Moreover, we can even explicitly calculate how they move as they desingularize. 
\subsection{Presentation of the problem}
We denote by $\Omega^\pm$ the two fluid domains in the $(x,y)$-plane, separated by the graphical interface $\eta(t,x)$. Since the fluids are propagating through porous media, the fluid velocities $u^\pm$ and pressures $p^\pm$ satisfy Darcy's law in their respective domains 
\begin{equation*}
	\begin{aligned}
	\nabla\cdot u^\pm&=0&\qquad&\text{in }\Omega^\pm,\\
	\mu^\pm u^\pm&=-\nabla_{x,y}p^\pm-(0,\rho\pm)&\qquad&\text{in }\Omega^\pm,
	\end{aligned}
\end{equation*}
where $\mu^\pm>0$ and $\rho^\pm>0$ denote the viscosity and density constants in $\Omega^\pm$. Moreover, we assume that $\rho^->\rho^+$, ensuring that the denser fluid lies below. Denoting by $$n=\frac{1}{\sqrt{1+(\partial_x\eta)^2}}(-\partial_x\eta,1)$$ 
the upward pointing normal vector on the interface, these equations are coupled with the dynamic boundary conditions
\begin{equation*}
	\begin{aligned}
	u^+\cdot n&=u^-\cdot n&\qquad&\text{on }\eta(t,x)\\
	p^+&=p^-&\qquad&\text{on }\eta(t,x),
	\end{aligned}
\end{equation*}
guaranteeing continuity of the normal velocity fields and pressures of the fluids across the interface, as well as the kinematic boundary condition
\begin{equation*}
	\partial_t\eta=\sqrt{1+(\partial_x\eta)^2}u^-\cdot n\qquad\text{on }\eta(t,x),
\end{equation*}
ensuring that the interface moves with the fluids. 
\begin{figure}
	\centering
	\includegraphics[scale=0.7]{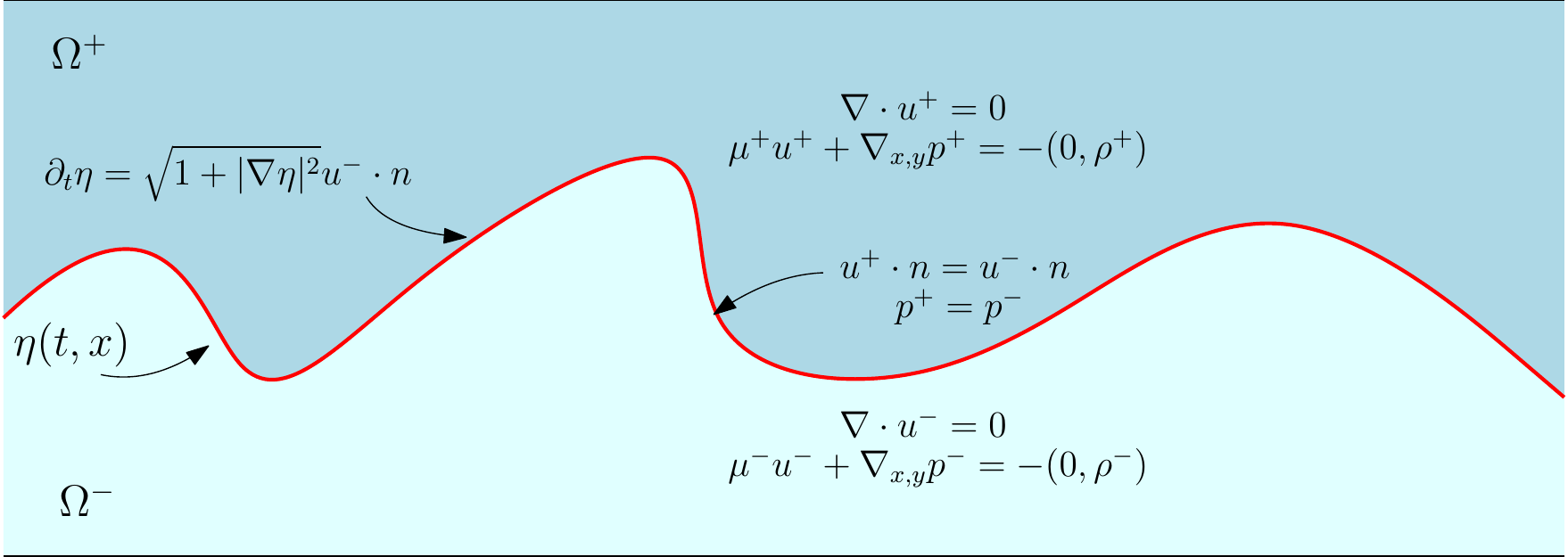}
	\caption{The Muskat problem. The interface $\eta$ separates the two fluids.}
		\label{fig:muskat}
\end{figure}
The Muskat problem can be reformulated as an integral evolution problem on the interface and specifically, in the case of equal viscosities $\mu^-=\mu^+$, it admits a particularly elegant and compact formulation. Upon renormalizing all physical constants we get the one-dimensional evolution equation of the form
\begin{equation}\label{form1}
	\partial_t\eta=\frac{1}{\pi}\int_{\R}\frac{\partial_{x}\Delta_\alpha\eta}{1+(\Delta_\alpha\eta)^2}d\alpha,\qquad\Delta_\alpha \eta(x,\alpha):=\frac{\eta(x)-\eta(x-\alpha)}{\alpha},
\end{equation}
which we refer to as the Muskat equation.
The problem admits a scaling invariance in the sense that if $\eta(t,x)$ is a solution to the problem, then so is $\lambda^{-1}\eta(\lambda t,\lambda x)$ for any $\lambda>0$. Spaces like $\dot{H}^{3/2}$ and $\dot{W}^{1,\infty}$ are thus critical. For the rest of this paper, it will be convenient to work with the slope of the interface $h(t,x):=\partial_{x}\eta(t,x)$. We can hence rewrite \eqref{form1} as the following initial value problem
\begin{equation}\label{eq0}
\begin{split}
	\partial_th(t,x)&=\frac{1}{\pi}\frac{d}{dx}\int_{\R}\frac{\partial_{x}h^*(t,x,\alpha)}{1+(h^*(t,x,\alpha))^2}d\alpha\qquad\text{with}\qquad h^*(t,x,\alpha):=\frac{1}{\alpha}\int_{x-\alpha}^{x}h(t,y)\,dy,\\h(0,\cdot)&=h_0.
 \end{split}
\end{equation}

Our goal is to prove local well-posedness of \eqref{eq0} in a critical space that allows for discontinuous initial data and to describe the evolution of such discontinuities:
\begin{equation*}
\begin{aligned}
h_0(x)=\sum_{j\in\mathcal{J}}h_{j,0}(x),
\end{aligned}
\end{equation*}
where the integers $j\in\mathcal{J}:=\{1,\ldots,M\}$ index elements in the finite set of points $\mathcal{C}=\{a_1,\dots, a_M\}$ and $h_{j,0}$ has a  discontinuity at $x=a_j\in\R$ . To analyze the problem it is important to allow for a renormalization procedure 
\begin{equation}\label{h_g}
\begin{aligned}
h(t,x)&=g(t,x+q(t,x)).
\end{aligned}
\end{equation}
We require the function $q$ to satisfy the bounds
\begin{equation}\label{q_bounds1}
\begin{aligned}
|q(t,x)|\lesssim \ep\, t\ln(2/t)\Phi_{\leq 0}(x),\qquad
\sup_{x\in\R}|\partial_x q(t,x)|\lesssim \ep,\qquad \partial_tq(t,x)=0\quad\text{for }t\geq 1,
\end{aligned}
\end{equation}
for some suitable, smooth cutoff function $\Phi$ defined more precisely in \eqref{q_bounds}. We then perform a fixed point argument in the critical spaces $Z_1$ and $Z_2$. Here $Z_1$ denotes the space of functions $F:[0,\infty)\times\R\to\C$ induced by the norm
\begin{equation}\label{normZ1}
\begin{split}
\|F\|_{Z_1}:=\sup_{t\in[0,\infty)}\big\{\|F(t)\|_{L^\infty}+\sup_{k\in\mathbb Z}(2^{k}t)^{1/10}\|P_{k}F(t)\|_{L^\infty}+\|x\partial_xF(t)\|_{L^\infty}+\sup_{k\in\mathbb Z}(2^{k}t)^{1/10}\|P_{k}(x\partial_{x}F(t))\|_{L^\infty}\big\},
\end{split}
\end{equation}
and $Z_2$ denotes the space of functions $F:[0,\infty)\times \R\to\C$ induced by the norm
\begin{equation}\label{normZ2}
\|F\|_{Z_2}:=\sup_{t\in[0,\infty)}\big\{\sup_{k\in\Z}2^{k/2}\max\{2^kt,1/(2^kt)\}^{1/10}\|P_k F(t)\|_{L^2} \big\}.
\end{equation}
We define the $Z$ space as the sum space $Z=Z_1+Z_2$. 

Here and above, $P_k$, $k\in\mathbb{Z}$, denote  standard Littlewood-Paley operators on $\R$, given by $P_k f(x)=\mathcal{F}^{-1}(\varphi_k(\xi)\widehat{f}(\xi))(x)$, where $\mathcal{F}^{-1}$ denotes inverse Fourier transform, $\widehat{f}$ denotes the Fourier transform of the function $f$, and $\varphi_k(\cdot)=\varphi_0(2^{-k}\cdot)$ with $\varphi_0$ an even, smooth function compactly supported on an annulus, such that $\{\varphi_k\}_k$ is a dyadic partition of unity.

Furthermore, it will be useful to define $N$ to be the space of functions $F:[0,\infty)\times\mathbb{R}\to\mathbb{C}$ induced by the norm
\begin{equation}\label{normN}
\|F\|_N:=\sup_{t\in[0,\infty)}\big\{\sup_{k\in\Z}2^{k/2}(2^kt)^{-1/10}\|P_kF\|_{L^2}\big\}
\end{equation}
which measures the nonlinearity. We remark that the $N$-norm admits a loss for $2^kt<1$. However, this will not be an issue after running Duhamel's formula for the fixed point argument.

\subsection{Statement of the main result}
The main result of this paper is the following theorem.

\begin{theorem}\label{main_thm}
	Let $\mathcal{C}=\{a_1,\ldots,a_M\}$ be a finite set of points indexed by $j\in\mathcal{J}=\{1,\ldots,M\}$ and let $\ep_0>0$ depend only on  $\mathcal{C}$. Let $h_{0}=\sum_{j\in\mathcal{J}}h_{j,0}$ be such that
	\begin{equation*}\label{ini_data}
	\begin{split}
	&\sum_{j\in\mathcal{J}}\big(\|h_{j,0}\|_{L^\infty}+\|(x-a_j)\partial_{x}h_{j,0}\|_{L^\infty})=\ep\leq \varepsilon_0\ll 1,\quad\text{and}\quad
	\sum_{j\in\mathcal{J}}\|h_{j,0}(x+a_j)+h_{j,0}(a_j-x)\|_{L^p}\le 1,
	\end{split}	
	\end{equation*}
    for some $p\in [1,\infty)$. Then there is a unique solution $h:[0,\infty)\times\R\to\R$ to the initial value problem of the form
	\begin{equation*}
	h(t,x)=\sum_{j\in\mathcal{J}}g_j(t,x-a_j+q(t,x))
	\end{equation*}
	for a suitable change of variables $q(t,x)$ which satisfies the bounds \eqref{q_bounds1} and functions $g_j\in Z$ with $\|g_j\|_Z\lesssim\ep$.
\end{theorem}
\begin{figure}
	\centering
	\includegraphics[scale=1]{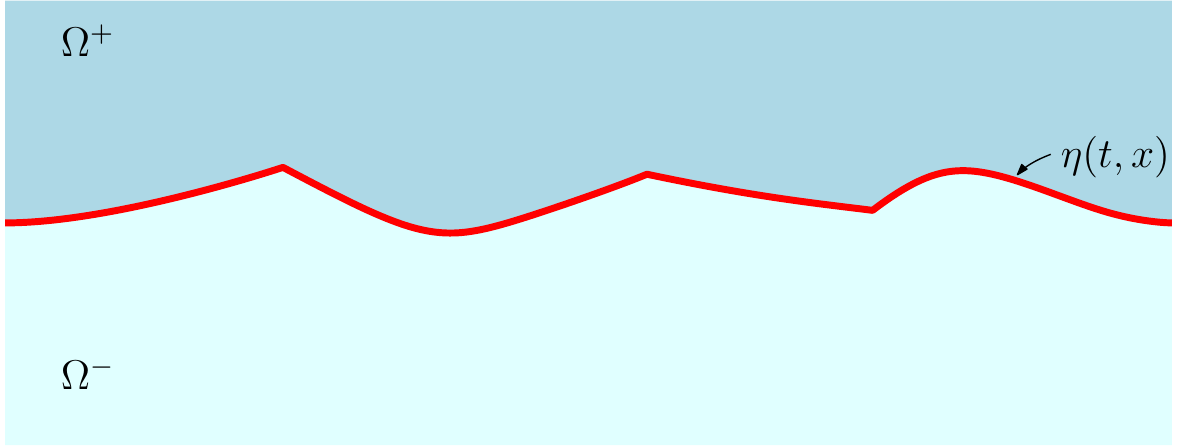}
	\caption{The initial data: the interface $\eta$ consists of a superposition of corners
		\label{fig:initial}
	}
\end{figure}

\begin{definition}
We define a \textit{symmetric corner} at a point $a$ on the interface $\eta$, to be a solution for which the derivative  $h(t,x)=\partial_x\eta(t,x)$ is odd up to lower order terms, in a neighborhood of $a$. More specifically, this means that the quantity $[{h(t,x+a)+h(t,a-x)}]/{x}$ is integrable around $x=0$. 
\end{definition}

\begin{remark} 	Informally stated, we assume the initial data of the interface $\eta(t,x)$ to be a finite superposition of not necessarily symmetric corners, see Figure~\ref{fig:initial}. We then show that the equation admits a solution for which the corners desingularize -- as if the solution were a superposition of free evolutions -- and move, before reaching time $T=1$. After the desingularization takes place, we fall back into a traditional small-data critical theory in Sobolev spaces setting and our solutions exist globally in time. We further point out that the function $q(t,x)$ is only relevant until time $T=1$, hence the extra assumption $\partial_tq(t,x)=0$ for $t\geq1$ in \eqref{q_bounds1}. Furthermore, the $L^p$ condition on the initial data ensures that any potential corner at infinity will be symmetric.
\end{remark}
\begin{remark}
    This theorem provides a precise description of the solution we construct. We are able to identify the location of each corner through the definition of the $Z_1$ norm, and then capture the nonlinear movement of these corners (of order $t\log(t)$), as they desingularize through the normalization in $q$. The parabolic nature of the problem ensures that these corners do indeed smooth out (see \eqref{smoothing}). We point out that after time $t>0$, the solution is given by a shifted free evolution plus a remainder that lands in subcritical spaces (see the definition of the $Z_2$ norm in \eqref{normZ2}), and hence become instantaneously $C^\infty$.
\end{remark}
\begin{remark}\label{discre_self}
    In comparison with previous results, Theorem \eqref{main_thm} allows interfaces with linear growth at infinity. Therefore, it includes as a particular case the self-similar solutions constructed in \cite{GarciaJuarez-GomezSerrano-Nguyen-Pausader:self-similar-muskat}. Moreover, Theorem \eqref{main_thm} also allows to construct discretely self-similar solutions, $h(t,x)=f(\log t,x/t)$, with $f(s,y)$ bounded, periodic in $s$ and H\"older regular in $y$ (discretely self-similar solutions originate from initial data such as $h_0=\varepsilon\,\text{sign}(x)\sin{(\log{(|x|)})}$).
    In this regard, we notice that while our main interest is the description of evolving corners, the result also covers other types of discontinuities.
\end{remark}
\begin{remark}
    In the case of a single, symmetric corner, there is no need for the logarithm correction term $q(t,x)$. Indeed, in this setting, our choice of function spaces for $Z_1$ and $Z_2$ are sufficient for closing the fixed point argument. As will become apparent in the proof, the logarithmic loss arises either when a corner is not symmetric, or when two neighboring corners interact in a given setting.
\end{remark}

\begin{remark}
The natural analogue of Theorem~\ref{main_thm} also holds for periodic initial data $h_0$. The proof follows in a very similar fashion, with Fourier transforms replaced by Fourier series (since frequencies are now integers instead of real numbers). Furthermore, we have the additional conservation law $\int_{\mathbb T}h(t,x)\,dx=0$, and the additional condition on the initial data $\sum_{j\in\mathcal{J}}\|h_{j,0}(x+a_j)+h_{j,0}(a_j-x)\|_{L^p}\le 1$ is not needed in this setting. 
\end{remark}

\begin{remark}\label{remark16}
	Although we prove well-posedness in the sum space $Z=Z_1+Z_2$, the $Z_1$ space only captures the free evolution and the $Z_2$ space, the nonlinearity after applying Duhamel's formula. More specifically, we will show that the solution $h(t,x)$ can be written as the free evolution of each corner translated by $q(t,x)$ plus a perturbation term $h^2$ satisfying better bounds,
	\begin{equation}\label{h_fixed_point}
	\begin{aligned}
	h(t,x)=\sum_{j\in\mathcal{J}}h_{j}^1(t,x)+h^2(t,x),\quad \text{where}  \quad h_{j}^1(t,x)=e^{-t|\nabla|}h_{j,0}(t,x+q(t,x)).
	\end{aligned}
	\end{equation}
    An important observation is that the correction function $q(t,x)$ only appears in the part of the solution which lies in the $Z_1$ space, and hence only depends on the free evolution. In particular, it can be calculated explicitly as a bilinear expression of the initial data. This significantly disentangles the problem when it comes to carrying out the fixed point iterations with Duhamel's formula.
\end{remark}

\subsection{Historical considerations}
The Muskat problem \cite{Muskat:porous-media} was derived in the 1930's by Morris Muskat as a model for oil extraction. It has since attracted a lot of attention both analytically and numerically, and has proven to have many interesting behaviors. These include the formation of singularities, such as the ones in the works by Castro {\em{et al.}} \cite{Castro-Cordoba-Fefferman-Gancedo-LopezFernandez:rayleigh-taylor-breakdown, Castro-Cordoba-Fefferman-Gancedo:breakdown-muskat} in the form of overhanging interfaces leading to loss of regularity, switch of stability as shown by C\'ordoba, G\'omez-Serrano and Zlato\v{s} \cite{Cordoba-GomezSerrano-Zlatos:stability-shifting-muskat, Cordoba-GomezSerrano-Zlatos:stability-shifting-muskat-II}, with interfaces that turn but then go back to their equilibrium, and splash singularities in the one-phase setting \cite{Castro-Cordoba-Fefferman-Gancedo:splash-singularities-muskat}. See also the work of Shi \cite{Shi:regularity-muskat} on analyticity of solutions that have turned over. We refer to the excellent surveys by Gancedo and Granero-Belinch\'on-Lazar \cite{Gancedo:survey-muskat}, \cite{GraneroBelinchon-Lazar:growth-muskat} for extended background on the problem.

Concerning well-posedness, there has been a flurry of work done for the sub-critical regime in the last two and half decades. The local in time well-posedness of the Cauchy problem on sub-critical Sobolev spaces, as well as the global existence for small data are now well understood. The first  results in high-regularity Sobolev spaces date back to Yi \cite{Yi:local-muskat, Yi:global-muskat}, Ambrose \cite{Ambrose:well-posedness-hele-shaw} and Caflisch, Howison, and Siegel \cite{Siegel-Caflisch-Howison:global-existence-muskat}, who additionally showed that some unstable settings are ill-posed.
D. C\'ordoba and Gancedo \cite{Cordoba-Gancedo:contour-dynamics-3d-porous-medium} proved well-posedness in the infinite depth setting without viscosity jump in the space $H^{d+2}(\mathbb {R}^d)$, for $d=1,2$, and later, with A. C\'ordoba, extended this result to allow for viscosity jump and non-graphical interface \cite{Cordoba-Cordoba-Gancedo:interface-heleshaw-muskat}, \cite{Cordoba-Cordoba-Gancedo:muskat-3d}. The work by Cheng, Granero-Belinchón and Shkoller \cite{Cheng-GraneroBelinchon-Shkoller:well-posedness-h2-muskat} lowered the regularity needed in two dimensions to $H^2$, without relying on the contour equation and hence permitting more general domains. Later works focused on lowering the regularity up to barely subcritical spaces.
In the case of constant viscosity, Constantin, Gancedo, Shvydkoy and Vicol \cite{Constantin-Gancedo-Shvydkoy-Vicol:global-regularity-muskat-finite-slope} constructed solutions with initial data in $W^{2,p}$ for $p\in(1,\infty]$, Matioc \cite{Matioc:viscous-displacement-muskat,Matioc:local-existence-muskat-hs} with initial data in $H^2$ and $H^{3/2+\ep}$ respectively and Abels and Matioc \cite{Abels-Matioc:muskat-lp-subcritical} for the $L^p$-Sobolev subcritical range. Alazard and Lazar \cite{Alazard-Lazar:paralinearization-muskat} were able to allow non $L^2$-data. Finally, H. Q. Nguyen and Pausader \cite{Nguyen-Pausader:paradifferential-muskat} proved that the full $d$-dimensional Muskat problem (with or without bottom and with or without viscosity jump) is well-posed in $H^s(\mathbb{R}^d)$ for all $s>d/2+1$.

We now turn to the study of the Muskat problem in critical spaces. The first results are small data solutions in the Wiener algebra $\mathcal{L}_{1,1}$, space which consists of taking one derivative of the function and evaluating the $L^1$-norm in Fourier space. Constantin, C\'ordoba, Gancedo and Strain \cite{Constantin-Cordoba-Gancedo-Strain:global-existence-muskat} proved the existence of small-data solutions, without viscosity jump, and then, along with Piazza in \cite{Constantin-Cordoba-Gancedo-RodriguezPiazza-Strain:muskat-global-2d-3d}, extended this result to the 3D setting. This result was improved by Gancedo, Garc\'ia-Ju\'arez, Patel and Strain \cite{Gancedo-GarciaJuarez-Patel-Strain:muskat-viscosity-global} who constructed small-data strong solutions allowing for viscosity jump in both 2D and 3D. 
Further small data critical results include H. Q. Nguyen \cite{Nguyen:global-solutions-muskat-besov} who proved a well-posedness result in the Besov space $\dot{B}_{\infty,1}^1$, a space embedded in the critical space $\dot{W}^{1,\infty}$, and Cameron,  who first studied well-posedness for interfaces in $\dot{W}^{1,\infty}\cap L^2$ in 2D \cite{Cameron:global-wellposedness-muskat-slope-less-1}, and then in 3D with just sub-linear growth at infinity \cite{Cameron:global-muskat-3D}. Cameron's works have the additional particularity of allowing for ``medium"-sized initial data, in the sense that it is bounded by $1$ as opposed to by some small $\ep$-value. For large data, global solutions are not generally expected given that singularities may arise. Nonetheless, Deng, Lei and Lin \cite{Deng-Lei-Lin:2d-muskat-monotone-data} constructed global weak solutions assuming that the initial interface is monotonic.

Recently, there has also been significant work done in the critical Sobolev space $\dot{H}^{3/2}$. The first result dates back to C\'ordoba and Lazar \cite{Cordoba-Lazar:global-wellposedness-muskat-H32}, who although working in the subcritical space $\dot{H}^{5/2}$, considered smallness and derived the a priori estimates in the critical space. This result was then extended to the 3D setting by Gancedo and Lazar \cite{Gancedo-Lazar:global-wellposedness-muskat-3d}. The first fully critical result is due to Alazard and Q. H. Nguyen, constructing 2D solutions with initial data in 
$H^{3/2} \cap \dot{W}^{1,\infty}$ \cite{Alazard-QHNguyen:cauchy-muskat-ii-critical}, together with the log-subcritical work
\cite{Alazard-QHNguyen:muskat-non-lipschitz} where unbounded slopes are allowed.
In these works, the solutions are either large data, local in time, or small data global in time. They later were able to obtain well-posedness in $H^{3/2}$ \cite{Alazard-QHNguyen:muskat-endpoint}, dropping the $L^\infty$ assumption altogether. They then \cite{Alazard-QHNguyen:muskat-3d} further extended this to the 3D setting in $\dot{H}^{2}\cap W^{1,\infty}$.  

Critical well-posedness in the Sobolev norm does not allow for corners, which leads us to the work of Chen, Q. H. Nguyen and Xu \cite{Chen-Nguyen-Xu:muskat-c1} who studied initial data for the interface in $\dot{C}^1$, again, either large data, local in time, or small data, global in time. In this work, although the interface is smooth, the uniqueness result does allow for discontinuities. Very recently \cite{GarciaJuarez-GomezSerrano-Nguyen-Pausader:self-similar-muskat}, three of the authors of this paper proved the existence of small data self-similar solutions, starting from exact corners, which desingularize instantaneously. 

Our result differs from the previous ones in the sense that we carry out our analysis by using a fixed point argument, rather than energy estimates or maximum principles. We start our fixed point method from an ansatz differing from the free evolution (cf. Remark \ref{h_fixed_point}). This argument enables us to describe the solution in a much more precise way as a convergent series in well-adapted function spaces in which a suitable shifted free evolution is the first term. As a result, we can attain a better understanding of which space the corners lie in, what they look like (not necessarily symmetric) as well as what their behavior is (movement and immediate desingularization). Moreover, as opposed to \cite{Cameron:global-wellposedness-muskat-slope-less-1, Cameron:global-muskat-3D} and all previous works, the space $Z_1$ includes functions with linear growth at infinity: these encompass not only self-similar solutions \cite{GarciaJuarez-GomezSerrano-Nguyen-Pausader:self-similar-muskat} but also discretely self-similar solutions (see Remark \ref{discre_self}).

Finally, it is interesting to notice that we show instant desingularization of the corners, hence no `waiting time' phenomena is possible for small corners in two-phase Muskat. We refer to \cite{Choi-Kim:waiting-time-hele-shaw, Choi-Jerison-Kim:one-phase-hele-shaw-lipschitz, Bazaliy-Vasylyeva:two-phase-hele-shaw} for related results in the one-phase setting and in Hele-Shaw flows.

We finish this subsection highlighting the following recent papers on the one-phase problem: \cite{Dong-Gancedo-Nguyen:one-phase-muskat, Agrawal-Patel-Wu:corners-one-phase-muskat, Nguyen-Tice:one-phase-muskat-traveling}.

\subsection{Outline of the paper}
The main difficulty of this problem is that unless we restrict ourselves to the setting of a single, symmetric corner, we end up with a logarithmic loss in our a priori estimates. As a result, the first crucial step we must do is a renormalization of the problem, and subsequently, in \eqref{ga_eq}, a reformulation of the problem in terms of the renormalized function $g(t,x-x_j+q(t,x))$. This necessitates the introduction of the function $q(t,x)$ (see equations \eqref{q_bounds}-\eqref{qtilde_bounds}). These ensure on the one hand, that we do in fact have a change of variables, and on the other, that the function spaces remain invariant by the renormalization, Lemma~\ref{lem:g_h_Z}. The definition of $q(t,x)$ will depend only on the free evolution of the initial data, hence avoiding a circular argument (see Remark \ref{remark16}).

After performing a Taylor expansion on the denominator of the integral term in \eqref{nonlin2}, which reveals the parabolicity of the problem \eqref{smoothing}, we reformulate it as a pseudoproduct in order to study the interactions between the various frequencies. In an effort to keep the arguments transparent and easier to follow, we first carry out the entire analysis for the trilinear setting (in which we only consider the first term in the Taylor expansion). We provide a generalization of all lemmas to the full nonlinearity in Section~\ref{sect:full}. 

Our main goal will now be to find suitable a priori bounds on the localized nonlinearity, ensuring that the integral term in Duhamel's formula remains a mere perturbation in the $Z_2$ space. To this end, we carefully study the kernels of the pseudoproduct in Lemmas~\ref{lemma:K}-\ref{lemma:K2}, which in turn enable us to perform localized $L^2$ estimates, Lemma~\ref{lemma:TP}. These estimates will prove to yield desired bounds on all high-high frequency interactions, Lemma~\ref{lemmaG2} as well as for all high-low interactions which contain at least one function in the $Z_2$ space, Lemma~\ref{lemmaG1Z2}. We remark that from the definition of the $Z_2$ space in \eqref{normZ2}, this space allows us to sum over each frequency at a time. 

The remaining difficulty is the setting for which all functions in the pseudoproduct are in the $Z_1$ space. The difficulty is two-fold. On the one hand, we need to put the highest frequency in $L^2$ to prevent a loss of derivative, an inflexible restriction which prevents us from getting any more gain from this function. On the other hand, the $Z_1$ norm as defined in \eqref{normZ1}, does not contain any $P_k$. By considering the elements in this space localized at frequencies, we lose information which in turn is responsible for logarithmic losses. To rectify this, we carry out a more careful study of the functions in the $Z_1$ space, Lemma\ref{lemma:error}. In particular, we find that they can be split into core and error terms. 

Upon further pursuing the analysis of the pseudoproduct when all functions lie in $Z_1$, we seek to identify the exact term for which the logarithmic loss occurs. By gradually peeling away all bounded terms Lemmas~\ref{lem:G1less}-\ref{lemma:L2}, we are eventually left with a velocity-like bilinear expression \eqref{lpo1}. Indeed, it then becomes apparent that subtracting this velocity term from the localized pseudoproduct yields the desired bounds on the nonlinearity, Lemma~\ref{lem:G1greq4}. 

From the formulation of the problem in terms of the renormalized function \eqref{ga_eq}, it is apparent that this velocity term must be equal to the time derivative of the correction function, $\partial_tq(t,x)$. Checking that the bounds of $q(t,x)$ assumed at the beginning are indeed satisfied, Lemma~\ref{lemma:L-inf}, hence completes the renormalization bootstrap procedure. In Section~\ref{sect:full} we show how all the results in the trilinear case generalize to the full nonlinearity, and finally, we conclude with Section~\ref{CloseProof}  in which we explicitly construct the correction term $q$ and carry out a fixed point argument to construct the solution. 
	
\section{Preliminaries}
\subsection{Reformulation of the problem}
In addition to a critical well-posedness result, we want to track the location of the corners and understand their behavior. As mentioned previously, unless we restrict ourselves to the setting in which the interface consists of a single, symmetric corner, our estimates produce a logarithmic term. This suggests the introduction of a renormalization procedure. We define the new function $g$ as in \eqref{h_g}, such that $h(t,x)=g(t,x+q(t,x))$ and $q(t,x)$ satisfies the bounds
\begin{equation}\label{q_bounds}
\begin{aligned}
|q(t,x)|\lesssim \ep\, t\ln(2/t)\Phi_{\leq 0}(x),\qquad
\sup_{x\in\R}|\partial_x q(t,x)|\lesssim \ep,\qquad \partial_tq(t,x)=0\quad\text{for }t\geq 1.
\end{aligned}
\end{equation}
Here, for any $l\in\Z$, $\Phi_{\leq l}:\R\to[0,1]$ is a smooth cutoff function supported in $\cup_{j\in\mathcal{J}}(a_j-2^{l+2},a_j+2^{l+2})$ and equal to $1$ in the set $\cup_{j\in\mathcal{J}}(a_j-2^{l-2},a_j+2^{l-2})$ and satisfying natural bounds $\|\partial^n_x\Phi_{\leq l}(x)\|_{L^\infty}\lesssim 2^{-nl}$, $n\in\{1,2,3\}$

In particular, the function $Q(t,x):=x+q(t,x)$ is strictly increasing and bijective on $\R$. We define its inverse $\widetilde{Q}(t,y)=y+\widetilde{q}(t,y)$, 
\begin{equation}\label{changeV}
Q(t,\widetilde{Q}(t,x))=x,\qquad \widetilde{Q}(t,Q(t,y))=y.
\end{equation}
We assume that the function $\widetilde{q}$ satisfies the bounds
\begin{equation}\label{qtilde_bounds}
|\partial_y^n\widetilde{q}(t,y)|\lesssim\varepsilon\sum_{j\in\mathcal{J}}\frac{1}{|y-a_j|^{n-1}},\qquad n\in\{1,2,3\},\,t\in[0,\infty),\,y\in\R.
\end{equation} 
\begin{remark}
    The bounds above can be viewed as a norm for a function space in which $q(t,x)$ and $\tilde{q}(t,x)$ must lie. We remark that this space is also critical, and indeed, if it were not, $q$ would be perturbative and wouldn't be necessary to begin with.
\end{remark}

Performing a Taylor expansion in the denominator, we get for $h=h(t,x):[0,\infty)\times\mathbb{R}\to\mathbb{R}$, the following one-dimensional evolution equation
\begin{equation}\label{nonlin1}
(\partial_t+|\nabla|)h=\mathcal{N},\qquad h(0,\cdot)=h_0,
\end{equation}
with
\begin{equation}\label{nonlin2}
	\mathcal{N}=\sum_{n\geq1}\mathcal{N}_{n}[h,\ldots,h]\quad\text{ where }\quad\mathcal{N}_{n}[h_1,h,\ldots,h](x):=\frac{(-1)^n}{\pi}\frac{d}{dx}\int_{\R}\partial_x h_1^\ast(x,\alpha)\cdot\big(h^*(x,\alpha)\big)^{2n}\,d\alpha.
\end{equation}
From \eqref{nonlin1}, we see that the linear solution takes the form
\begin{equation}\label{smoothing}
    h^{(0)}(t,x)=e^{-t|\nabla|}h_0(x),
\end{equation}
thus revealing the parabolic nature of the equation. This is the key property which enables us to treat the problem as a semi-linear one (in the sense that we can perform a fixed point argument to construct the solution), despite the fact that the nonlinear term contains a derivative.

We denote
\begin{equation}\label{Nha2}
\mathcal{N}_{h_j}(x):=\sum_{n\geq1}\mathcal{N}_{n}[h_j,h,\ldots,h](x).
\end{equation}
We will construct $h(t,x)$ (correspondingly $g(t,x)$) as the sum 
\begin{equation*}\label{g_ga}
h(t,x)=\sum_{j\in\mathcal{J}}h_j(t,x),
\end{equation*}
where we are denoting 
\begin{equation}\label{ha_ga}
h_j(t,x):=g_j(t,x-a_j+q(t,x)),
\end{equation}
each $g_j$ will be defined to satisfy the following equation (see Lemma \ref{lem:ga_eq}):
\begin{equation}\label{ga_eq}
\begin{aligned}
\partial_t g_j(t,x)\!+\!|\nabla|g_j(t,x)&=F_j(t,x+a_j),\\
g_j(0,x)&=g_{j,0}(x)=h_{j,0}(x+a_j),
\end{aligned}
\end{equation}
with
\begin{equation*}
	\begin{aligned}
	F_j(t,x+a_j)&=\partial_{x}h_j(t,\tilde{Q}(t,x+a_j))\partial_t\tilde{q}(t,x+\!a_j)+\mathcal{N}_{h_j}(t,\tilde{Q}(t,x+a_j))\\
	&\quad-|\nabla|g_j(t,x)\partial_{x}q(t,\tilde{Q}(t,x+a_j))-E_j(t,\tilde{Q}(t,x+a_j)),
	\end{aligned}
\end{equation*}
and $q,\widetilde{q}$ are defined as above. 
We remark that 
\begin{equation*}\label{q_sum}
    q(t,x)=\sum_{n\geq 1}q_n(t,x),
\end{equation*}
where each $q_n$ serves as a logarithmic correction term for its respective associated nonlinearity $\mathcal{N}_n$, as defined in \eqref{nonlin2}.

We need to solve the system of equations for $g_j$, with $j\in\mathcal{J}$. The aim is to construct the functions $g_j$ solving \eqref{ga_eq} using a fixed-point argument in the space $Z_2$. We will decompose each $g_j$, $$g_j=g_{j}^1+g_{j}^2=e^{-t|\nabla|}g_{0,j}+\int_0^te^{-(t-s)|\nabla|}F_j(s)ds,$$
so that $g_{j}^1\in Z_1$ and $g_{j}^2\in Z_2$. 

From \eqref{ga_eq}, we see that the only way to cancel any logarithmic loss which may arise in the nonlinearity is by means of the correction term $\partial_t\tilde{q}$. In order to identify $\tilde{q}$, we must first strip off all the terms in the nonlinear part which can be bounded until we are only left with the logarithmic singularity. We then conclude the bootstrap argument by verifying that this leftover term satisfies better bounds than the ones in \eqref{q_bounds}-\eqref{qtilde_bounds}.

\subsection{Function Spaces Lemmas}\label{sec:spaces}
	
In this section, we gather useful lemmas concerning the functions spaces defined in \eqref{normZ1}-\eqref{normN}. Most significantly, we verify that the change of variables does not alter the norms.  
	\begin{lemma}\label{lemma:Z}
		For $F_1\in Z_1$, $F_2\in Z_2$, $k\in\Z$ and $t\in[0,\infty)$ we have
		\begin{align}\label{Z1Z2est1}
		\|P_k F_1(t)\|_{L^2}&\lesssim 2^{-k/2}\min\{1,(2^kt)^{-1/10}\}\|F_1\|_{Z_1},\\\label{Z1Z2est2}
		\|P_k F_2(t)\|_{L^2}&\lesssim2^{-k/2}\min\{(2^kt)^{1/10},(2^kt)^{-1/10}\}\|F_2\|_{Z_2}.
		\end{align}
		In particular, for any $F\in Z$, we have 
		\begin{align}\label{ZtotEst1}
		(1+2^k t)^{1/10}\|P_kF(t)\|_{L^\infty}&\lesssim\|F\|_Z,\\\label{ZtotEst2}
		(1+2^k t)^{1/10}\|P_kF(t)\|_{L^2}&\lesssim2^{-k/2}\|F\|_Z.
		\end{align}
	\end{lemma}
 \begin{remark}
     From \eqref{ZtotEst2} we see that by taking the $L^2$ norm of frequency localized functions, we have a gain in derivatives. As a result, since we have a loss of derivatives in the nonlinear term, we will always take the function with the highest frequency in $L^2$. 
 \end{remark}
	\begin{proof}
		From the definition of $Z_2$, for $F\in Z_2$ we clearly have $(1+2^k t)^{1/10}\|P_kF(t)\|_{L^2}\lesssim 2^{-k/2}\|F\|_{Z_2}$. That $P_kF$ is also in $L^\infty$ follows from the Sobolev embedding
		\begin{equation*}
		\|P_kF(t)\|_{L^\infty}\lesssim2^{k/2}\|P_kF(t)\|_{L^2}.
		\end{equation*}
		For $F\in Z_1$, we clearly have $(1+2^k t)^{1/10}\|P_kF(t)\|_{L^{\infty}}\lesssim\|F\|_{Z_1}$. It remains to check that $(1+2^k t)^{1/10}P_kF(t)$ is also in $L^2$. We do this in the Fourier space. We have
		\begin{align*}
		\widehat{P_k F}(\xi)&=\varphi_k(\xi)\int_{\R}F(x)e^{-ix\xi}dx=\varphi_k(\xi)\int_{|x|\leq 2^{-k}}F(x)e^{-ix\xi}dx+\varphi_k(\xi)\int_{|x|\geq 2^{-k}}F(x)e^{-ix\xi}dx\\
		&=:A(\xi)+B(\xi).
		\end{align*}
		We can easily check that
		\begin{equation*}
		\|A(\xi)\|_{L^2}\lesssim2^{-k/2}\|F(x)\|_{L^\infty}.
		\end{equation*}
		Moreover, using integration by parts, we get
		\begin{align*}
		\int_{|x|\geq 2^{-k}}F(x)e^{-ix\xi}dx=\frac{F(-2^{-k})}{-i\xi}e^{i2^{-k}\xi}-\frac{F(2^{-k})}{-i\xi}e^{-i2^{-k}\xi}-\frac{i}{\xi}\int_{|x|\geq 2^{-k}}\partial_{x}F(x)e^{-ix\xi}dx
		\end{align*}
		which yields
		\begin{equation*}
		\|B(\xi)\|_{L^2}\lesssim2^{-k/2}\|F(x)\|_{L^\infty}+2^{-k}\Big\|\frac{1_{|x|\geq 2^{-k}}}{x}x\partial_{x}F(x)\Big\|_{L^2}\lesssim2^{-k/2}\|F(x)\|_{L^\infty}+2^{-k/2}\|x\partial_{x}F(x)\|_{L^\infty}.
		\end{equation*}
		Combining with Plancherel's theorem, we now get
		\begin{align*}
		\|(1+2^k t)^{1/10}P_kF(x)\|_{L^2}\lesssim2^{-k/2}\|F(x)\|_{Z_1},
		\end{align*}	
		which concludes the proof.
	\end{proof}
	
	Since the nonlinearity is expressed in terms of the function $h$ but we have $Z$-norm control on the solutions $g_j$, we begin with the following two  lemmas.
	For the remainder of this section, we will not explicitly include the dependence on $t$ for simplicity of notation.
	
\begin{lemma}\label{lem:g_h}
	Assume $g\in L^2$, $q$ satisfies \eqref{q_bounds}--\eqref{qtilde_bounds}, $t\in[0,\infty)$, $a\in\R$ and $h(x):=g(x-a+q(t,x))$ (compare with \eqref{ha_ga}). Then, for any $k\in\Z$ we have
	\begin{equation}\label{cook1}
	\|P_k h\|_{L^2}\lesssim\sum_{k'\leq k}\|P_{k'}g\|_{L^2}\,2^{-(k-k')}+\sum_{k'\geq k}\|P_{k'}g\|_{L^2}\,2^{-\frac{1}{2}(k'-k)}.
	\end{equation}
	Moreover, for $g(y)=h(y+a+\tilde{q}(t,y+a))$, for any $k\in\mathbb{Z}$ we have
		\begin{equation}\label{cook1-1}
	\|P_k g\|_{L^2}\lesssim\sum_{k'\leq k}\|P_{k'}h\|_{L^2}\,2^{-(k-k')}+\sum_{k'\geq k}\|P_{k'}h\|_{L^2}\,2^{-\frac{1}{2}(k'-k)}.
	\end{equation}
\end{lemma}

\begin{proof}
	Given $k\in\Z$ we write
	\begin{equation}\label{cook2}
	\begin{split}
	&h(x)=\sum_{k'\in\Z}(P_{k'}g)(x-a+q(t,x))=h_1(x)+h_2(x),\\
	&h_1(x):=\sum_{k'\leq k+4}(P_{k'}g)(x-a+q(t,x)),\\
	&h_2(x):=\sum_{k'\geq k+5}(P_{k'}g)(x-a+q(t,x)).
	\end{split}
	\end{equation}
	
	To estimate $\|P_kh_1\|_{L^2}$ we take a derivative in $x$ and write
	\begin{equation*}
	h'_1(x)=\sum_{k'\leq k+4}(1+q'(t,x))(P_{k'}g)'(x-a+q(t,x))
	\end{equation*}
	Therefore, since $|q'(t,x)|\lesssim\ep$, it follows that
	\begin{equation*}
	\|h'_1\|_{L^2}\lesssim \sum_{k'\leq k+4}\|(P_{k'}g)'\|_{L^2}\lesssim \sum_{k'\leq k+4}2^{k'}\|P_{k'}g\|_{L^2}.
	\end{equation*}
	Therefore
	\begin{equation}\label{cook3}
	\|P_kh_1\|_{L^2}\lesssim \sum_{k'\leq k+4}2^{k'-k}\|P_{k'}g\|_{L^2}.
	\end{equation}
	
	To estimate $\|P_kh_2\|_{L^2}$ we let $G_{k'}(x):=(P_{k'}g)(x-a+q(t,x))$ and take the Fourier transform to write
	\begin{equation*}\label{cook3.5}
	\begin{split}
	\widehat{G_{k'}}(\xi)=\int_{\R}G_{k'}(x)e^{-ix\xi}dx&=\int_{\R}(P_{k'}g)(x-a+q(t,x))e^{-ix\xi}dx\\
	&=\frac{1}{2\pi}\int_{\R}\int_{\R}\widehat{P_{k'}g}(\eta)e^{i\eta(x-a+q(t,x))}e^{-ix\xi}dxd\eta\\
	&=\frac{1}{2\pi}\int_{\R}\widehat{P_{k'}g}(\eta)L(\xi,\eta)d\eta,
	\end{split}
	\end{equation*}
	where here
	\begin{equation*}
	L(\xi,\eta):=\int_{\R}e^{-ix\xi}e^{i\eta (x-a+q(t,x))}dx.
	\end{equation*}
	We make the change of variables $x=\widetilde{Q}(t,y)$ and decompose the kernel $L$ as $L=L_1+L_2$, where
	\begin{equation*}\label{cook4}
	\begin{split}
	&L_1(\xi,\eta):=e^{-i\eta a}\int_{\R}e^{-i\xi\widetilde{Q}(t,y)}e^{i\eta y}\widetilde{Q}'(t,y)\Phi_{\leq -k'}(y)\,dy,\\
	&L_2(\xi,\eta):=e^{-i\eta a}\int_{\R}e^{-i\xi\widetilde{Q}(t,y)}e^{i\eta y}\widetilde{Q}'(t,y)(1-\Phi_{\leq -k'})(y)\,dy,
	\end{split}
	\end{equation*}
	and the functions $\Phi_{\leq l}$ are defined as before. 
	
	We would like to prove that if $|\xi|\in[2^{k-1},2^{k+1}]$, $|\eta|\in[2^{k'-1},2^{k'+1}]$, and $k'\geq k+5$ then
	\begin{equation}\label{cook4.5}
	|L_n(\xi,\eta)|\lesssim 2^{-k'},\qquad n\in\{1,2\}.
	\end{equation}
	The bounds follow easily for the function $L_1$, due to the support restriction on $y$. To bound $L_2$ we integrate by parts in $y$,
	\begin{equation*}
	\begin{split}
	L_2(\xi,\eta)&=ie^{-i\eta a}\int_{\R}e^{i(\eta y-\xi\widetilde{Q}(t,y))}\frac{d}{dy}\Big\{\frac{\widetilde{Q}'(t,y)(1-\Phi_{\leq -k'}(y))}{\eta-\xi\widetilde{Q}'(t,y)}\Big\}\,dy\\
	&=ie^{-i\eta a}\int_{\R}e^{i(\eta y-\xi\widetilde{Q}(t,y))}\frac{\eta\widetilde{Q}''(t,y)}{(\eta-\xi\widetilde{Q}'(t,y))^2}(1-\Phi_{\leq -k'}(y))\,dy+O(2^{-k'}).
	\end{split}
	\end{equation*}
	This is not enough to prove the bounds \eqref{cook4.5}, because of the factors $1/|y-x_j|$ in the bounds for $|\widetilde{Q}''(t,y)|$ in \eqref{qtilde_bounds}, which lead to logarithmic losses. However, we can integrate by parts in $y$ once more to see that
	\begin{equation*}
	\begin{split}
	|L_2&(\xi,\eta)|\lesssim 2^{-k'}+\Big|\int_{\R}e^{i(\eta y-\xi\widetilde{Q}(t,y))}\frac{\eta\widetilde{Q}''(t,y)}{(\eta-\xi\widetilde{Q}'(t,y))^2}(1-\Phi_{\leq -k'}(y))\,dy\Big|\\
	&\lesssim 2^{-k'}+2^{k'}\Big|\int_{\R}e^{i(\eta y-\xi\widetilde{Q}(t,y))}\frac{d}{dy}\Big\{\frac{\widetilde{Q}''(t,y)(1-\Phi_{\leq -k'}(y))}{(\eta-\xi\widetilde{Q}'(t,y))^3}\Big\}\,dy\Big|\\
	&\lesssim 2^{-k'}+2^{k'}\int_{\R}\frac{|\widetilde{Q}'''(t,y)||1-\Phi_{\leq -k'}(y)|}{2^{3k'}}+\frac{|\widetilde{Q}''(t,y)||\Phi'_{\leq -k'}(y)|}{2^{3k'}}+\frac{|\xi||\widetilde{Q}''(t,y)|^2|1-\Phi_{\leq -k'}(y)|}{2^{4k'}}\,dy.
	\end{split}
	\end{equation*}
	Using now \eqref{qtilde_bounds} we estimate
	\begin{equation*}
	\begin{split}
	|L_2(\xi,\eta)|&\lesssim 2^{-k'}+2^{-2k'}\int_{\R}\Big(\sum_{j\in\mathcal{J}}\frac{1}{|y-x_j|^2}\Big)|1-\Phi_{\leq -k'}(y)|+\Big(\sum_{j\in\mathcal{J}}\frac{1}{|y-x_j|}\Big)|\Phi'_{\leq -k'}(y)|\,dy\\
	&\lesssim 2^{-k'}.
	\end{split}
	\end{equation*}
	This completes the proof of \eqref{cook4.5} for $n=2$. 
	
	Given \eqref{cook4.5} we can now use the Cauchy-Schwartz inequality to estimate
	\begin{align*}
	\|P_kh_2\|_{L^2}&\lesssim \sum_{k'\geq k+5}\|P_kG_{k'}\|_{L^2}\lesssim \sum_{k'\geq k+5}\|\varphi_k(\xi)\widehat{G_{k'}}(\xi)\|_{L^2_\xi}\lesssim \sum_{k'\geq k+5}2^{k/2}\|\varphi_k(\xi)\widehat{G_{k'}}(\xi)\|_{L^\infty_\xi}\\
	&\lesssim \sum_{k'\geq k+5}2^{k/2}\|\varphi_k(\xi)L(\xi,\eta)\mathbf{1}_{[2^{k'-1},2^{k'+1}]}(|\eta|)\|_{L^\infty_\xi L^2_\eta}\|P_{k'}g\|_{L^2}\lesssim \sum_{k'\geq k+5}2^{k/2}2^{-k'/2}\|P_{k'}g\|_{L^2}.
	\end{align*}
	The desired bounds \eqref{cook1} follow using also \eqref{cook2} and \eqref{cook3}. The proof for \eqref{cook1-1} follows similarly, except that no change of variables is required.
\end{proof}

As an application, we show that $h$ satisfies similar estimates as those for $g$ in Lemma \ref{lemma:Z}. Moreover, the change of variables does not significantly alter the $N$-norm.

\begin{lemma}\label{lem:g_h_Z}
	For $k\in\Z$, $t\in[0,\infty)$, $f\in Z$, and $f^\tau$ defined by $f^\tau(t,x):=f(t,x-a+q(t,x))$, we have that 
	\begin{equation*}
	\begin{aligned}
	\|P_k f^\tau(t,x)\|_{L^2}&\lesssim 2^{-k/2}\min\{1,(2^kt)^{-1/10}\}\|f\|_{Z_1},\\
	\|P_k f^\tau(t,x)\|_{L^2}&\lesssim2^{-k/2}\min\{(2^kt)^{1/10},(2^kt)^{-1/10}\}\|f\|_{Z_2}.
	\end{aligned}
	\end{equation*}
	Furthermore, for $f\in N$  and $f^\tau(t,x)$ as above, we have that
	\begin{equation}\label{N1}
	\|P_k f^\tau(t,x)\|_{L^2}\lesssim2^{k/2}(2^kt)^{-1/10}\|f\|_{N}.	
	\end{equation}
\end{lemma}

\begin{proof} First, let $f\in Z_2$ and assume without loss of generality that $\|f\|_{Z_2}=1$. Using Lemmas \ref{lem:g_h} and \ref{lemma:Z}, we have
	\begin{equation*}
	\begin{aligned}
	\|P_{k}f^\tau(t,x)\|_{L^2}&\lesssim \sum_{k'\leq k}2^{-k'/2}(1+2^{k'}t)^{-2/10}(2^{k'}t)^{1/10}2^{-(k-k')}\\
	&\quad+\sum_{k'\geq k}2^{-k'/2}(1+2^{k'}t)^{-2/10}(2^{k'}t)^{1/10}2^{-\frac12(k'-k)}\\
	&\lesssim2^{-k}2^{\frac{k}{2}}\min\{(2^kt)^{1/10},(2^kt)^{-1/10}\}+2^{\frac{k}{2}}2^{-k}\min\{(2^kt)^{1/10},(2^kt)^{-1/10}\}\\
	&\lesssim2^{-\frac{k}{2}}\min\{(2^kt)^{1/10},(2^kt)^{-1/10}\}.
	\end{aligned}
	\end{equation*}
	Now let $f\in Z_1$, $\|f\|_{Z_1}=1$, then, using again Lemmas \ref{lem:g_h} and \ref{lemma:Z},
	\begin{equation*}
	\begin{split}
	\|P_{k}f^\tau(t,x)\|_{L^2}&\lesssim \sum_{k'\leq k}2^{-k'/2}(1+2^{k'}t)^{-1/10}2^{-(k-k')}+\sum_{k'\geq k}2^{-k'/2}(1+2^{k'}t)^{-1/10}2^{-\frac12(k'-k)}\\	&\lesssim2^{-\frac{k}2}\min\{1,(2^kt)^{-1/10}\}+2^{-\frac{k}2}\min\{1,(2^kt)^{-1/10}\}\\
	&\lesssim2^{-\frac{k}2}\min\{1,(2^kt)^{-1/10}\}.
	\end{split}
	\end{equation*}
	We now let $f\in N$, $\|f\|_N=1$, then, using once more Lemma~\ref{lem:g_h},
	\begin{equation*}
	\begin{split}
		\|P_kf^\tau(t,x)\|_{L^2}&\lesssim\sum_{k'\leq k}2^{k'/2}(2^{k'}t)^{-1/10}2^{-(k-k')}
		+\sum_{k'\geq k}2^{k'/2}(2^{k'}t)^{-1/10}2^{-\frac{1}{2}(k'-k)}\\
		&\lesssim2^{-k}\sum_{k'\leq k}2^{\frac{3}{2}k'}(2^kt)^{-1/10}+2^{k/2}\sum_{k'\geq k}(2^kt)^{-\frac{1}{10}}\\&
		\lesssim2^{\frac{k}{2}}(2^kt)^{-1/10},
	\end{split}
	\end{equation*}
	thus concluding the proof.
\end{proof}

\subsection{Implementing the change of variables}
In this section, we demonstrate how to move from the physical to the renormalized formulation, and bound all linear error terms which arise in the process.   	

The first thing we consider is how the $|\nabla|$ handles the change of variables. The following lemma shows that, up to error terms with sufficiently nice bounds, we essentially have the chain rule. We point out that in order to handle $|\nabla|$, we must consider positive and negative frequencies independently, thus justifying the use of frequency projection operators to positive and negative frequencies in the proof below.	
	\begin{lemma}\label{lem:ga_eq} 
		Assume that $h_0$ is as in Theorem \ref{main_thm} and $q:[0,\infty)\times\R\to\R$ satisfies the bounds \eqref{q_bounds}. Assume that the functions $g_j\in Z$, $j\in\{1,\ldots,M\}$ solve the system \eqref{ga_eq}, with initial data $g_j(0,x):=h_{j,0}(x+a_j)$, where
		\begin{equation}\label{musk1}
		\begin{split}
		&h_j(t,x):=g_j(t,x-a_j+q(t,x)),\qquad h(t,x)=\sum_{j\in\mathcal{J}}g_j(t,x-a_j+q(t,x)),\\
		&\text{E}_j(t,x):=(|\nabla| h_j)(t,x)-(|\nabla| g_j)(x-a_j+q(t,x))(1+q'(t,x)),\\
		\end{split}
		\end{equation}
		
		Then $h$ solves the initial value problem \eqref{nonlin1}. Moreover, the remainders $E_j$ satisfy the bounds
		\begin{equation}\label{nabla4}
		\begin{aligned}
		\|P_k E_j\|_{L^2}&\lesssim\varepsilon 2^k\sum_{k'\leq k}\|P_{k'}g_j\|_{L^2}\,2^{-(k-k')}+\varepsilon2^k\sum_{k'\geq k}\|P_{k'}g_j\|_{L^2}\,2^{-\frac{1}{2}(k'-k)},\\
		\|P_k E_j\|_{L^2}&\lesssim \varepsilon2^{k/2}\min\{1,(2^kt)^{-1/10}\}\|g_j\|_{Z}.	
		\end{aligned}
		\end{equation}
	\end{lemma}

	\begin{proof}
		Consider $h_j$ as in \eqref{musk1}. Assume that 
		\begin{equation}\label{musk4}
		\partial_t h_j(t,x)+|\nabla|h_j(t,x)=\mathcal{N}_{h_j}(t,x),
		\end{equation}
		with $h_j(0,x)=h_{j,0}(x)$, which implies that $h=\sum_{j\in\mathcal{J}}h_j$ solves \eqref{nonlin1}. We compute
		\begin{align*}
		\partial_t g_j(t,x-a_j)&=\partial_t h_j(t,x+\tilde{q}(t,x))+\partial_{x}h_j(t,x+\tilde{q}(t,x))\partial_t\tilde{q}(t,x)\\
		&=\mathcal{N}_{h_j}(t,x+\tilde{q}(t,x))-(|\nabla|h_j)(t,x+\tilde{q}(t,x))+(\partial_{x}h_j)(t,x+\tilde{q}(t,x))\partial_t\tilde{q}(t,x),
		\end{align*}
		from which we get
		\begin{align*}
		\partial_t g_j(t,x-a_j)+(|\nabla|h_j)(t,x+\tilde{q}(t,x))=\mathcal{N}_{h_j}(t,x+\tilde{q}(t,x))+(\partial_{x}h_j)(t,x+\tilde{q}(t,x))\partial_t\tilde{q}(t,x),
		\end{align*}
		and consequently, using the second identity in \eqref{musk1}, we obtain \eqref{ga_eq}. Hence, reversing the process, we conclude that the identities \eqref{ga_eq} and \eqref{musk4} are equivalent.
		
		To prove the bounds \eqref{nabla4} we write
		\begin{equation*}
		\widehat{h}_j(\xi)=\int_{\R}h_j(x)e^{-ix\xi}\,dx=\int_{\R}g_j(x-a_j+q(x))e^{-ix\xi}\,dx=\frac{1}{2\pi}\int_\R\int_\R\widehat{g}_j(\eta)e^{i\eta(x-x_j+q(x))}e^{-ix\xi}\,dx d\eta.
		\end{equation*}
		Therefore
		\begin{equation}\label{nabla5}
		\widehat{|\nabla|h}_j(\xi)=\frac{1}{2\pi}\int_\R\int_\R\widehat{|\nabla|g}_j(\eta)\frac{|\xi|}{|\eta|}e^{i\eta(x-a_j+q(x))}e^{-ix\xi}\,dx d\eta.
		\end{equation}
		Let $P^+$ and $P^-$ denote the frequency projection operators to positive and negative frequencies,
		\begin{equation*}
		\widehat{P^+f}(\rho):=\widehat{f}(\rho)\mathbf{1}_+(\rho),\qquad \widehat{P^-f}(\rho):=\widehat{f}(\rho)\mathbf{1}_-(\rho).
		\end{equation*}
		It follows from \eqref{nabla5} that
		\begin{equation}\label{nabla6}
		\begin{split}
		\mathbf{1}_+(\xi)\widehat{|\nabla|h}_j(\xi)&=\mathbf{1}_+(\xi)\frac{1}{2\pi}\int_\R\int_\R\widehat{|\nabla|g}_j(\eta)\frac{|\xi|}{|\eta|}e^{i\eta(x-a_j+q(x))}e^{-ix\xi}\,dx d\eta=F_1(\xi)+G_1(\xi),\\
		F_1(\xi)&:=\mathbf{1}_+(\xi)\frac{1}{2\pi}\int_\R\int_\R\mathbf{1}_+(\eta)\widehat{|\nabla|g}_j(\eta)\frac{|\xi|}{|\eta|}e^{i\eta(x-a_j+q(x))}e^{-ix\xi}\,dx d\eta,\\
		G_1(\xi)&:=\mathbf{1}_+(\xi)\frac{1}{2\pi}\int_\R\int_\R\mathbf{1}_-(\eta)\widehat{|\nabla|g}_j(\eta)\frac{|\xi|}{|\eta|}e^{i\eta(x-a_j+q(x))}e^{-ix\xi}\,dx d\eta.
		\end{split}
		\end{equation}
		We can integrate by parts in $x$ to calculate $F_1(\xi)$,
		\begin{equation*}
		\begin{split}
		F_1(\xi)&=\mathbf{1}_+(\xi)\frac{1}{2\pi}\int_0^\infty\widehat{|\nabla|g}_j(\eta)\int_\R\frac{\xi}{\eta}e^{i\eta(x-a_j+q(x))}e^{-ix\xi}\,dx d\eta\\
		&=\mathbf{1}_+(\xi)\frac{1}{2\pi}\int_0^\infty\widehat{|\nabla|g}_j(\eta)\int_\R\frac{i}{\eta}e^{i\eta(x-a_j+q(x))}\frac{d}{dx}\{e^{-ix\xi}\}\,dx d\eta\\
		&=\mathbf{1}_+(\xi)\frac{1}{2\pi}\int_0^\infty\widehat{|\nabla|g}_j(\eta)\int_\R (1+q'(x))e^{i\eta(x-a_j+q(x))}e^{-ix\xi}\,dx d\eta\\
		&=\mathbf{1}_+(\xi) \int_\R (P^+|\nabla|g_j)(x-a_j+q(x))(1+q'(x))e^{-ix\xi}\,dx.
		\end{split}
		\end{equation*}
		Similarly,
		\begin{equation}\label{nabla6.2}
		\begin{split}
		G_1(\xi)=-\mathbf{1}_+(\xi) \int_\R (P^-|\nabla|g_j)(x-a_j+q(x))(1+q'(x))e^{-ix\xi}\,dx.
		\end{split}
		\end{equation}
		Similarly, we multiply the identity \eqref{nabla5} by $\mathbf{1}_-(\xi)$ and integrate by parts in $x$ to derive the identity
		\begin{equation}\label{nabla7}
		\begin{split}
		\mathbf{1}_-(\xi)\widehat{|\nabla|h}_j(\xi)&=F_2(\xi)+G_2(\xi),\\
		F_2(\xi)&:=\mathbf{1}_-(\xi)\int_\R (P^-|\nabla|g_j)(x-a_j+q(x))(1+q'(x))e^{-ix\xi}\,dx,\\
		G_2(\xi)&:=-\mathbf{1}_-(\xi) \int_\R (P^+|\nabla|g_j)(x-a_j+q(x))(1+q'(x))e^{-ix\xi}\,dx.
		\end{split}
		\end{equation}
		We add up the identities \eqref{nabla6} and \eqref{nabla7} and notice that
		\begin{equation*}
		\begin{split}
		&F_1(\xi)=\int_\R (P^+|\nabla|g_j)(x-a_j+q(x))(1+q'(x))e^{-ix\xi}\,dx + G_2(\xi),\\
		&F_2(\xi)=\int_\R (P^-|\nabla|g_j)(x-a_j+q(x))(1+q'(x))e^{-ix\xi}\,dx + G_1(\xi).
		\end{split}
		\end{equation*}
		Thus
		\begin{equation*}
		\widehat{|\nabla|h_j}(\xi)=\int_\R (|\nabla|g_j)(x-a_j+q(x))(1+q'(x))e^{-ix\xi}\,dx+2G_1(\xi)+2G_2(\xi).
		\end{equation*}
		Therefore, we have the precise identity
		\begin{equation*}\label{nabla8}
		\begin{split}
		&|\nabla|h_j(x)=(|\nabla|g_j)(x-a_j+q(x))(1+q'(x)) + E(x),\\
		&E(x):=\frac{1}{2\pi}\int_{\R}(2G_1(\xi)+2G_2(\xi))e^{ix\xi}\,d\xi.
		\end{split}
		\end{equation*}
		
		It remains to prove the bounds \eqref{nabla4} on the error term. We decompose $E=2E_1+2E_2$,
		\begin{equation*}
		E_1(x):=\frac{1}{2\pi}\int_{\R}G_1(\xi)e^{ix\xi}\,d\xi,\qquad E_2(x):=\frac{1}{2\pi}\int_{\R}G_2(\xi)e^{ix\xi}\,d\xi.
		\end{equation*}
		We begin with $E_1$ and use the formula \eqref{nabla6.2} to write
		\begin{equation}\label{nabla12}
		\begin{split}
		\widehat{P_kE_1}(\xi)&=-\varphi_k(\xi)\mathbf{1}_+(\xi)\int_\R\sum_{k'\in\Z} (P_{k'}^-|\nabla|g_j)(x-a_j+q(x))(1+q'(x))e^{-ix\xi}\,dx\\
		&=-\varphi_k(\xi)\mathbf{1}_+(\xi)\frac{1}{2\pi}\int_{-\infty}^0\sum_{k'\in\Z}\varphi_{k'}(\eta)\widehat{|\nabla|g_j}(\eta)\int_\R e^{i\eta(x-a_j+q(x))}(1+q'(x))e^{-ix\xi}\,dx d\eta,\\
		&=-\varphi_k(\xi)\mathbf{1}_+(\xi)\frac{1}{2\pi}\int_{-\infty}^0\sum_{k'\in\Z}\varphi_{k'}(\eta)\widehat{|\nabla|g_j}(\eta)L(\xi,\eta) d\eta,
		\end{split}
		\end{equation}
		with 
		\begin{equation*}
			L(\xi,\eta):=\int_\R e^{i\eta(x-a_j+q(x))}(1+q'(x))e^{-ix\xi}\,dx.
		\end{equation*}
		We make the change of variables $x=\tilde{Q}(y)$, as defined in \eqref{changeV}, and integrate by parts in $y$ once to get
		\begin{equation*}
			L(\xi,\eta)=e^{-i\eta a_j}\int_{\R}e^{-i\xi\tilde{q}(y)}e^{iy(\eta -\xi)}dy=\frac{\xi e^{-i\eta a_j}}{\eta-\xi}\int_{\R}\tilde{q}'(y)e^{i(\eta y-\xi\tilde{Q}(y))}dy.
		\end{equation*}
		Notice that we expanded $\tilde{Q}(y)=y+\tilde{q}(y)$ in order to get the necessary smallness factor from $\tilde{q}'(y)$ which we don't have for $\tilde{Q}'(y)$. 
  
  For any $k,k'\in\Z$ we let $k^\ast=\max(k,k')$ and decompose the kernel $L$ as $L=L_1+L_2$ with
		\begin{equation*}
			\begin{split}
			&L_1(\xi,\eta):=\frac{\xi e^{-i\eta a_j}}{\eta-\xi}\int_{\R}\tilde{q}'(y)e^{i(\eta y-\xi\tilde{Q}(y))}\Phi_{\leq -k^\ast}(y)dy,\\
			&L_2(\xi,\eta):=\frac{\xi e^{-i\eta a_j}}{\eta-\xi}\int_{\R}\tilde{q}'(y)e^{i(\eta y-\xi\tilde{Q}(y))}(1-\Phi_{\leq -k^\ast})(y)dy,
			\end{split}
		\end{equation*}
		where here $\Phi_{\leq -k^\ast}$ is defined as in \eqref{q_bounds}. We claim that for $\xi>0$, $|\xi|\in[2^{k-1},2^{k+1}]$ and $\eta<0$, $|\eta|\in[2^{k'-1},2^{k'+1}]$, we have
		\begin{equation}\label{claim}
			|L_n(\xi,\eta)|\lesssim\ep 2^{k-2k^\ast},\qquad n\in\{1,2\}.
		\end{equation}
		
		The bounds for $L_1$ follow in a straightforward way from the support of the cutoff function and the smallness factor from \eqref{q_bounds}. For the bounds on $L_2$ we integrate by parts once more in $y$ to get
		\begin{equation*}
		\begin{split}
			L_2(\xi,\eta)&=i\frac{\xi}{\eta-\xi}e^{-i\eta a_j}\int_{\R}e^{i(\eta y-\xi\tilde{Q}(y))}\frac{d}{dy}\Big\{\frac{\tilde{q}'(y)(1-\Phi_{\leq -k^\ast}(y))}{\eta-\xi\tilde{Q}'(y)}\Big\}dy\\
			&=i\xi e^{-i\eta a_j}\int_{\R}e^{i(\eta y-\xi\tilde{Q}(y))}\frac{ \tilde{q}''(y)}{(\eta-\xi\tilde{Q}'(y))^2}(1-\Phi_{\leq -k^\ast}(y))\,dy+\ep O(2^{k-2k^\ast}).
		\end{split}
		\end{equation*}
		Since the bounds for $\partial_y^2\tilde{q}$ are not sufficient and lead to logarithmic losses, we perform another integration by parts in $y$ to obtain
		\begin{equation*}
		\begin{split}
			|L_2(\xi,\eta)|&\lesssim \ep\frac{2^k}{2^{2k^\ast}}+2^k\bigg|\int_{\R}e^{i(\eta y-\xi\tilde{Q}(y))}\frac{d}{dy}\Big\{\frac{\tilde{q}''(y)(1-\Phi_{\leq -k^\ast}(y))}{(\eta-\xi\tilde{Q}'(y))^3}\Big\}dy\bigg|\\
			&\lesssim\ep\frac{2^k}{2^{2k^\ast}}+2^k\int_{\R}\Big(\frac{|\tilde{q}^{(3)}(y)|\cdot|1-\Phi_{\leq -k^\ast}(y)|}{2^{3k^\ast}}+\frac{|\tilde{q}^{\prime\prime}(y)|\cdot|\partial_y(\Phi_{\leq -k^\ast})(y)|}{2^{3k^\ast}}+\frac{|\xi||\tilde{q}^{\prime\prime}(y)|^2\cdot|1-\Phi_{\leq -k^\ast}(y)|}{2^{4k^\ast}}\Big)dy\\
			&\lesssim\ep 2^{k-2k^\ast},
		\end{split}	
		\end{equation*}
		where we used \eqref{qtilde_bounds} in the last line, thus proving the claim \eqref{claim}. It now follows from \eqref{nabla12} that
		\begin{equation*}
		\begin{split}
		|\widehat{P_kE_1}(\xi)|&\lesssim \ep \varphi_k(\xi)\int_\R\sum_{k'\in\Z}\frac{2^k}{2^{2k}+2^{2k'}}|\varphi_{k'}(\eta)\widehat{|\nabla|g}_j(\eta)|\,d\eta\\
		&\lesssim\ep \varphi_k(\xi)\sum_{k'\in\Z}\frac{2^k2^{k'}}{2^{2k}+2^{2k'}}2^{k'/2}\|P_{k'}g_j\|_{L^2}
		\end{split}
		\end{equation*}
		By using the Cauchy-Schwartz inequality we now get
		\begin{equation*}
		\begin{split}
	\|\widehat{P_kE_1}\|_{L^2}
		&\lesssim\ep\sum_{k'\in\Z}\frac{2^{3k/2}2^{3k'/2}}{2^{2k}+2^{2k'}}\|P_{k'}g_j\|_{L^2}.
		\end{split}
		\end{equation*}
		The expression $P_kE_2(x)$, can be estimated in a similar way, thus concluding the proof.
	\end{proof}
	
	The free evolution of $g_j$ will lie in $Z_1$, while we will show that the forced terms land in $Z_2$. We thus have to study the right-hand side terms in \eqref{ga_eq} in $Z_2$ for $g_j$ in the sum space $Z$. The following lemma provides estimates for one of the terms. Due to the bounds on $q$, given in \eqref{q_bounds}, it is natural to expect that the term involving a first derivative of $q$ would be an error term.

	\begin{lemma}\label{lem:err2}
		
		Let $g_j\in Z$ and $q(t,x)$ satisfy \eqref{q_bounds}. Then, for any $k\in\Z$ and $t\in[0,\infty)$,
		$$\|P_k\big((\partial_xq)(t,\tilde{Q}(t,x))|\nabla|g_j(t,x)\big)\|_{L^2}\lesssim \varepsilon2^{k/2}(2^kt)^{-1/10}\|g_j\|_{Z}.$$
		
	\end{lemma}
	\begin{proof}
		Let us denote $u(t,x)=(\partial_xq)(t,\tilde{Q}(t,x))$. We expand
		\begin{equation*}
		\begin{aligned}
		P_k(|\nabla|g_j\cdot u)&=P_k[(P_{[k-3,k+3]}|\nabla|g_j)\cdot P_{\leq k-4}u]+P_k[(P_{\leq k-4}|\nabla|g_j)\cdot P_{[k-3,k+3]}u]\\
		&+\sum_{k_1,k_2\geq k-3,\,|k_1-k_2|\leq 6}P_k[(P_{k_1}|\nabla|g_j)\cdot P_{k_2}u]\\
		&:=S_1+S_2+S_3.
		\end{aligned}
		\end{equation*}
		Then we estimate
		\begin{equation}\label{musk6}
		\begin{split}
		\|S_1\|_{L^2}\lesssim \|P_{[k-3,k+3]}|\nabla|g_j\|_{L^2}\cdot\|P_{\leq k-4}u\|_{L^\infty}\lesssim\varepsilon2^{k/2}\min\{1,(2^kt)^{-1/10}\}\|g_j\|_{Z},\\
		\|S_2\|_{L^2}\lesssim \|P_{\leq k-4}|\nabla|g_j\|_{L^2}\cdot\|P_{[k-3,k+3]}u\|_{L^\infty}\lesssim\varepsilon2^{k/2}\min\{1,(2^kt)^{-1/10}\}\|g_j\|_{Z},
		\end{split}
		\end{equation}
		using \eqref{ZtotEst2} and \eqref{q_bounds}.
		
		To estimate $S_3$, recalling that $Q(t,\tilde{Q}(t,x))=x$, we calculate that $u(t,x)$ can be rewritten as
		\begin{equation*}
			u(t,x)=\frac{-\partial_x\tilde{q}(t,x+a_j)}{1+\partial_x\tilde{q}(t,x+a_j)}.
		\end{equation*}
		We want to show that for any $l\in\mathbb{Z}$ we have
		\begin{equation}\label{claim2}
			\|P_lu\|_{L^2}\lesssim\ep2^{-l/2}.
		\end{equation}
		By translation invariance, in what follows, we can assume without loss of generality that $a_j=0$. We begin by splitting $P_lu(t,x)=P_lu_1(t,x)+P_lu_2(t,x)$ where 
		\begin{equation*}
		\begin{split}
			&P_lu_1(t,x):=P_l\Big(\frac{-\partial_x\tilde{q}(t,x)}{1+\partial_x\tilde{q}(t,x)}\Phi_{\leq -l}(x)\Big),\\
			&P_lu_2(t,x):=P_l\Big(\frac{-\partial_x\tilde{q}(t,x)}{1+\partial_x\tilde{q}(t,x)}(1-\Phi_{\leq -l}(x))\Big),	
		\end{split}
		\end{equation*}
	where the cutoff functions $\Phi_{\leq -l}$ are defined as in \eqref{q_bounds}. For $P_lu_1(t,x)$, using \eqref{q_bounds}, we easily get
	\begin{equation}\label{claim21}
		\|P_lu_1\|_{L^2}\lesssim\ep2^{-l/2}.
	\end{equation}
	To estimate $P_lu_2$ we need to take an extra derivative.  We write
	\begin{equation}\label{claim3}
		\|P_lu_2\|_{L^2}\lesssim2^{-l}\|P_lu'_2\|_{L^2},
	\end{equation}
	and we compute 
	\begin{equation*}
		P_l\partial_xu_2(t,x)=P_l\bigg(\frac{-\partial_x^2\tilde{q}(t,x)(1-\Phi_{\leq -l}(x))}{(1+\partial_x\tilde{q}(t,x))^2}+\frac{\partial_x\tilde{q}(t,x)(\partial_x\Phi_{\leq -l})(x)}{1+\partial_x\tilde{q}(t,x)}\bigg),
	\end{equation*}
	which, using \eqref{qtilde_bounds}, yields
	\begin{equation}\label{claim4}
		\|P_l\partial_xu_2\|\lesssim\ep2^{l/2}.
	\end{equation}
	Combining \eqref{claim3} with \eqref{claim4} and \eqref{claim21} yields \eqref{claim2}.Using again \eqref{ZtotEst2} and the Cauchy-Schwartz inequality,
		\begin{equation}\label{musk7}
		\begin{split}
		\|S_3\|_{L^2}&\lesssim \sum_{k_1,k_2\geq k-3,\,|k_1-k_2|\leq 6}2^{k/2}\|P_{k_1}|\nabla|g_j\|_{L^2}\cdot \|P_{k_2}u\|_{L^2}\\
		&\lesssim 2^{k/2}\sum_{k_1,k_2\geq k-3,\,|k_1-k_2|\leq 6}2^{k_1/2}(1+2^{k_1}t)^{-1/10}\|g_j\|_Z\cdot \varepsilon 2^{-k_2/2}\\
		&\lesssim \varepsilon2^{k/2}(2^kt)^{-1/10}\|g_j\|_{Z}.
		\end{split}
		\end{equation}
		The desired conclusion follows from \eqref{musk6} and \eqref{musk7}.
	\end{proof}

	\section{Reformulation in terms of a pseudoproduct}\label{sec:LitPal}

	We would now like to rewrite the nonlinearity $\mathcal{N}$ defined in \eqref{nonlin2} as a pseudoproduct. We begin by introducing some extra notation. Here and for the rest of the paper, we will define $\underline{\xi}\in\R^{2n+1}$ as $\underline{\xi}:=(\xi_1,\ldots,\xi_{2n+1})$ for scalars and similarly, for functions, $\underline{f}:=(f_1,f_2,\ldots,f_{2n+1})$. 
	
	We now pass to Fourier variables and write
	\begin{equation*}\label{eq23}
	\begin{split}
	&h(x)=\frac{1}{2\pi}\int_\R \widehat{h}(\xi)e^{ix\xi}\,d\xi,\\
	&h^\ast(x,\alpha)=\frac{1}{2\pi}\int_{x-\alpha}^x\frac{1}{\alpha}\Big\{\int_\R \widehat{h}(\xi)e^{iy\xi}\,d\xi\Big\}\,dy=\frac{1}{2\pi}\int_\R \widehat{h}(\xi)e^{ix\xi}\frac{1-e^{-i\alpha\xi}}{i\alpha\xi}\,d\xi,\\
	&\partial_xh^\ast(x,\alpha)=\frac{1}{2\pi}\int_\R \widehat{h}(\xi)e^{ix\xi}\frac{1-e^{-i\alpha\xi}}{\alpha}\,d\xi.
	\end{split}
	\end{equation*}
	We substitute these formulas into \eqref{nonlin2}, and we now need to understand the multilinear pseudoproduct
	\begin{equation}\label{nonlin3}
	\mathcal{N}_n(x)=\frac{(-1)^n}{\pi}\frac{1}{(2\pi)^{2n+1}}\frac{d}{dx}\bigg\{\int_{\R^{2n+1}}\prod_{\ell=1}^{2n+1}\hat{h}(\xi_\ell)e^{ix(\xi_1+\ldots+\xi_{2n+1})}m_n(\underline{\xi})\,d\underline{\xi}\bigg\},
	\end{equation}
	where the multiplier $m_n:\mathbb{R}^{2n+1}\to\mathbb{C}$ is given by
	\begin{equation*}\label{nonlin4}
	m_n(\underline{\xi}):=-\xi_1\sum_{\ell=1}^{2n+1}\xi_\ell\int_{\R}\prod_{\imath=1}^{2n+1}\frac{1-e^{-i\alpha\xi_\imath}}{i\alpha\xi_\imath}d\alpha.
	\end{equation*}
	By writing
	\begin{equation*}
		\widehat{f_\ell}(\xi_\ell)=\int_\R f_\ell(y_\ell)e^{-i\xi_\ell y_\ell}\,dy_\ell,\qquad \ell\in\{1,\ldots,2n+1\},
	\end{equation*}
	the identity \eqref{nonlin3} can be written in the physical space
	\begin{equation*}\label{nonlin5}
	\mathcal{T}_n(\underline{f})(x)=\frac{(-1)^n}{\pi}\frac{d}{dx}\int_{\R^{2n+1}}\frac{d}{dy_1}\prod_{i=1}^{2n+1}f_i(y_i)K_n(x-\underline{y})d\underline{y},
	\end{equation*}
	with the kernel defined as
	\begin{equation*}\label{nonlin6}
	K_n(\underline{z}):=\int_{\R^{2n+1}}e^{i\sum_{\ell=1}^{2n+1}z_\ell\xi_\ell}m_n^s(\underline{\xi})\,d\underline{\xi}\qquad\text{with}\qquad m_n^s(\underline{\xi})=\frac{1}{(2\pi)^{2n+1}}\int_{\R}\prod_{i=1}^{2n+1}\frac{1-e^{-i\alpha\xi_i}}{i\alpha\xi_i}d\alpha.
	\end{equation*}
	We remark that pseudoproducts can also be analyzed in conjunction with the Littlewood-Paley projections. Indeed, \eqref{nonlin3} shows that
	\begin{equation}\label{nonlin8}
	\mathcal{T}_n(\underline{f})(x)=\frac{(-1)^n}{\pi}\frac{d}{dx}\bigg(\sum_{\underline{k}\in\mathbb{Z}^{2n+1}}\int_{\R^{2n+1}}i\xi_1\prod_{\ell=1}^{2n+1}\widehat{P_{k_\ell}f_\ell}(\xi_\ell) m^s_{\underline{k}}\,d\underline{\xi}\bigg)\quad\text{with}\quad
	m^s_{\underline{k}}(\underline{\xi})=m^s(\underline{\xi})\prod_{i=1}^{2n+1}\tilde{\varphi}_{k_i}(\xi_i),
	\end{equation}
	where, with $\widetilde{\varphi}_l:=\sum_{|a|\leq 2}\varphi_{l+a}$, 
	\begin{equation*}\label{eq31}
	\begin{split}
	m^s_{k_1,\ldots,k_{2n+1}}(\underline{\xi}):=m^s(\underline{\xi})\prod_{\ell=1}^{2n+1}\widetilde{\varphi}_{k_\ell}(\xi_\ell).
	\end{split}
	\end{equation*}
	Throughout our analysis, we will need frequency cutoffs in \eqref{nonlin8} on both the functions (in order to estimate the pseudoproduct) and on the multiplier (to obtain bounds on the kernel). In order to be allowed to do this, we need to ensure that $\varphi_l\tilde{\varphi}_l=\varphi_l$, and hence take a slightly larger support for $\tilde{\varphi}_l$.
	
	In order to work in the physical space, we define
	\begin{equation}\label{Lk}
	L_k(x,\alpha)=\frac{1}{2\pi}\int_\R\widetilde{\varphi}_k(\xi)\frac{1-e^{-i\alpha\xi}}{i\xi\alpha}e^{ix\xi}\,d\xi,
	\end{equation}
	and, with $\widetilde{\varphi}_{\leq k}:=\varphi_{\leq k+2}$,
	\begin{equation}\label{Llk}
	L_{\leq k}(x,\alpha)=\frac{1}{2\pi}\int_\R\widetilde{\varphi}_{\leq k}(\xi)\frac{1-e^{-i\alpha\xi}}{i\xi\alpha}e^{ix\xi}\,d\xi.
	\end{equation}
	The full kernel can now be expressed as
	\begin{equation}\label{kernel1}
	K_{k_1, \ldots, k_{2n+1}}(\underline{x})=\int_\R \prod_{\ell=1}^{2n+1} L_{k_\ell}(x_\ell,\alpha)\,d\alpha.
	\end{equation}
	In physical space, the trilinear pseudoproduct hence takes the form
	\begin{equation}\label{nonlin8.1}
	\mathcal{T}_n(P_{k_1}f_1,\ldots,P_{k_{2n+1}}f_{2n+1})(x)=\frac{(-1)^n}{\pi}\frac{d}{dx}\int_{\R^{2n+1}}\frac{d}{dy_1}\prod_{\ell=1}^{2n+1}P_{k_\ell}f_\ell(y_\ell)K_{\underline{k}}(x-\underline{y})\,d\underline{y}.
	\end{equation}

	\subsection{The kernels $L_k$ and $L_{\leq k}$}
	
	Our goal is to analyze the kernels $K_{k_1,\ldots,k_{2n+1}}$ defined in \eqref{kernel1}. To this end, we have the following lemma.
	
	\begin{lemma}\label{lemma:K}
		(i) Let $L_k$ and $L_{\leq k}$ be defined as in \eqref{Lk} and \eqref{Llk}. Then the following bounds hold:
		\begin{equation}\label{min}
		\int_{\R}|L_k(x,\alpha)|\,dx\lesssim\min(1,(2^k|\alpha|)^{-1}),
		\end{equation}
		and
		\begin{equation}\label{Lleq_k}
		\int_{\R}|L_{\leq k}(x,\alpha)|\,dx\lesssim 1.
		\end{equation}
		
		(ii) Letting
		\begin{equation*}\label{kernel2}
		K_{k_1, k_2, \leq k_3\ldots,\leq k_{2n+1}}(\underline{x}):=\int_\R |L_{k_1}(x_1,\alpha)L_{ k_2}(x_2,\alpha)\prod_{\ell=3}^{2n+1}L_{\leq k_\ell}(x_\ell,\alpha)|\,d\alpha,
		\end{equation*}
		we have
		\begin{equation}\label{musk10.5}
		\|K_{k_1, k_2, \leq k_3,\ldots,\leq k_{2n+1}}\|_{L^1(\R^{2n+1})}\lesssim 2^{-\max(k_1,k_2)}(1+|k_1-k_2|).
		\end{equation}
	\end{lemma} 
	
	\begin{proof}
		(i) From \eqref{Lk}, we have
		\begin{equation*}
		L_k(x,\alpha)=\frac{1}{2\pi}\int_{\R}\widetilde{\varphi}_0(\xi/{2^k})\frac{1-e^{-i\alpha\xi}}{i\xi\alpha}e^{ix\xi}\,d\xi=\frac{1}{2\pi}\int_{\R}\widetilde{\varphi}_0(\xi)e^{i2^kx\xi}\frac{1-e^{-i\alpha\xi2^k}}{i\xi\alpha}\,d\xi.
		\end{equation*}
		We now denote by $\psi_0(x)$ the inverse Fourier transform of $\widetilde{\varphi}_0(\xi)/i\xi$ and we thus obtain the formula 
		\begin{equation}\label{Lk_formula}
		L_k(x,\alpha)=\frac{\psi_0(2^kx)-\psi_0(2^k(x-\alpha))}{\alpha}.
		\end{equation}
		We now use \eqref{Lk_formula} to evaluate the $L^1$-norm of $L_k$. For $2^k|\alpha|\geq1$, we can write
		\begin{equation*}
		\int_{\R}\bigg|\frac{[\psi_0(2^kx)-\psi_0(2^k(x-\alpha))]}{\alpha}\bigg|\,dx\leq
		\int_{\R}\bigg|\frac{\psi_0(y)}{2^k\alpha}\bigg|+\bigg|\frac{\psi_0(y-2^k\alpha)}{2^k\alpha}\bigg|\,dy\lesssim(2^k|\alpha|)^{-1},
		\end{equation*}
		where in the last step we used the fact that $\psi_0$, being the inverse of the Fourier transform of $\widetilde{\varphi}_0(\xi)/(i\xi)$,  is a Schwartz function, and is hence bounded in $L^1$-norm. For $2^k|\alpha|\leq1$, we use the property that all Schwartz functions satisfy
		\begin{equation*}
		|\psi_0(p)-\psi_0(\rho+p)|\lesssim|\rho|\langle p\rangle^{-4},
		\end{equation*}
		for $|\rho|\leq1$ and any $p\in\R$. As a result, we get
		\begin{equation*}\label{aux1}
		\int_{\R}\bigg|\frac{[\psi_0(2^kx)-\psi_0(2^k(x-\alpha))]}{\alpha}\bigg|\,dx\lesssim\int_{\R}\frac{|\alpha2^k|\langle2^kx\rangle^{-4}}{|\alpha|}\,dx=\int_{\R}\langle y\rangle^{-4}dy\lesssim1,
		\end{equation*}
		thus completing the proof for \eqref{min}.
		
		We now prove \eqref{Lleq_k}. We have
		\begin{equation*}
		L_{\leq k}(x,\alpha)=\frac{1}{2\pi}\int\widetilde{\varphi}_{\leq 0}(\xi)e^{i2^kx\xi}\frac{1-e^{-i\alpha\xi2^k}}{i\xi\alpha}\,d\xi.
		\end{equation*}
		Denoting by $\psi_{\leq 0}(x)$ the inverse Fourier transform of $\widetilde{\varphi}_{\leq 0}(\xi)/i\xi$ yields
		\begin{equation}\label{Lleq_k_formula}
		L_{\leq k}(x,\alpha)=\frac{\psi_{\leq0}(2^kx)-\psi_{\leq0}(2^k(x-\alpha))}{\alpha}.
		\end{equation}
		However, since $\widetilde{\varphi}_{\leq0}(\xi)/i\xi$ has a singularity at $0$ the function $\psi_{\leq0}$ is not in the class of Schwartz functions, thus preventing us from directly taking the $L^1$-norm as we did above. Instead, we notice that the derivative of $\psi_{\leq 0}$, given by 
		\begin{equation*}
		\psi_{\leq0}'(x)=\frac{1}{2\pi}\partial_{x}\bigg(\int_{\R}\frac{\varphi_{\leq 2}(\xi)e^{ix\xi}}{i\xi}\,d\xi\bigg)=\frac{1}{2\pi}\int_{\R}\varphi_{\leq 2}(\xi)e^{ix\xi}\,d\xi,
		\end{equation*}
		is in the class of Schwartz functions. We thus rewrite the difference in \eqref{Lleq_k_formula} as
		\begin{equation*}\label{psi'}
		\Big|\frac{\psi_{\leq0}(2^kx)-\psi_{\leq0}(2^k(x-\alpha))}{\alpha}\Big|=\frac{1}{|\alpha|}\Big|\int_{2^k(x-\alpha)}^{2^kx}\psi_{\leq0}'(y)\,dy\Big|\lesssim\frac{1}{|\alpha|}\int_{2^k(x-\alpha)}^{2^kx}\langle y\rangle^{-6}\,dy.
		\end{equation*}
		We thus obtain the bounds
		\begin{equation}\label{alternate_bounds}
		\Big|\frac{\psi_{\leq0}(2^kx)-\psi_{\leq0}(2^k(x-\alpha))}{\alpha}\Big|\lesssim\frac{1}{|\alpha|}\bigg(\mathbbm{1}_{|x|\geq2|\alpha|}2^k|\alpha|\langle2^kx\rangle^{-4}+\mathbbm{1}_{|x|\leq2|\alpha|} \bigg).
		\end{equation}
		Integrating \eqref{alternate_bounds} in $x$ now yields \eqref{Lleq_k}.
		
		(ii) Without loss of generality we may assume that $k_1\geq k_2$. Using \eqref{min}--\eqref{Lleq_k} we estimate
		\begin{equation*}
		\begin{split}
		\int_{\R^3}|K_{k_1,k_2,\leq k_3,\ldots,\leq k_{2n+1}}(\underline{x})|\,d\underline{x}&\lesssim \int_{\R}\min(1,(2^{k_1}|\alpha|)^{-1})\min(1,(2^{k_2}|\alpha|)^{-1})\,d\alpha\\
		&\lesssim 2^{-k_1}(1+|k_1-k_2|),
		\end{split}
		\end{equation*}
		as claimed.
	\end{proof}
	
	\subsection{The modified kernels}
	For certain estimates of the pseudoproduct, working simply with the $L^1$-norm of the kernel will lead to a logarithmic loss. To overcome this difficulty, we introduce the following decomposition of  ${L}_k$ and $L_{\leq k}$,
	\begin{equation}\label{modified_kernel}
	\begin{split}
	L_k(x,\alpha)&=\widetilde{L}_k(x,\alpha)+\frac{1}{\alpha}\min(1,2^k\alpha)\psi_0'(2^kx),\\
	L_{\leq k}(x,\alpha)&=\widetilde{L}_{\leq k}(x,\alpha)+\frac{1}{\alpha}\min(1,2^k\alpha)\psi_{\leq 0}'(2^kx),
	\end{split}
	\end{equation}
	where here
	\begin{equation}\label{Lktilde}
	\begin{split}
	\widetilde{L}_k(x,\alpha)&:=\frac{\psi_0(2^kx)-\psi_0(2^k(x-\alpha))}{\alpha}-\frac{1}{\alpha}\min(1,2^k\alpha)\psi_0'(2^kx),\\
	\widetilde{L}_{\leq k}(x,\alpha)&:=\frac{\psi_{\leq 0}(2^kx)-\psi_{\leq 0}(2^k(x-\alpha))}{\alpha}-\frac{1}{\alpha}\min(1,2^k\alpha)\psi'_{\leq 0}(2^kx).
	\end{split}
	\end{equation}
	The following lemma provides us with bounds on the $L^1$-norm of the modified kernels $\widetilde{L}_k$ and $\widetilde{L}_{\leq k}$ and highlights the improvement from Lemma \ref{lemma:K}.
	
	\begin{lemma}\label{lemma:K2}
		(i) Let $\widetilde{L}_k$ and $\widetilde{L}_k$ be defined as in \eqref{Lktilde}. Then 
		\begin{equation}\label{tilde_L}
		\begin{split}
		&\int_{\R}|\widetilde{L}_k(x,\alpha)|\,dx\lesssim\min(2^k|\alpha|,(2^k|\alpha|)^{-1}),\\
		&\int_{\R}|\widetilde{L}_{\leq k}(x,\alpha)|\,dx\lesssim\min(2^k|\alpha|,1).
		\end{split}
		\end{equation}
		
		(ii) In particular, if 
		\begin{equation}\label{kernel6}
		\widetilde{K}_{k_1, k_2, \leq k_3,\ldots,\leq k_{2n+1}}(\underline{x}):=\int_\R |L_{k_1}(x_1,\alpha)\widetilde{L}_{ k_2}(x_2,\alpha)L_{\leq k_3}(x_3,\alpha)\ldots L_{\leq k_{2n+1}}|\,d\alpha,
		\end{equation}
		and if $k_1\geq k_2$, then
		\begin{equation}\label{musk10.6}
		\|\widetilde{K}_{k_1, k_2, \leq k_3\ldots,\leq k_{2n+1}}\|_{L^1(\R^{2n+1})}\lesssim 2^{-k_1}.
		\end{equation}
		
	\end{lemma}
	
	\begin{proof}
		For $2^k|\alpha|\geq1$, the bounds follow from Lemma \ref{lemma:K} and \eqref{modified_kernel}.
        When $2^k|\alpha|\leq1$ we have
		\begin{equation*}
		\widetilde{L}_k(x,\alpha)=\frac{\psi_0(2^kx)-\psi_0(2^k(x-\alpha))-2^k\alpha\psi_0'(2^kx)}{\alpha},
		\end{equation*}
		from which we get
		\begin{equation*}\label{aux}
		\begin{aligned}
		\int_\R\bigg|\frac{\psi_{0}(2^kx)-\psi_{0}(2^k(x-\alpha))-2^k\alpha\psi_{0}'(2^kx)}{\alpha} \bigg|\,dx
		&\leq\int_{\R}\frac{1}{2^k|\alpha|}\bigg|\int_{y-2^k\alpha}^{y}[\psi_{0}'(p)-\psi_{0}'(y)]\,dp\bigg|dy\\
		&\leq\frac{1}{2^k|\alpha|}\int_{\R}\int_{y-2^k\alpha}^{y}\int_y^{ p}|\psi_{0}''(s)|\,dsdpdy\\
		&\lesssim2^k|\alpha|.   
		\end{aligned}
		\end{equation*}
		thus concluding the proof of the bounds in the first line of \eqref{tilde_L}.
		
		The proof of the bounds in the second line of \eqref{tilde_L} is similar, since $\psi'_{\leq 0}$ is a Schwarz function. Finally, to prove \eqref{musk10.6} we estimate
		\begin{equation*}
		\begin{split}
		\int_{\R^3}|\widetilde{K}_{k_1,k_2,\leq k_3}(x_1,x_2,x_3)|\,dx_1dx_2dx_3&\lesssim \int_{\R}\min(1,(2^{k_1}|\alpha|)^{-1})\min((2^{k_2}|\alpha|,(2^{k_2}|\alpha|)^{-1})\,d\alpha\lesssim 2^{-k_1},
		\end{split}
		\end{equation*}
		as claimed.
	\end{proof}

\subsection{Study of the functions in the $Z_1$ space}
Finally, before we can begin estimating the psuedoproducts, we need a further understanding of the space $Z_1$. Specifically, in the following lemmas, we take advantage of the shape of the functions in the space $Z_1$ to split them into core and error terms. Assume that $g\in Z_1$ and define
\begin{equation}\label{T}
T(g)(t,x,\alpha,k):=\int_{\R}g(t,x-a-y+q(t,x-y))\frac{\psi_{\leq 0}(2^{k}y)-\psi_{\leq 0}(2^{k}(y-\alpha))}{\alpha}\,dy,
\end{equation}
where $k\in\Z$ and $a\in\R$. Recall that
\begin{equation*}
\psi_{\leq 0}(x)=\frac{1}{2\pi}\int_{\R}\frac{1}{i\xi}e^{ix\xi}\varphi_{\leq 2}(\xi)\,d\xi.
\end{equation*}
It is easy to see that $\psi_{\leq 0}$ is an odd function on $\R$ and
\begin{equation}\label{musk22}
\begin{split}
&(\partial_x\psi_{\leq 0})(x)=\frac{1}{2\pi}\int_{\R}e^{ix\xi}\varphi_{\leq 2}(\xi)\,d\xi,\\
&\lim_{x\to\infty}\psi_{\leq 0}(x)=1/2,\qquad \lim_{x\to-\infty}\psi_{\leq 0}(x)=-1/2.
\end{split}
\end{equation}

By looking at \eqref{T}, we notice that if $1/\alpha$ had a power that is only slightly better than $1$, the expression would be integrable. This motivates the following lemma.

\begin{lemma}[Decomposition of functions in $Z_1$]\label{lemma:error}
	Assume $\|g\|_{Z_1}= 1$, $k\in\Z$, $a\in\R$, and $|\alpha|\gtrsim 2^{-k}$. Then
	\begin{equation*}
	\big|T(t,x,\alpha,k)-p(t,x,\alpha)\big|\lesssim\bigg[\frac{|\alpha||x-a+q(t,x)|}{|x-a+q(t,x)|^2+|\alpha|^2}\bigg]^{9/10}+\frac{1}{2^{k}|\alpha|},
	\end{equation*}
	where here
	\begin{equation}\label{pbigDef}
	p(t,x,\alpha):=
	\begin{cases}
	0 & \text{if }|\alpha|\in [|x-a+q(t,x)|/4,4|x-a+q(t,x)|],\\
	g^+(t,x,\alpha)&\text{if }|\alpha|> 4|x-a+q(t,x)|,\quad\alpha>0,\\
	g^-(t,x,\alpha) &\text{if }|\alpha|> 4|x-a+q(t,x)|,\quad\alpha<0,\\
	g(t,x-a+q(t,x)) &\text{if }|\alpha|< |x-a+q(t,x)|/4,
	\end{cases}
	\end{equation}
	\begin{equation}\label{gpgm}
	\begin{aligned}
	g^+(t,x,\alpha)&:=\frac{1}{\alpha}\int_0^\alpha g(t,-y+q(t,x-y)-q(t,x))\,dy,\\
	g^-(t,x,\alpha)&:=\frac{-1}{\alpha}\int_{\alpha}^0 g(t,-y+q(t,x-y)-q(t,x))\,dy.
	\end{aligned}
	\end{equation}
\end{lemma}
\begin{proof}
	We fix $|\alpha|\gtrsim 2^{-k}$ and $x\in\mathbb{R}$. For $|x-a+q(t,x)|\leq |\alpha|/4$ and $\alpha>0$, we begin by observing that
	\begin{equation*}
	\psi_{\leq 0}(2^{k}y)-\psi_{\leq 0}(2^{k}(y-\alpha))=2^{k}\int_{y-\alpha}^{y}\psi_{\leq 0}'(2^{k}z)\,dz=\int_{2^{k}(y-\alpha)}^{2^{k}y}\psi_{\leq 0}'(\rho)\,d\rho,
	\end{equation*}
	where $\psi_{\leq 0}'$ is a Schwartz function. As a result, for the values of $y$ such that $y$ and $y-\alpha$ have the same sign we get errors terms. In fact, for $y\geq\alpha $, we have
	\begin{equation*}
	\bigg|\int_{y\geq\alpha}g(t,x-a-y+q(t,x-y))\frac{\psi_{\leq 0}(2^{k}y)-\psi_{\leq 0}(2^{k}(y-\alpha))}{\alpha}dy\bigg|\lesssim\frac{\Vert g\Vert_{L^\infty}}{|\alpha|}\bigg|\int_{y\geq\alpha}\langle 2^{k}(y-\alpha)\rangle^{-2}dy\bigg|\lesssim\frac{1}{2^{k}|\alpha|}.
	\end{equation*}
	Similarly, for $y\leq 0$, we get
	\begin{equation*}
	\bigg|\int_{y\leq 0}g(t,x-a-y+q(t,x-y))\frac{\psi_{\leq 0}(2^{k}y)-\psi_{\leq 0}(2^{k}(y-\alpha))}{\alpha}dy\bigg|\lesssim\frac{\Vert g\Vert_{L^\infty}}{|\alpha|}\bigg|\int_{y\leq 0}\langle 2^{k}y\rangle^{-2}dy\bigg|\lesssim\frac{1}{2^{k}|\alpha|}.
	\end{equation*}	
	It remains to understand the core terms, when $y\in [0,\alpha]$ and we have to integrate through $0$. The idea here is to  approximate the $\psi_{\leq 0}(2^{k}y)$ and $\psi_{\leq 0}(2^{k}(y-\alpha))$ terms by the value of $\psi_{\leq 0}$ at $\pm\infty$ respectively. More precisely, using \eqref{musk22} we write
	\begin{equation}\label{C_0}
	\begin{split}
	&|\psi_{\leq 0}(2^{k}y)-1/2|=\Big|\int_{2^{k}y}^{\infty}\psi_{\leq 0}'(\rho)d\rho\Big|\lesssim \langle 2^ky\rangle^{-2},\\
	&|\psi_{\leq 0}(2^{k}(y-\alpha))+1/2|=\Big|\int_{-\infty}^{2^{k}(y-\alpha)}\psi_{\leq 0}'(\rho)d\rho\Big|\lesssim \langle 2^k(y-\alpha)\rangle^{-2}.
	\end{split}
	\end{equation}
	Estimating as before and recalling also the definition \eqref{gpgm} we have
	\begin{equation*}\label{core1}
	\begin{aligned}
	&\int_{0}^{\alpha}g(t,x-a-y+q(t,x-y))\frac{\psi_{\leq 0}(2^{k}y)-1/2+1/2-(\psi_{\leq 0}(2^{k}(y-\alpha))+1/2-1/2)}{\alpha}dy\\
	&=\int_0^\alpha\frac{g(t,x-a-y+q(t,x-y))}{\alpha}dy+O\left(\frac{\Vert g\Vert_{L^\infty}}{2^{k}|\alpha|}\right)\\
	&=g^+(t,x,\alpha)+\int_0^\alpha\frac{g(t,x-a-y+q(t,x-y))-g(t,-y+q(t,x-y)-q(t,x))}{\alpha}dy+O\left(\frac{1}{2^{k}|\alpha|}\right).
	\end{aligned}
	\end{equation*}
	
	To estimate the last error term, for simplicity of notation, we set $$X(t,x):=x-a+q(t,x),$$
 and we recall that we are assuming $\vert X(t,x)\vert\le \vert\alpha\vert/4$. Then
	\begin{equation}\label{errorB}
	\begin{split}
	&\Big|\int_0^\alpha\frac{g(t,x-a-y+q(t,x-y))-g(t,-y+q(t,x-y)-q(t,x))}{\alpha}dy\Big|\\
	&\lesssim \frac{\vert X(t,x)\vert}{\alpha}\Vert g\Vert_{L^\infty}+\frac{1}{\alpha}\Big|\int_{2|X(t,x)|}^\alpha\int_0^{X(t,x)}g'(t,\beta-y+q(t,x-y)-q(t,x))\,d\beta dy\Big|\\
	&\lesssim \frac{\vert X(t,x)\vert}{\alpha}+\frac{1}{\alpha}\int_{2|X(t,x)|}^\alpha \frac{|X(t,x)|}{|y|}\,dy\lesssim \bigg[\frac{\vert X(t,x)\vert}{\alpha}\bigg]^{9/10},
	\end{split}
	\end{equation}
	since from the definition of the $Z_1$-norm, we know that $|y||g'(t,\beta-y+q(t,x-y)-q(t,x))|\lesssim 1$. Indeed, since $|y|\geq 2|X(t,x)|$ and $|\beta|\leq |X(t,x)|$, we can assume that the argument in the second line of \eqref{errorB} is of size $y$. This completes the proof in the case $|x-a+q(t,x)|\leq |\alpha|/4$ and $\alpha>0$.
	
	Assume now that $|x-a+q(t,x)|\leq |\alpha|/4$ and $\alpha<0$. The error estimates follow identically and we can thus assume that $y\in[\alpha,0]$. Arguing exactly as above, we obtain
	\begin{equation*}
	\begin{aligned}
	\int_{\alpha}^{0}&g(t,x-a-y+q(t,x-y))\frac{\psi_{\leq 0}(2^{k}y)+1/2-1/2-(\psi_{\leq 0}(2^{k}(y\!-\!\alpha))-1/2+1/2)}{\alpha}dy\\
	&=-\int_\alpha^0\frac{g(t,x-a-y+q(t,x-y))}{\alpha}dy+\frac{1}{2^{k}|\alpha|}O(\Vert g\Vert_{L^\infty})\\
	&=g^-(t,x,\alpha)-\int_\alpha^0\frac{g(t,x-a-y+q(t,x-y))-g(t,-y+q(t,x-y)-q(t,x))}{\alpha}dy+\frac{1}{2^{k}|\alpha|}O(1).
	\end{aligned}
	\end{equation*}
	The error term can be estimated as in \eqref{errorB}, and the desired conclusion follows in this case as well.
	
	We consider now the case $|x-a+q(t,x)|\geq 4|\alpha|$ with $\alpha>0$. For $y\geq\alpha$ and $y\leq0$, we get the same error terms as for the previous case. It remains to set $y\in[0,\alpha]$. Using \eqref{C_0} and estimating any remaining error terms as above, we obtain
	\begin{align*}
	\int_{0}^{\alpha}&g(t,x-a-y+q(t,x-y))\frac{\psi_{\leq 0}(2^{k}y)-\psi_{\leq 0}(2^{k}(y-\alpha))}{\alpha}dy\\
	&=\int_0^\alpha\frac{g(t,x-a-y+q(t,x-y))}{\alpha}dy+\frac{1}{2^{k}|\alpha|}O(1)\\
	&=\frac{1}{\alpha}\int_{0}^\alpha \bigg(g(t,x-a+q(t,x))+\int_{x-a+q(t,x)}^{x-a-y+q(t,x-y)}g'
	(t,\rho)d\rho\bigg) dy+\frac{1}{2^{k}|\alpha|}O(1)\\
	&=g(t,x-a+q(t,x))+\bigg(\frac{|\alpha|}{|x-a+q(t,x)|}+\frac{1}{2^{k}|\alpha|}\bigg)O(1).	
	\end{align*}
	The last step follows from the fact that $\|xg'\|_{L^\infty}\le \Vert g\Vert_{Z_1}\le 1$, the regularity of $q$ in \eqref{q_bounds}, and the assumption $4|\alpha|\leq |x-a+q(t,x)|$. The case $\alpha<0$ follows identically. This completes the proof. 
\end{proof}
\begin{remark}
	Using the change of variable \eqref{changeV}, we can rewrite the core terms $p$ in terms of $y$ to get
	\begin{equation}\label{lpo22}
	\widetilde{p}_j(t,y,\alpha):=
	\begin{cases}
	0 & \text{if }|\alpha|\in [|y-a_j|/4,4|y-a_j|],\\
	\widetilde{g}^+_j(t,y,\alpha)&\text{if }|\alpha|> 4|y-a_j|,\quad\alpha>0,\\
	\widetilde{g}^-_j(t,y,\alpha) &\text{if }|\alpha|> 4|y-a_j|,\quad\alpha<0,\\
	g_j(t,y-a_j) &\text{if }|\alpha|< |y-a_j|/4,
	\end{cases}
	\end{equation}
	\begin{equation}\label{lpo23}
	\begin{aligned}
	\widetilde{g}^+_j(t,y,\alpha)&:=\frac{1}{\alpha}\int_{0}^\alpha g_j[t,-\rho+q(t,\widetilde{Q}(t,y)-\rho)-q(t,\widetilde{Q}(t,y))]\,d\rho,\\
	\widetilde{g}^-_j(t,y,\alpha)&:=\frac{-1}{\alpha}\int_{\alpha}^{0} g_j[t,-\rho+q(t,\widetilde{Q}(t,y)-\rho)-q(t,\widetilde{Q}(t,y))]\,d\rho.
	\end{aligned}
	\end{equation}
\end{remark}

We can further approximate the core terms $p$ of functions in the $Z_1$ space, provided $2^kt\geq2$ and $|\alpha|\leq2t$. The idea of this lemma comes from the uncertainty principle, which tells us that a function with frequencies less than $2^k$ must be constant on intervals of size $2^{-k}$.
\begin{lemma}\label{lemma:error2}
	Assume $\|g\|_{Z_1}= 1$, $a\in\R$, and define the function $p=p(g)$ as in \eqref{pbigDef}--\eqref{gpgm}. Then
	\begin{equation*}
	\big|p(t,x,\alpha)-p^\ast(t,x,\alpha)\big|\lesssim [|\alpha|/t]^{1/50},
	\end{equation*}
	provided that $|\alpha|\leq 2t$, $k(t,|\alpha|)$ is the smallest integer satisfying $2^{k(t,|\alpha|)}\geq (|\alpha|t)^{-1/2}$, and
	\begin{equation*}\label{pbigDef7}
	p^\ast(t,x,\alpha):=
	\begin{cases}
	0 & \text{if }|\alpha|\in [|x-a+q(t,x)|/4,4|x-a+q(t,x)|],\\
	P_{\leq k(t,\alpha)}g(t,0)&\text{if }|\alpha|> 4|x-a+q(t,x)|,\\
	g(t,x-a+q(t,x)) &\text{if }|\alpha|< |x-a+q(t,x)|/4.
	\end{cases}
	\end{equation*}
\end{lemma}

\begin{proof} We only need to consider the case $|\alpha|> 4|x-a+q(t,x)|$. We decompose 
	\begin{equation*}
	g= P_{\leq k(t,|\alpha|)}g+\sum_{l\geq k(t,|\alpha|)+1}P_lg.
	\end{equation*}
	We examine the formula \eqref{gpgm}. For $\alpha>0$ we estimate
	\begin{equation*}
	\begin{split}
	\Big|\frac{1}{\alpha}&\int_0^\alpha g(t,-y+q(t,x-y)-q(t,x))\,dy-P_{\leq k(t,\alpha)}g(t,0)\Big|\\
	&\leq\frac{1}{\alpha}\Big|\int_0^\alpha [P_{\leq k(t,|\alpha|)}g(t,-y+q(t,x-y)-q(t,x))-P_{\leq k(t,\alpha)}g(t,0)]\,dy\Big|\\
	&+\sum_{l\geq k(t,|\alpha|)+1}\frac{1}{\alpha}\int_0^\alpha \big|P_{l}g(t,-y+q(t,x-y)-q(t,x))\big|\,dy\\
	&\lesssim |\alpha|\|\partial_xP_{\leq k(t,|\alpha|)}g(t)\|_{L^\infty_x}+\sum_{l\geq k(t,|\alpha|)+1}\|P_lg(t)\|_{L^\infty}\\
	&\lesssim |\alpha|2^{k(t,|\alpha|)}+(2^{k(t,|\alpha|)}t)^{-1/10},
	\end{split}
	\end{equation*}
	where we used the definition of the $Z_1$ norm in the last line. The desired bounds follow due to the choice of $k(t,|\alpha|)$, $2^{k(t,|\alpha|)}\approx (|\alpha|t)^{-1/2}$. The analysis is similar in the case $\alpha<0$, and the lemma is proved.
\end{proof}
	
	\section{Bounds on trilinear operators}\label{sec:G1G2}
	
	In this section we start our analysis of the nonlinearities $\mathcal{N}_{h_j}$ defined in \eqref{Nha2}. We decompose these nonlinearities and show that most of the components satisfy suitable perturbative estimates, while one particular component requires renormalization.

  For the benefit of the reader, we carry out the analysis for the trilinear nonlinearity (taking into account only the first term of the Taylor expansion in \eqref{nonlin2}). A generalization to the full nonlinearity is carried out in the next section. We hence define the trilinear operator 
\begin{equation*}
\label{Nha1}
\mathcal{N}_{3}[f_1,f_2,f_3](x):=-\frac{1}{\pi}\frac{d}{dx}\bigg\{\int_{\R}\partial_x f_1^*(x,\alpha)f^*_2(x,\alpha)f^*_3(x,\alpha)d\alpha\bigg\}
\end{equation*}
and denote
\begin{equation*}\label{Nha3}
\mathcal{N}_{3,h_j}(x):=\mathcal{N}[h_j,h,h](x),
\end{equation*}
where $h=h(t,x):[0,\infty)\times\mathbb{R}\to\mathbb{R}$ solves
\begin{equation*}\label{eq1}
(\partial_t +|\nabla|)h=\mathcal{N}_{3},\qquad h(0,\cdot)=h_0.
\end{equation*}
Since each logarithmic correction $q_{2n+1}$ cancels the logarithmic loss arising in the respective nonlinearity $\mathcal{N}_{2n+1}$, we also denote by
\begin{equation*}\label{q3}
\mathfrak{q}(t,x):=q_{\{2n+1=3\}}(t,x),\quad\text{associated to }\quad\mathcal{N}_3,
\end{equation*}
as defined above. Similarly we define
\begin{equation}\label{Q3}
\mathfrak{Q}(t,x):=x+\mathfrak q(t,x)\qquad\text{and}\qquad\widetilde{\mathfrak{Q}}(t,y):=y+\widetilde{\mathfrak q} (t,y),
\end{equation}
the analogues of $Q$ and $\widetilde Q$ as defined in \eqref{changeV}, for the trilinear setting.	

From \eqref{nonlin8.1}, we see that the trilinear pseudoproducts we need to understand are of the form

\begin{equation}\label{PT}
\begin{aligned}
\mathcal{T}_3(&P_{k_1}f_1,P_{k_2}f_2,P_{k_3}f_3)(x)\\
&=-\frac{1}{\pi}\frac{d}{dx}\int_{\R^3}\frac{d}{dy_1}(P_{k_1}f_1)(y_1)(P_{k_2}f_2)(y_2)(P_{k_3}f_3)(y_3)K_{k_1,k_2,k_3}(x-y_1,x-y_2,x-y_3)\,dy_1dy_2dy_3.
\end{aligned}
\end{equation}	
	
	\subsection{Localized $L^2$ estimates}

	The following lemma provides bounds for trilinear pseudoproducts and will prove to be of fundamental importance throughout the rest our analysis. 
	\begin{lemma}\label{lemma:TP}
		Let $\mathcal{T}_3(P_{k_1}f_1,P_{k_2}f_2,P_{k_3}f_3)$ be defined as in \eqref{PT}. Then we have
		\begin{equation*}\label{trilinear1}
\|P_k\big[\mathcal{T}_3(P_{k_1}f_1,P_{k_2}f_2,P_{k_3}f_3)\big]\|_{L^2}\lesssim2^{k}\|P_{k_1}f_1\|_{L^2}\|P_{k_2}f_2\|_{L^\infty}\|P_{k_3}f_3\|_{L^\infty}.
		\end{equation*}
		Alternatively, we also have
		\begin{equation*}\label{trilinear2}
		\|P_k\big[\mathcal{T}_3(P_{k_1}f_1,P_{k_2}f_2,P_{k_3}f_3)\big]\|_{L^2}\lesssim2^{k}2^{k/2}\|P_{k_1}f_1\|_{L^2}\|P_{k_2}f_2\|_{L^2}\|P_{k_3}f_3\|_{L^\infty}.
		\end{equation*}
		
	\end{lemma}
	\begin{proof}
		Without loss of generality we may assume that $k_1\geq k_2\geq k_3$. We introduce \eqref{Lktilde} to split
		\begin{align*}
		K_{k_1,k_2,k_3}(z_1,z_2,z_3)=K_{k_1,k_2,k_3}^1(z_1,z_2,z_3)+K_{k_1,k_2,k_3}^2(z_1,z_2,z_3)+K_{k_1,k_2,k_3}^3(z_1,z_2,z_3),
		\end{align*}
		where
		\begin{equation*}
		K_{k_1,k_2,k_3}^1(z_1,z_2,z_3):=\int_{\R}L_{k_1}(z_1,\alpha)\widetilde{L}_{k_2}(z_2,\alpha)L_{k_3}(z_3,\alpha)\,d\alpha,
		\end{equation*}
		\begin{equation*}
		K_{k_1,k_2,k_3}^2:=K_{k_1,k_2,k_3}^{2a}+K_{k_1,k_2,k_3}^{2b}+K_{k_1,k_2,k_3}^{2c},
		\end{equation*}
		with
		\begin{align*}
		K_{k_1,k_2,k_3}^{2a}(z_1,z_2,z_3)&:=	\int_{\R}L_{k_1}(z_1,\alpha)\frac{1}{\alpha}\min(1,2^{k_2}\alpha)\psi_{0}'(2^{k_2}z_2)\widetilde{L}_{k_3}(z_3,\alpha)\,d\alpha, \\
		K_{k_1,k_2,k_3}^{2b}(z_1,z_2,z_3)&:=\int_{|\alpha|\leq2^{-k_1}}L_{k_1}(z_1,\alpha)\frac{1}{\alpha}\min(1,2^{k_2}\alpha)\psi_{0}'(2^{k_2}z_2)\frac{1}{\alpha}\min(1,2^{k_3}\alpha)\psi_{0}'(2^{k_3}z_3)\,d\alpha,\\
		K_{k_1,k_2,k_3}^{2c}(z_1,z_2,z_3)&:=\int_{|\alpha|\geq2^{-k_2}}L_{k_1}(z_1,\alpha)\frac{1}{\alpha}\min(1,2^{k_2}\alpha)\psi_{0}'(2^{k_2}z_2)\frac{1}{\alpha}\min(1,2^{k_3}\alpha)\psi_{0}'(2^{k_3}z_3)\,d\alpha,
		\end{align*}	
		and
		\begin{equation*}
		K_{k_1,k_2,k_3}^{3}(z_1,z_2,z_3):=2^{k_2+k_3}\psi_{0}'(2^{k_2}z_2)\psi_{0}'(2^{k_3}z_3)\int_{|\alpha|\in [2^{-k_1},2^{-k_2}]}L_{k_1}(z_1,\alpha)\,d\alpha.
		\end{equation*}
		
		Using Lemmas~\ref{lemma:K} and~\ref{lemma:K2} we get
		\begin{equation}\label{musk10}
		\begin{split}
		&\|K_{k_1,k_2,k_3}^1\|_{L^1(\R^3)}\lesssim2^{-k_1},\\
		&\|K_{k_1,k_2,k_3}^{2a}\|_{L^1(\R^3)}\lesssim2^{-k_1},\qquad
		\|K_{k_1,k_2,k_3}^{2b}\|_{L^1(\R^3)}\lesssim2^{-k_1},\qquad
		\|K_{k_1,k_2,k_3}^{2c}\|_{L^1(\R^3)}\lesssim2^{-k_1}.
		\end{split}
		\end{equation}
		By expressing $L_{k_1}$ using \eqref{Lk_formula} we obtain
		\begin{align}\label{eq.4.20}
		\int_{|\alpha|\in [2^{-k_1},2^{-k_2}]}\frac{[\psi_{0}(2^{k_1}x)-\psi_{0}(2^{k_1}(x-\alpha))]}{\alpha}\,d\alpha=\int_{|\alpha|\in [2^{-k_1},2^{-k_2}]}\frac{[-\psi_{0}(2^{k_1}(x-\alpha))]}{\alpha}\,d\alpha,
		\end{align}
		where the cancellation is a result of $1/\alpha$ being odd. Reinserting \eqref{eq.4.20} into $K^3$ yields
		\begin{align*}
		K_{k_1,k_2,k_3}^3=2^{k_2+k_3}\psi_{0}'(2^{k_2}z_2)\psi_{0}'(2^{k_3}z_3)\int_{|\alpha|\in [2^{-k_1},2^{-k_2}]}\frac{[-\psi_{0}(2^{k_1}(x-\alpha))]}{\alpha}\,d\alpha.
		\end{align*}
		
		For $l\in\{1,2,3\}$ we define
		\begin{align*}
		\mathcal{T}_3^l&(P_{k_1}f_1,P_{k_2}f_2,P_{k_3}f_3)(x)\\
		&:=-\frac{1}{\pi}\frac{d}{dx}\int_{\R^3}\frac{d}{dy_1}(P_{k_1}f_1)(y_1)(P_{k_2}f_2)(y_2)(P_{k_3}f_3)(y_3)K_{k_1,k_2,k_3}^l(x-y_1,x-y_2,x-y_3)\,dy_1dy_2dy_3.
		\end{align*}
		In view of \eqref{musk10}, for $l\in\{1,2\}$ we have
		\begin{equation}\label{musk11}
		\begin{split}
		&\|P_k\big[\mathcal{T}_3^l(P_{k_1}f_1,P_{k_2}f_2,P_{k_3}f_3)\big]\|_{L^2}\lesssim2^{k}\|P_{k_1}f_1\|_{L^2}\|P_{k_2}f_2\|_{L^\infty}\|P_{k_3}f_3\|_{L^\infty},\\
		&\|P_k\big[\mathcal{T}_3^l(P_{k_1}f_1,P_{k_2}f_2,P_{k_3}f_3)\big]\|_{L^2}\lesssim2^{k}2^{k/2}\|P_{k_1}f_1\|_{L^2}\|P_{k_2}f_2\|_{L^2}\|P_{k_3}f_3\|_{L^\infty}.
		\end{split}
		\end{equation}
		We would like to prove similar bounds for $l=3$. By rearranging all the terms, we obtain
		\begin{align}\label{T5}
		\mathcal{T}_3^3(P_{k_1}h_1,P_{k_2}h_2,P_{k_3}h_3)(x)=-\frac{1}{\pi}\frac{d}{dx}\Big\{P_{k_2}h_2(x)P_{k_3}h_3(x)\!\!\int_{|\alpha|\in[2^{-k_1},2^{-k_2}]}\frac{-P_{k_1}h_1(x-\alpha)}{\alpha}\,d\alpha\Big\},
		\end{align}
		where we have used that by definition $\widehat{\partial_x\psi_0}=\widetilde{\varphi}_0$, hence
		\begin{equation*}\label{aux2}
		P_{k_j}f_j(x)=\int_{\R}2^{k_j}(P_{k_j}f_j)(x-y_j)\psi_{0}'(2^{k_j}y_j)\,dy_j,\qquad j\in\{1,2,3\}.
		\end{equation*}
		From the boundedness of the truncated Hilbert transform we have
		\begin{equation}\label{musk12}
		\Big\|\int_{|\alpha|\in[2^{-k_1},2^{-k_2}]}\frac{-P_{k_1}h_1(x-\alpha)}{\alpha}\,d\alpha\Big\|_{L^2_x}\lesssim \|P_{k_1}h_1\|_{L^2},
		\end{equation}
		uniformly in $k_1$ and $k_2$. The bounds \eqref{musk11} for $l=3$ follow from \eqref{T5} and \eqref{musk12}. This completes the proof of the lemma.
	\end{proof}

	\subsection{Trilinear estimates in the space $Z_2$} 
	
	For any triple $(k_1,k_2,k_3)\in\Z^3$ we let $(k_1^\ast,k_2^\ast,k_3^\ast)\in\Z^3$ denote its non-increasing rearrangement, $k_1^\ast\geq k_2^\ast\geq k_3^\ast$. Given $k\in\Z$ we define two sets:
	\begin{align*}
	S_{k,1}&:=\{(k_1,k_2,k_3)\in\Z^3:\,k^\ast_1\in[k-3,k+3]\text{ and }k^\ast_2,k^\ast_3\leq k-6\},\\
	S_{k,2}&:=\{(k_1,k_2,k_3)\in\Z^3:\,|k^\ast_1-k^\ast_2|\leq10 \text{ where } k_1^\ast\geq k-3 \text{ and }k_2^\ast\geq k-5 \}.
	\end{align*}	
	We localize and decompose the nonlinearities $\mathcal{N}_3$; more generally we decompose
	\begin{equation*}\label{G1G2}
	P_k\big[\mathcal{T}_3(h_{1},h_2,h_3)\big](x)=G_{k,1}(x)+G_{k,2}(x),
	\end{equation*}
	where, for $l\in\{1,2\}$,
	\begin{equation*}\label{musk15}
	\begin{split}
	&G_{k,l}[h_1,h_2,h_3](x):=-\frac{1}{\pi}\sum_{k_1,k_2,k_3\in S_{k,l}}\\
	&\qquad P_k\frac{d}{dx}\int_{\R^3}\frac{d}{dy_1}P_{k_1}h_1(y_1)P_{k_2}h_2(y_2)P_{k_3}h_3(y_3)K_{k_1,k_2,k_3}(x-y_1,x-y_2,x-y_3)\,dy_1dy_2dy_3.
	\end{split}
	\end{equation*}
	
	\subsubsection{High-high-to-low estimates} We estimate first the term $G_{k,2}$, for all input functions in the space $Z$.
	
	\begin{lemma}\label{lemmaG2}
		
		Assume that $f_{1},f_{2}, f_{3}\in Z$, and define $h_j(t,x):=f_j(t,x-a_j+\mathfrak q(t,x))$ as in \eqref{ha_ga}, for some points $a_j\in\R$. Then for any $k\in\Z$ and $t\in[0,\infty)$ we have
		\begin{equation}\label{sum00}
		\|G_{k,2}[h_1,h_2,h_3](t)\|_{L^2}\lesssim 2^{k/2}(1+2^k t)^{-2/10}\|f_{1}\|_Z\|f_{2}\|_Z\|f_{3}\|_{Z}.
		\end{equation}
	\end{lemma}
	\begin{proof} We write $S_{k,2}$ as a disjoint union $S_{k,2}=S_{k,2}^1\cup S_{k,2}^2\cup S_{k,2}^3$ where
		\begin{equation*}\label{musk20}
		\begin{split}
		&S_{k,2}^1:=\{(k_1,k_2,k_3)\in S_{k,2}:\, k_1\leq\min(k_2-1,k_3-1)\},\\		
  &S_{k,2}^2:=\{(k_1,k_2,k_3)\in S_{k,2}:\, k_2\leq\min(k_1,k_3-1)\},\\
  &S_{k,2}^3:=\{(k_1,k_2,k_3)\in S_{k,2}:\, k_3\leq\min(k_1,k_2)\}.
		\end{split}
		\end{equation*}
		Then we decompose the functions accordingly, $G_{k,2}=G_{k,2}^1+G_{k,2}^2+G_{k,2}^3$, where
		\begin{equation*}\label{musk21}
		\begin{split}
		&G_{k,2}^a[h_1,h_2,h_3](x):=-\frac{1}{\pi}\sum_{k_1,k_2,k_3\in S_{k,2}^a}\\
		&\qquad P_k\frac{d}{dx}\int_{\R^3}\frac{d}{dy_1}P_{k_1}h_1(y_1)P_{k_2}h_2(y_2)P_{k_3}h_3(y_3)K_{k_1,k_2,k_3}(x-y_1,x-y_2,x-y_3)\,dy_1dy_2dy_3,
		\end{split}
		\end{equation*}
		for $a\in\{1,2,3\}$.
		
		We will only prove the estimates \eqref{sum00} for the functions $G_{k,2}^3$; the estimates for the functions $G_{k,2}^2$ are similar, while the estimates for $G_{k,2}^1$ are easier because the derivative hits the low frequency factor. 
		
		Assume without loss of generality that $\|f_{j}\|_{Z}=1$. As in the proof of Lemma \ref{lemma:TP} and using \eqref{musk10.5} and Lemma \ref{lem:g_h_Z}, we estimate
		\begin{equation*}
		\begin{aligned}
		\|G^3_{k,2}(t)\|_{L^2}&\lesssim \sum_{\substack{|k_1-k_2|\leq10,\,k_1,k_2\geq k-5}}2^{k}2^{k/2}2^{k_1}\|P_{k_1}h_1(t)\|_{L^2}\|P_{k_2}h_2(t)\|_{L^2}\|h_3(t)\|_{L^\infty}\|K_{k_1,k_2,\leq\min(k_1,k_2)}\|_{L^1}\\
		&\lesssim \sum_{k_1\geq k-5}2^{k}2^{k/2}2^{-k_1}\min\{1,(2^{k_1}t)^{-2/10}\}\\
		&\lesssim 2^{k/2}\min\{1,(2^kt)^{-2/10}\},
		\end{aligned}
		\end{equation*}
		as claimed.
	\end{proof}
	
	\subsubsection{Low-high-to-high estimates} We estimate now most of the components of the trilinear expressions $G_{k,1}$. We write $S_{k,1}$ as a disjoint union $S_{k,1}=S_{k,1}^1\cup S_{k,1}^2\cup S_{k,1}^3$ where
	\begin{equation*}\label{musk30}
	S_{k,1}^l:=\{(k_1,k_2,k_3)\in S_{k,1}:\, k_l=\max(k_1,k_2,k_3)\},\qquad l\in\{1,2,3\}.
	\end{equation*}
	Then we decompose the functions accordingly, $G_{k,1}=G_{k,1}^1+G_{k,1}^2+G_{k,1}^3$, where, for $a\in\{1,2,3\}$,
	\begin{equation}\label{musk31}
	\begin{split}
	&G_{k,1}^a[h_1,h_2,h_3](x):=-\frac{1}{\pi}\sum_{k_1,k_2,k_3\in S_{k,1}^a}\\
	&\qquad P_k\frac{d}{dx}\int_{\R^3}\frac{d}{dy_1}P_{k_1}h_1(y_1)P_{k_2}h_2(y_2)P_{k_3}h_3(y_3)K_{k_1,k_2,k_3}(x-y_1,x-y_2,x-y_3)\,dy_1dy_2dy_3.
	\end{split}
	\end{equation}
	
	We show first how to bound the trilinear expressions $G_{k,1}^2$ and $G_{k,1}^3$.
	
	\begin{lemma}\label{lemmaG12G13}
		Assume that $f_1,f_2,f_3\in Z$ and define $h_j(t,x):=f_j(t,x-a_j+\mathfrak q(t,x))$ as in \eqref{ha_ga} for some points $a_j\in\R$. Then, for $a\in\{2,3\}$, $k\in\Z$ and $t\in[0,\infty)$ we have
		\begin{equation*}
		\|G_{k,1}^a(t)\|_{L^2}\lesssim 2^{k/2}(1+2^k t)^{-2/10}\|f_{1}\|_Z\|f_{2}\|_Z\|f_{3}\|_{Z}.	
		\end{equation*}
	\end{lemma}
	
	\begin{proof} The two cases are similar, so we will assume that $a=2$. We may also assume that $\|f_{j}\|_{Z}=1$. As in the proof of Lemma \ref{lemma:TP} and using \eqref{musk10.5} and Lemma \ref{lem:g_h_Z}, we estimate
		\begin{equation*}
		\begin{aligned}
		\|G^2_{k,1}(t)\|_{L^2}&\lesssim \sum_{\substack{|k_2-k|\leq 3,\,k_1\leq k-6}}2^{k}2^{k_1}\|P_{k_1}h_1(t)\|_{L^\infty}\|P_{k_2}h_2(t)\|_{L^2}\|h_3(t)\|_{L^\infty}\|K_{k_1,k_2,\leq k-6}\|_{L^1}\\
		&\lesssim \sum_{k_1\leq k-6}2^{k}2^{k_1}\min\{1,(2^{k_1}t)^{-1/10}\}2^{-k/2}\min\{1,(2^{k}t)^{-1/10}\}2^{-k}|k-k_1|\\
		&\lesssim \sum_{k_1\leq k-6}2^{k/2}2^{k_1-k}|k-k_1|\min\{1,(2^{k_1}t)^{-1/10}\}\min\{1,(2^{k}t)^{-1/10}\}\\
		&\lesssim 2^{k/2}\min\{1,(2^{k}t)^{-1/10}\}^2,
		\end{aligned}
		\end{equation*}
		as claimed.
	\end{proof}
	
	We estimate now the trilinear expression $G_{k,1}^1$ when one of the low frequency inputs is in the space $Z_2$.
	
	\begin{lemma}\label{lemmaG1Z2}
		Assume that $(f_1,f_2,f_3)\in Z\times Z\times Z_2$ or $(f_1,f_2,f_3)\in Z\times Z_2\times Z$ and define $h_j(t,x)=f_j(t,x-a_j+\mathfrak q(t,x))$ as before. Then, for any $k\in\Z$ and $t\in[0,\infty)$,
		\begin{equation*}
		\begin{split}
		&\|G_{k,1}^1\|_{L^2}\lesssim2^{k/2}\min\{1,(2^kt)^{-1/10}\}\|f_{1}\|_{Z}\|f_{2}\|_{Z}\|f_{3}\|_{Z_2},\\
		&\|G_{k,1}^1\|_{L^2}\lesssim2^{k/2}\min\{1,(2^kt)^{-1/10}\}\|f_{1}\|_{Z}\|f_{2}\|_{Z_2}\|f_{3}\|_{Z}.
		\end{split}
		\end{equation*}
	\end{lemma}
	
	\begin{proof}
		The two bounds are similar, so we will only prove the bounds in the second line, corresponding to $(f_1,f_2,f_3)\in Z\times Z_2\times Z$. We have two cases: if $f_3\in Z_2$ then we may assume that $\|f_1\|_{Z}=\|f_2\|_{Z_2}=\|f_3\|_{Z_2}=1$ and use Lemmas \ref{lemma:TP} and \ref{lem:g_h_Z} to estimate
		\begin{align*}
		\|G_{k,1}^1(t)\|_{L^2}&\lesssim\sum_{\substack{k_1\in[k-3,k+3],\,k_2,k_3\leq k-6}}2^{k}\|P_{k_1}h_1(t)\|_{L^2}\|P_{k_2}h_2(t)\|_{L^\infty}\|P_{k_3}h_3(t)\|_{L^\infty}\\
		&\lesssim 2^k 2^{-k/2}\min\{1,(2^kt)^{-1/10})\}\Big(\sum_{k_2\leq k-6}\min\{(2^{k_2}t)^{1/10},(2^{k_2}t)^{-1/10})\}\Big)^2\\
		&\lesssim 2^{k/2}\min\{(2^{k}t)^{2/10},(2^kt)^{-1/10}\}.
		\end{align*}
		as claimed.
		On the other hand, if $f_3\in Z_1$ then we may assume that $\|f_1\|_{Z}=\|f_2\|_{Z_2}=\|f_3\|_{Z_1}=1$. We examine the formula \eqref{musk31} and decompose $G_{k,1}^1=G_{k,1}^{1,a}+G_{k,1}^{1,b}$ where
		\begin{equation*}\label{musk35}
		\begin{split}
		&G_{k,1}^{1,a}[h_1,h_2,h_3](x):=-\frac{1}{\pi}\sum_{|k_1-k|\leq 3,\,k_2\leq k-6,\,k_3\in[k_2-3,k-6]}\\
		&\qquad P_k\frac{d}{dx}\int_{\R^3}\frac{d}{dy_1}P_{k_1}h_1(y_1)P_{k_2}h_2(y_2)P_{k_3}h_3(y_3)K_{k_1,k_2,k_3}(x-y_1,x-y_2,x-y_3)\,dy_1dy_2dy_3,
		\end{split}
		\end{equation*}
		\begin{equation*}\label{musk36}
		\begin{split}
		&G_{k,1}^{1,b}[h_1,h_2,h_3](x):=-\frac{1}{\pi}\sum_{|k_1-k|\leq 3,\,k_2\leq k-6}\\
		&\qquad P_k\frac{d}{dx}\int_{\R^3}\frac{d}{dy_1}P_{k_1}h_1(y_1)P_{k_2}h_2(y_2)P_{\leq k_2-4}h_3(y_3)K_{k_1,k_2,\leq k_2-4}(x-y_1,x-y_2,x-y_3)\,dy_1dy_2dy_3.
		\end{split}
		\end{equation*}
		
		The function $G_{k,1}^{1,a}$ can be estimated using just Lemmas \ref{lemma:TP} and \ref{lem:g_h_Z},
		\begin{equation}\label{musk37}
		\begin{split}
		\|G_{k,1}^{1,a}(t)\|_{L^2}&\lesssim 2^{k}\sum_{k_1\in[k-3,k+3]}\sum_{k_2\leq k-6, k_3\in [k_2-3,k-6]}\|P_{k_1}h_1(t)\|_{L^2}\|P_{k_2}h_2(t)\|_{L^\infty}\|P_{k_3}h_3(t)\|_{L^\infty}\\
		&\lesssim \sum_{k_2\leq k-6, k_3\in [k_2-3,k-6]} 2^{k/2}(1+2^{k}t)^{-1/10}(1+2^{k_2}t)^{-2/10}(2^{k_2}t)^{1/10}(1+2^{k_3}t)^{-1/10}\\
		&\lesssim2^{k/2}\min\{1,(2^kt)^{-1/10}\},
		\end{split}
		\end{equation}
		as desired.
		To estimate $G_{k,1}^{1,b}$, we further decompose $G_{k,1}^{1,b}=G_{k,1}^{1,c}+G_{k,1}^{1,d}$ where
		\begin{equation*}\label{musk38}
		\begin{split}
		&G_{k,1}^{1,c}[h_1,h_2,h_3](x):=-\frac{1}{\pi}\sum_{|k_1-k|\leq 3,\,k_2\leq k-6}\\
		&\qquad P_k\frac{d}{dx}\int_{\R^3}\frac{d}{dy_1}P_{k_1}h_1(y_1)P_{k_2}h_2(y_2)P_{\leq k_2-4}h_3(y_3)\widetilde{K}_{k_1,k_2,\leq k_2-4}(x-y_1,x-y_2,x-y_3)\,dy_1dy_2dy_3,
		\end{split}
		\end{equation*}
		\begin{equation}\label{musk39}
		\begin{split}
		G_{k,1}^{1,d}[h_1,h_2,h_3](x):=&-\frac{1}{\pi}\sum_{|k_1-k|\leq 3,\,k_2\leq k-6}P_k\frac{d}{dx}\int_{\R^3}\frac{d}{dy_1}P_{k_1}h_1(y_1)P_{k_2}h_2(y_2)P_{\leq k_2-4}h_3(y_3)\\
		&\times(K_{k_1,k_2,\leq k_2-4}-\widetilde{K}_{k_1,k_2,\leq k_2-4})(x-y_1,x-y_2,x-y_3)\,dy_1dy_2dy_3.
		\end{split}
		\end{equation}
		The kernels $\widetilde{K}$ are defined in \eqref{kernel6}. 
		Using \eqref{musk10.6} and Lemma \ref{lem:g_h_Z} we estimate
		\begin{equation}\label{musk37.1}
		\begin{split}
		\|G_{k,1}^{1,c}(t)\|_{L^2}&\lesssim 2^{k}\sum_{k_1\in[k-3,k+3]}\sum_{k_2\leq k-6}2^{k_1}\|P_{k_1}h_1(t)\|_{L^2}\|P_{k_2}h_2(t)\|_{L^\infty}\|P_{\leq k_2-4}h_3(t)\|_{L^\infty}2^{-k_1}\\
		&\lesssim \sum_{k_2\leq k-6} 2^{k/2}(1+2^{k}t)^{-1/10}(1+2^{k_2}t)^{-2/10}(2^{k_2}t)^{1/10}\\
		&\lesssim2^{k/2}\min\{1,(2^kt)^{-1/10}\}.
		\end{split}
		\end{equation}
		To bound $G_{k,1}^{1,d}$ we write
		\begin{equation*}
		\begin{split}
		(K_{k_1,k_2,\leq k_2-4}-\widetilde{K}_{k_1,k_2,\leq k_2-4})(y_1,y_2,y_3)&=\int_{\R}L_{k_1}(y_1,\alpha)\frac{\min(1,2^{k_2}\alpha)}{\alpha}\psi'_0(2^{k_2}y_2)L_{\leq k_2-4}(y_3,\alpha)\,d\alpha\\
		&=J^1_{k_1,k_2}(y_1,y_2,y_3)+J^2_{k_1,k_2}(y_1,y_2,y_3)+J^3_{k_1,k_2}(y_1,y_2,y_3),
		\end{split}
		\end{equation*}
		where, using also the decomposition \eqref{modified_kernel},
		\begin{equation*}
		\begin{split}
		&J^1_{k_1,k_2}(y_1,y_2,y_3):=\int_{|\alpha|\notin[2^{-k_1},2^{-k_2}]}L_{k_1}(y_1,\alpha)\frac{\min(1,2^{k_2}\alpha)}{\alpha}\psi'_0(2^{k_2}y_2)L_{\leq k_2-4}(y_3,\alpha)\,d\alpha,\\
		&J^2_{k_1,k_2}(y_1,y_2,y_3):=\int_{|\alpha|\in[2^{-k_1},2^{-k_2}]}L_{k_1}(y_1,\alpha)\frac{\min(1,2^{k_2}\alpha)}{\alpha}\psi'_0(2^{k_2}y_2)\widetilde{L}_{\leq k_2-4}(y_3,\alpha)\,d\alpha,\\
		&J^3_{k_1,k_2}(y_1,y_2,y_3):=\int_{|\alpha|\in[2^{-k_1},2^{-k_2}]}L_{k_1}(y_1,\alpha)\frac{\min(1,2^{k_2}\alpha)}{\alpha}\psi'_0(2^{k_2}y_2)\frac{\min(1,2^{k_2-4}\alpha)}{\alpha}\psi_{\leq 0}'(2^{k_2-4}y_3)\,d\alpha.
		\end{split}
		\end{equation*}
		Then we define the operators $G_{k,1}^{1,e}$, $G_{k.1}^{1,f}$, and $G_{k,1}^{1,g}$ as in \eqref{musk39}, by replacing the kernel $(K_{k_1,k_2,\leq k_2-4}-\widetilde{K}_{k_1,k_2,\leq k_2-4})$ with the kernels $J^1_{k_1,k_2}$, $J^2_{k_1,k_2}$, and $J^3_{k_1,k_2}$ respectively. In view of Lemmas \ref{lemma:K}  and \ref{lemma:K2} we have
		\begin{equation*}
		\|J^1_{k_1,k_2}\|_{L^1(\R^3)}+\|J^2_{k_1,k_2}\|_{L^1(\R^3)}\lesssim 2^{-k_1},
		\end{equation*}
		so one can estimate as in \eqref{musk37.1} 
		\begin{equation}\label{musk37.2}
		\|G_{k,1}^{1,e}(t)\|_{L^2}+\|G_{k,1}^{1,f}(t)\|_{L^2}\lesssim2^{k/2}\min\{1,(2^kt)^{-1/10}\}.
		\end{equation}
		Finally, to bound the functions $G_{k,1}^{1,g}$ we notice that
		\begin{equation*}\label{simple_id}
		\begin{split}
		&\int_{\R}P_{k_2}h_2(x-y_2)2^{k_2}\psi'_0(2^{k_2}y_2)\,dy_2=P_{k_2}h_2(x),\\
		&\int_{\R}P_{\leq k_2-4}h_3(x-y_3)2^{k_2-4}\psi'_{\leq 0}(2^{k_2-4}y_2)\,dy_2=P_{\leq k_2-4}h_3(x).
		\end{split}
		\end{equation*}
		Moreover, using the formula \eqref{Lk_formula} and integration by parts in $y_1$, 
		\begin{equation}\label{simple_id2}
		\int_{\R}\frac{d}{dy_1}P_{k_1}h_1(y_1)L_{k_1}(x-y_1,\alpha)\,dy_1=\frac{P_{k_1}h_1(x)-P_{k_1}h_1(x-\alpha)}{\alpha}.
		\end{equation}
		Therefore
		\begin{equation*}
		\begin{split}
		G_{k,1}^{1,g}[h_1,h_2,h_3](x):=&-\frac{1}{\pi}\sum_{|k_1-k|\leq 3,\,k_2\leq k-6}P_k\frac{d}{dx}\Big\{P_{k_2}h_2(x)P_{\leq k_2-4}h_3(x)\\
		&\times\int_{|\alpha|\in[2^{-k_1},2^{-k_2}]}\frac{P_{k_1}h_1(x)-P_{k_1}h_1(x-\alpha)}{\alpha}\,d\alpha\Big\}.
		\end{split}
		\end{equation*}
		In view of the boundedness of the truncated Hilbert transform we can now estimate 
		\begin{equation}\label{musk37.3}
		\begin{split}
		\|G_{k,1}^{1,g}(t)\|_{L^2}&\lesssim 2^{k}\sum_{k_1\in[k-3,k+3]}\sum_{k_2\leq k-6}\|P_{k_1}h_1(t)\|_{L^2}\|P_{k_2}h_2(t)\|_{L^\infty}\|P_{\leq k_2-4}h_3(t)\|_{L^\infty}\\
		&\lesssim \sum_{k_2\leq k-6} 2^{k/2}(1+2^{k}t)^{-1/10}(1+2^{k_2}t)^{-2/10}(2^{k_2}t)^{1/10}\\
		&\lesssim2^{k/2}\min\{1,(2^kt)^{-1/10}\}.
		\end{split}
		\end{equation}
		The conclusion of the lemma follows from \eqref{musk37}, \eqref{musk37.1}, \eqref{musk37.2}, and \eqref{musk37.3}. 
	\end{proof}

	\subsubsection{Estimating the high-low-to-high term for $Z_1$ inputs}
	
	It remains to understand the term $G_{k,1}^1$
	\begin{equation*}\label{musk44}
	\begin{aligned}
	G_{k,1}^1[h_1,h_2,h_3](x)&=-\frac{1}{\pi}\sum_{k_1\in[k-3,k+3]}P_k \frac{d}{dx}\int_{\R^4}\frac{d}{dx}P_{k_1}h_{1}(x-y_1)P_{\leq k-6}h_2(x-y_2)P_{\leq k-6}h_3(x-y_3)\\
	&\hspace{4cm}\times L_{k_1}(y_1,\alpha)L_{\leq k-6}(y_2,\alpha)L_{\leq k-6}(y_3,\alpha)d\alpha dy_1dy_2dy_3.
	\end{aligned}
	\end{equation*}
	for inputs $f_2$ and $f_3$ in $Z_1$. Ideally we would like to prove similar bounds as before
	\begin{equation*}\label{musk45}
	\|G_{k,1}^1(t)\|_{L^2}\lesssim2^{k/2}\min\{1,(2^kt)^{-1/10}\}\|f_{1}\|_{Z}\|f_{2}\|_{Z_1}\|f_{3}\|_{Z_1}.
	\end{equation*}
	Unfortunately this is not possible due to a logarithmic loss. In the next three lemmas we prove these bounds for certain parts of the trilinear operators $G_{k,1}^1$.

	\begin{lemma}\label{lem:G1less} 
		Assume that $f_{1}\in Z$, $f_{2}, f_{3}\in Z_1$, and define $h_j(t,x):=f_j(t,x-a_j+\mathfrak q(t,x))$ as before, for some points $a_j\in\R$. Define
		\begin{equation*}
		\begin{aligned}
		G_{k,1}^{1,1}[h_1,h_2,h_3](x)&:=-\frac{1}{\pi}\sum_{k_1\in[k-3,k+3]}P_k \frac{d}{dx}\int_{|\alpha|\leq 2^{-k}}\int_{\R^3}\frac{d}{dx}P_{k_1}h_{1}(x-y_1)P_{\leq k-6}h_2(x-y_2)P_{\leq k-6}h_3(x-y_3)\\
		&\hspace{4cm}\times L_{k_1}(y_1,\alpha)L_{\leq k-6}(y_2,\alpha)L_{\leq k-6}(y_3,\alpha)d\alpha dy_1dy_2dy_3.
		\end{aligned}
		\end{equation*}
		Then
		\begin{equation*}
		\|G_{k,1}^{1,1}(t)\|_{L^2}\lesssim 2^{k/2}\min\{1,(2^kt)^{-1/10}\}\|f_{1}\|_{Z}\|f_{2}\|_{L^\infty}\|f_{3}\|_{L^\infty}.
		\end{equation*}
	\end{lemma}
	
	\begin{proof} This follows directly from Lemmas \ref{lemma:K} and \ref{lem:g_h_Z}:
		\begin{equation*}
		\begin{split}
		\|G_{k,1}^{1,1}(t)\|_{L^2}&\lesssim 2^{k}\sum_{k_1\in[k-3,k+3]}2^{k_1}\|P_{k_1}h_1(t)\|_{L^2}\|P_{\leq k-6}h_2(t)\|_{L^\infty}\|P_{\leq k-6}h_3(t)\|_{L^\infty}2^{-k_1}\\
		&\lesssim 2^{k/2}(1+2^{k}t)^{-1/10}\|f_{1}\|_{Z}\|f_{2}\|_{L^\infty}\|f_{3}\|_{L^\infty},	
		\end{split}
		\end{equation*}
		as claimed.
	\end{proof}

	Next, we proceed with the integral over $2^{k}|\alpha|\geq 1$. We recall the formulas \eqref{Lk_formula} (which we use for the kernel $L_{k_1}$) and \eqref{simple_id2} and decompose $G_{k,1}^1-G_{k,1}^{1,1}=G_{k,1}^{1,2}+G_{k,1}^{1,3}+G_{k,1}^{1,4}$ where
	\begin{equation}\label{G11}
	\begin{split}
	G_{k,1}^{1,2}[h_1,h_2,h_3](x)&:=-\frac{1}{\pi}\sum_{k_1\in[k-3,k+3]}P_k\int_{|\alpha|\geq 2^{-k}}\int_{\R^3}\frac{d}{dx}P_{k_1}h_{1}(x-y_1) \frac{d}{dx}\big\{P_{\leq k-6}h_2(x-y_2)P_{\leq k-6}h_3(x-y_3)\big\}\\
	&\hspace{4cm}\times L_{k_1}(y_1,\alpha)L_{\leq k-6}(y_2,\alpha)L_{\leq k-6}(y_3,\alpha)d\alpha dy_1dy_2dy_3.
	\end{split}
	\end{equation}
	\begin{equation}\label{G12}
	\begin{aligned}
	G_{k,1}^{1,3}[h_1,h_2,h_3]:=\frac{1}{\pi}\sum_{k_1\in[k-3,k+3]}P_{k}&\int_{|\alpha|\geq 2^{-k}}\int_{\R^2}\frac{d}{dx}P_{k_1}h_1(x-\alpha)P_{\leq k-6}h_2(x-y_2)P_{\leq k-6}h_3(x-y_3)\\
	&\times\frac{L_{\leq k-6}(y_2,\alpha)L_{\leq k-6}(y_3,\alpha)}{\alpha}d\alpha dy_2 dy_3.
	\end{aligned}
	\end{equation}
	\begin{equation}\label{G13}
	\begin{aligned}
	G_{k,1}^{1,4}[h_1,h_2,h_3]:=-\frac{1}{\pi}\sum_{k_1\in[k-3,k+3]}P_{k}\Big\{\frac{d}{dx}P_{k_1}h_1(x)&\int_{|\alpha|\geq 2^{-k}}\int_{\R^2}P_{\leq k-6}h_2(x-y_2)P_{\leq k-6}h_3(x-y_3)\\
	&\times\frac{L_{\leq k-6}(y_2,\alpha)L_{\leq k-6}(y_3,\alpha)}{\alpha}d\alpha dy_2 dy_3\Big\}.
	\end{aligned}
	\end{equation}
	We remark that we used integration by parts for the term in \eqref{G12}, as was done in \eqref{simple_id2}.	We prove now suitable bounds on the functions $G_{k,1}^{1,2}$.

	\begin{lemma}\label{lem:G1greq1}
		Assume that $f_{1}\in Z$, $f_{2}, f_{3}\in Z_1$, and define $h_j(t,x):=f_j(t,x-a_j+\mathfrak q(t,x))$ as before, for some points $a_j\in\R$. 
		Then, with $G_{k,1}^{1,2}$ as in \eqref{G11},
		\begin{equation*}\|G_{k,1}^{1,2}(t)\|_{L^2}\lesssim 2^{k/2}\min\{1,(2^kt)^{-1/10}\}\|f_{1}\|_{Z}\|f_{2}\|_{L^\infty}\|f_{3}\|_{L^\infty}.
		\end{equation*}
	\end{lemma}
	
	\begin{proof}
		The $d/dx$ derivative can hit either the function $h_2$ or $h_3$. The two cases are identical. Using Lemma \ref{lemma:K} (ii) and \ref{lem:g_h_Z} we estimate
		\begin{equation*}
		\begin{split}
		\|G_{k,1}^{1,2}(t)\|_{L^2}&\lesssim \sum_{k_1\in[k-3,k+3]}\sum_{k_2\leq k-6}2^{k_1}\|P_{k_1}h_1(t)\|_{L^2}2^{k_2}\|P_{k_2}h_2(t)\|_{L^\infty}\|P_{\leq k-6}h_3(t)\|_{L^\infty}2^{-k_1}|k_1-k_2|\\
		&\lesssim \sum_{k_2\leq k-6} 2^{k/2}(1+2^{k}t)^{-1/10}2^{k_2-k}|k-k_2|\|f_{1}\|_{Z}\|f_{2}\|_{L^\infty}\|f_{3}\|_{L^\infty}\\
		&\lesssim2^{k/2}\min\{1,(2^kt)^{-1/10}\}\|f_{1}\|_{Z}\|f_{2}\|_{L^\infty}\|f_{3}\|_{L^\infty},
		\end{split}
		\end{equation*}
		as claimed.
	\end{proof}
	
	In the following lemma, we estimate the nonlocal term $G_{k,1}^{1,3}$.
	
	\begin{lemma}\label{lemma:L2}
		Assume that $f_{1}\in Z$, $f_{2}, f_{3}\in Z_1$, and define $h_j(t,x):=f_j(t,x-a_j+\mathfrak q(t,x))$ as before, for some points $a_j\in\R$. 
		Then, with $G_{k,1}^{1,3}$ as in \eqref{G12},
		\begin{equation*}
		\|G_{k,1}^{1,3}(t)\|_{L^2}\lesssim 2^{k/2}\min\{1,(2^kt)^{-1/10}\}\|f_{1}\|_{Z}\|f_{2}\|_{Z_1}\|f_{3}\|_{Z_1}.
		\end{equation*}	
	\end{lemma}
	
	\begin{proof}
		We may assume that $\|f_{1}\|_Z=\|f_{2}\|_{Z_1}=\|f_{3}\|_{Z_1}=1$. Formula \eqref{G12} shows that
		\begin{equation*}
		G_{1,k}^{1,3}=\frac{1}{\pi}\sum_{k_1\in[k-3,k+3]}P_k\int_{|\alpha|\geq 2^{-k}}\frac{d}{dx}P_{k_1}h_1(x-\alpha)\frac{T_2(x,\alpha,k-8)T_3(x,\alpha,k-8)}{\alpha}d\alpha,
		\end{equation*}
		where $T_2=T(f_2)$ and $T_3=T(f_3)$ are defined as in \eqref{T}.
		We begin by denoting
		\begin{equation*}
		K(x,\alpha):=\frac{T_2(x,\alpha,k-8)T_3(x,\alpha,k-8)}{\alpha},
		\end{equation*}
		and we notice that, for all $\alpha$ and $x$, we have
		\begin{equation}\label{K}
		|K(x,\alpha)|\lesssim\frac{1}{|\alpha|}\qquad\text{and}\qquad|\partial_\alpha K(\alpha,x)|\lesssim\frac{1}{\alpha^2}.
		\end{equation}
		Since a simple change of variable yields $\frac{d}{dx}P_{k_1}h_1(x-\alpha)=-\frac{d}{d\alpha}[P_{k_1}h_1(x-\alpha)]$, we can integrate by parts to see that
		\begin{equation*}
		\begin{split}
		-\int_{|\alpha|\geq 2^{-k}}\frac{d}{dx}P_{k_1}h_1(x-\alpha)K(x,\alpha)\,d\alpha&=-\int_{|\alpha|\geq 2^{-k}}P_{k_1}h_1(x-\alpha)\partial_\alpha K(x,\alpha)\,d\alpha\\
		&-P_{k_1}h_1(x-2^{-k})K(x,2^{-k})+P_{k_1}h_1(x+2^{-k})K(x,-2^{-k}).
		\end{split}	
		\end{equation*}
		Using the bounds \eqref{K} and Lemma \ref{lem:g_h_Z} we have
		\begin{equation*}
		\|G_{k,1}^{1,3}(t)\|_{L^2}\lesssim\sum_{k_1\in[k-3,k+3]}2^k\|P_{k_1}h_1(t)\|_{L^2}\lesssim 2^{k/2}\min\{1,(2^kt)^{-1/10}\},
		\end{equation*}
		as claimed.
	\end{proof}

	It remains to understand the term $G^{1,4}_{k,1}$ in \eqref{G13}. To this end, we need to define a new object, which we refer to as the associated velocity field. We provide the following general definition: for any two functions $g_1,g_2\in Z_1$ and any two base points $a_1,a_2\in\R$  we define the bilinear expression
	\begin{equation}\label{lpo1}
	V[g_1,g_2](t,x):=-\frac{1}{\pi}\int_{|\alpha|\geq t}\frac{p_1(t,x,\alpha)p_2(t,x,\alpha)}{\alpha}\,d\alpha,
	\end{equation}
	where $p_1$ and $p_2$ are defined as in \eqref{pbigDef}, with the indices $\{1,2\}$ referring to the numbering of the base points $a_1$ and $a_2$.

	\begin{lemma}\label{lem:G1greq4}
		Assume that $f_{1}\in Z$, $f_{2}, f_{3}\in Z_1$, and define $h_j(t,x):=f_j(t,x-a_j+\mathfrak q(t,x))$ as before, for some points $a_j\in\R$. 
		Then, with $G_{k,1}^{1,4}$ as in \eqref{G13},
		\begin{equation*}
		\Big\|G_{k,1}^{1,4}(t,x)-P_{k}\Big\{\frac{d}{dx}h_1(t,x)\cdot P_{\leq k-4}V[f_2,f_3](t,x)\Big\}\Big\|_{L^2_x}\lesssim 2^{k/2}(2^kt)^{-1/10}\|f_{1}\|_{Z}\|f_{2}\|_{Z_1}\|f_{3}\|_{Z_1}.
		\end{equation*}
	\end{lemma}
	
	\begin{proof} From the definitions and Lemma \ref{lem:g_h_Z} it suffices to prove that
		\begin{equation}\label{lpo2}
		\Big\|-\frac{1}{\pi}\int_{|\alpha|\geq 2^{-k}}\frac{T_2(t,x,\alpha,k-8)T_3(t,x,\alpha,k-8)}{\alpha}d\alpha-V[f_2,f_3](t,x)\Big\|_{L^\infty_x}\lesssim [1+(2^kt)^{-1/10}]\|f_{2}\|_{Z_1}\|f_{3}\|_{Z_1},
		\end{equation}
		where $T_2=T_2(f_2)$ and $T_3=T_3(f_3)$ are defined as in \eqref{T}.
		
		We may assume that $\|f_2\|_{Z_1}=\|f_3\|_{Z_1}=1$. In order to prove \eqref{lpo2}, we proceed in various steps. We begin by exploiting the structure of the functions in the $Z_1$ space to remove all the error terms. Indeed, from Lemma \ref{lemma:error}, for any $t,x\in[0,\infty)\times\R$  we have
		\begin{equation}\label{lpo3}
		\begin{split}
		&\Big|\int_{|\alpha|\geq 2^{-k}}\frac{T_2(t,x,\alpha,k-8)T_3(t,x,\alpha,k-8)}{\alpha}d\alpha-\int_{|\alpha|\geq 2^{-k}}\frac{p_2(t,x,\alpha)p_3(t,x,\alpha)}{\alpha}d\alpha\Big|\\
		&\lesssim \int_{|\alpha|\geq 2^{-k}}\frac{1}{|\alpha|}\Big\{\frac{|\alpha||x-a_2+\mathfrak q(t,x)|}{|x-a_2+\mathfrak q(t,x)|^2+|\alpha|^2}+\frac{|\alpha||x-a_3+\mathfrak q(t,x)|}{|x-a_3+\mathfrak q(t,x)|^2+|\alpha|^2}+\frac{1}{2^{k}|\alpha|}\Big\}^{9/10}\,d\alpha\\
		&\lesssim 1.
		\end{split}
		\end{equation}
		Next, we notice that if $2^kt\lesssim 1$ then
		\begin{equation}\label{lpo4}
		\Big|\int_{|\alpha|\geq 2^{-k}}\frac{p_2(t,x,\alpha)p_3(t,x,\alpha)}{\alpha}d\alpha-\int_{|\alpha|\geq t}\frac{p_2(t,x,\alpha)p_3(t,x,\alpha)}{\alpha}d\alpha\Big|\lesssim \log(2+(2^kt)^{-1}),
		\end{equation}
		since $|p_2(t,x,\alpha)|+|p_3(t,x,\alpha)|\lesssim 1$. We see now that the bounds \eqref{lpo2} follow from \eqref{lpo3} and \eqref{lpo4} if $2^kt\lesssim 1$. 
		
		It remains to prove that
		\begin{equation*}\label{lpo5}
		\Big|\int_{|\alpha|\geq 2^{-k}}\frac{p_2(t,x,\alpha)p_3(t,x,\alpha)}{\alpha}d\alpha-\int_{|\alpha|\geq t}\frac{p_2(t,x,\alpha)p_3(t,x,\alpha)}{\alpha}d\alpha\Big|\lesssim 1,
		\end{equation*}
		uniformly for any $(t,x)\in[0,\infty)\times\R$ satisfying $2^kt\geq 2$. This follows from Lemma \ref{lemma:error2} in which we show that for $2^kt\geq2$ and $|\alpha|\leq2t$, the functions $p$ can be approximated by some function $p^*$, which is even in $\alpha$. Thus we obtain
		\begin{equation*}
		\int_{|\alpha|\in[2^{-k},t]}\frac{p^\ast_2(t,x,\alpha)p^\ast_3(t,x,\alpha)}{\alpha}d\alpha=0.
		\end{equation*}
		for any $(t,x)\in[0,\infty)\times\R$ satisfying $2^kt\geq 2$.
	\end{proof}

\subsection{The velocity field $V$} We examine now the formula \eqref{lpo1} defining the velocity field $V[g_1,g_2]$. Assume that $\mathfrak q$ satisfies the bounds \eqref{q_bounds} and, in addition,
\begin{equation}\label{q_boundsN}
\|\partial_x\mathfrak q(t,.)\|_{L^4_x}\lesssim\varepsilon t^{1/4}\qquad\text{ for any }t\in[0,T].
\end{equation}
Assume that the functions $g_1,g_2\in Z_1$ satisfy the bounds
\begin{equation}\label{lpo10}
\|g_1\|_{Z_1}+\|g_2\|_{Z_1}\lesssim \varepsilon,\qquad \|g_1^e\|_{L^\infty_tL^p_x}+\|g_2^e\|_{L^\infty_tL^p_x}\lesssim 1,
\end{equation}
for some $p\in[1,\infty)$, where here $g_j^e(t,x):=g_j(t,x)+g_j(t,-x)$, denotes the even part of the function $g_j$. Assume that $a_1,a_2\in\R$ are two points, and define the velocity field $V[g_1,g_2]$ as in \eqref{lpo1}. Then:

\begin{lemma}\label{lemma:L-inf}
	With the assumptions above and $\widetilde{\mathfrak Q}$ defined as in \eqref{Q3} and \eqref{changeV}, we can decompose
	\begin{equation*}\label{lpo11}
	V[g_1,g_2](t,\widetilde{\mathfrak Q}(t,y))=V_1[g_1,g_2](t,y)+V_2[g_1,g_2](t,y),
	\end{equation*}
	where
	\begin{equation}\label{lpo11.1}
	\begin{split}
	&|V_1[g_1,g_2](t,y)|\lesssim \varepsilon^2\log (2/t)\mathbf{1}_{[0,2^{-20}]}(t)\sum_{j\in\{1,2\}}\mathbf{1}_{[a_j-2^{-10},a_j+2^{-10}]}(y),\\
	&|\partial_y^nV_1[g_1,g_2](t,y)|\lesssim\varepsilon^2\sum_{j\in\{1,2\},\,|y-a_j|\geq t}\frac{1}{|y-a_j|^n}\ln\Big(\frac{2|y-a_j|}{t}\Big),
	\end{split}
	\end{equation}
	for $n\in\{1,2,3\}$, and
	\begin{equation}\label{lpo11.2}
	|V_2[g_1,g_2](t,y)|\lesssim \varepsilon.
	\end{equation}
	The sum in the second line of \eqref{lpo11.1} is taken over all the indices $j$ with the property that $|y-a_j|\geq t$ (so it vanishes if $|y-a_j|\leq t$ for all $j\in\{1,2\}$).
\end{lemma}

\begin{proof} From the definitions, we have
	\begin{equation*}\label{lpo21}
	V[g_1,g_2](t,\widetilde{\mathfrak Q}(t,y))=-\frac{1}{\pi}\int_{|\alpha|\geq t}\frac{\widetilde{p}_1(t,y,\alpha)\widetilde{p}_2(t,y,\alpha)}{\alpha}\,d\alpha,
	\end{equation*}
	where $\widetilde{\mathfrak Q}$ is defined as in \eqref{Q3}, and for $j\in\{1,2\}$, $\widetilde p_j$, using the change of variables \eqref{changeV}, is defined as in \eqref{lpo22}-\eqref{lpo23}. Moreover, recall the definition of $g_j^*(t,0,\alpha)$ for $j\in\{1,2\}$, defined in \eqref{eq0}: 
	\begin{equation}\label{lpo25}
	g_j^*(t,0,\alpha):=\frac{1}{\alpha}\int_0^\alpha g_j(t,-\rho)\,d\rho.
	\end{equation}

{\bf{Step 1.}} We begin by showing an important inequality which will be used several times in the proof of this lemma. We have that
	\begin{equation}\label{lpo28}
	|g_j^*(t,0,\alpha)-\widetilde{g}_j(t,y,\alpha)|\lesssim \varepsilon\frac{|y-a_j|}{|\alpha|}+\varepsilon |t|^{1/4}|\alpha|^{-1/4}
	\end{equation}
	for $j\in\{1,2\}$ and $t,\alpha,y$ satisfying $|\alpha|\geq 4|y-a_j|$, $|\alpha|\geq t$, where $\widetilde{g}_j(t,y,\alpha):= \widetilde{g}^+_j(t,y,\alpha)\mathbf{1}_{[0,\infty)}(\alpha)+\widetilde{g}^-_j(t,y,\alpha)\mathbf{1}_{(-\infty,0]}(\alpha)$. Indeed, assuming for example that $\alpha>0$ and using the bounds $|g_j(t,z)|+|zg'_j(t,z)|\lesssim\|g_j\|_{Z_1}\lesssim \varepsilon$ for any $(t,z)\in[0,\infty)\times\R$ we estimate
	\begin{equation*}
	\begin{split}
	|g_j^*(t,0,\alpha)-\widetilde{g}_j(t,y,\alpha)|&\lesssim \frac{\varepsilon|y-a_j|}{|\alpha|}+\frac{1}{\alpha}\int_{2|y-a_j|}^{\alpha}\big|g_j(t,-\rho)-g_j[t,-\rho+\mathfrak q(t,\widetilde{\mathfrak Q}(t,y)-\rho)-\mathfrak q(t,\widetilde{\mathfrak Q}(t,y))]\big|\,d\rho\\
&\lesssim\frac{\varepsilon|y-a_j|}{|\alpha|}+\frac{\varepsilon}{\alpha}\int_{2|y-a_j|}^{\alpha}\bigg|\int_{-\rho+\mathfrak q(t,\widetilde{\mathfrak Q}(t,y)-\rho)-\mathfrak q(t,\widetilde{\mathfrak Q}(t,y))}^{-\rho}|\eta|^{-1}d\eta\bigg|\,d\rho\\
&=\frac{\varepsilon|y-a_j|}{|\alpha|}+\frac{\varepsilon}{\alpha}\int_{2|y-a_j|}^{\alpha}\bigg|\log{\bigg(1-\frac{\mathfrak q(t,\widetilde{\mathfrak Q}(t,y)-\rho)-\mathfrak q(t,\widetilde{\mathfrak Q}(t,y))}{\rho}\bigg)}\bigg|\,d\rho,
	\end{split}
	\end{equation*}
 thus we obtain
 	\begin{equation*}
	\begin{split}
	|g_j^*(t,0,\alpha)-\widetilde{g}_j(t,y,\alpha)|&\lesssim\frac{\varepsilon|y-a_j|}{|\alpha|}+\frac{\varepsilon}{\alpha}\int_{2|y-a_j|}^{\alpha}\frac{\big|\mathfrak q(t,\widetilde{\mathfrak Q}(t,y)-\rho)-\mathfrak q(t,\widetilde{\mathfrak Q}(t,y))\big|}{|\rho|}\,d\rho\\
	&\lesssim \frac{\varepsilon|y-a_j|}{|\alpha|}+\frac{\varepsilon}{\alpha}\int_{2|y-a_j|}^{\alpha}\frac{|\rho|^{3/4}\|\partial_x\mathfrak q(t,.)\|_{L^4}}{|\rho|}\,d\rho\\
	&\lesssim \frac{\varepsilon|y-a_j|}{|\alpha|}+\varepsilon^2 t^{1/4}|\alpha|^{-1/4},
	\end{split}
	\end{equation*}
	using H\"older's inequality in the second line and the assumption \eqref{q_boundsN} to prove the last bounds. 

	{\bf{Step 2.}} We define the first $L^\infty$ error
	\begin{equation*}\label{lpo12}
	V_2^1[g_1,g_2](t,y):=-\frac{1}{\pi}\int_{|\alpha|\geq \max(t,2^{-20})}\frac{\widetilde{p}_1(t,y,\alpha)\widetilde{p}_2(t,y,\alpha)}{\alpha}\,d\alpha.
	\end{equation*}
	We prove that $V_2^1$ satisfies the $L^\infty$ bounds \eqref{lpo11.2}. Without loss of generality we may assume that $|y-a_1|\leq |y-a_2|$. Clearly
	\begin{equation}\label{lpo13}
	\int_{|\alpha|\in[\max(t,2^{-20}),|y-a_1|]}\frac{\widetilde{p}_1(t,y,\alpha)\widetilde{p}_2(t,y,\alpha)}{\alpha}\,d\alpha=0,
	\end{equation}
	due to the definition \eqref{pbigDef} and the oddness of the integrand. For $|\alpha|\in[|y-a_1|,|y-a_2|]$, using the fact that $|\widetilde{p}_j(t,y,\alpha)|\lesssim\|g_j(t)\|_{L^\infty}\lesssim\ep$ for $j\in\{1,2\}$, and that $\alpha$ lies in a unit size interval, we get
 \begin{equation*}
     \int_{|\alpha|\in[|y-a_1|,|y-a_2|]}\Big|\frac{\widetilde{p}_1(t,y,\alpha)\widetilde{p}_2(t,y,\alpha)}{\alpha}\Big|\,d\alpha\lesssim\ep^2.
 \end{equation*}
 On the other hand, if $|\alpha|\geq  |y-a_2|$ then we rewrite
 \begin{equation}\label{lpo13.1}
\begin{split}
     \int_{|\alpha|\geq\max{(2^{-20},|y-a_2|)}}\frac{\widetilde{p}_1(t,y,\alpha)\widetilde{p}_2(t,y,\alpha)}{\alpha}d\alpha=\int_{\alpha\geq\max{(2^{-20},|y-a_2|)}}\frac{[\widetilde{p}_1(t,y,\alpha)+\widetilde{p}_1(t,y,-\alpha)]\widetilde{p}_2(t,y,\alpha)}{\alpha}d\alpha&\\-\int_{\alpha\geq\max{(2^{-20},|y-a_2|)}}\frac{[\widetilde{p}_2(t,y,\alpha)+\widetilde{p}_2(t,y,-\alpha)]\widetilde{p}_1(t,y,\alpha)}{\alpha}d\alpha.
\end{split}
 \end{equation}
 Using \eqref{lpo22}-\eqref{lpo23}, we see that
\begin{equation*}
    \widetilde{p}_1(t,y,\alpha)+\widetilde{p}_1(t,y,-\alpha)=\widetilde{g}_j^+(t,y,\alpha)+\widetilde{g}_j^-(t,y,-\alpha).
    \end{equation*}
    From \eqref{lpo28} we get
    \begin{equation*}
    \begin{split}
    |\widetilde{p}_1(t,y,\alpha)+\widetilde{p}_1(t,y,-\alpha)|&\lesssim \varepsilon\frac{|y-a_j|}{|\alpha|}+\varepsilon |t|^{1/4}|\alpha|^{-1/4}+|g_j^*(t,0,\alpha)+g_j^*(t,0,-\alpha)|\\&\lesssim \varepsilon\frac{|y-a_j|}{|\alpha|}+\varepsilon |t|^{1/4}|\alpha|^{-1/4}+|\alpha|^{-1/p}\|g(t,x)+g(t,-x)\|_{L^p}\\&\lesssim \varepsilon\frac{|y-a_j|}{|\alpha|}+\varepsilon |t|^{1/4}|\alpha|^{-1/4}+|\alpha|^{-1/p},
    \end{split}
    \end{equation*}
    for some $p\in[1,\infty)$, where we used \eqref{lpo10} in the last line. Since $|\widetilde{p}_2(t,y,\alpha)|\lesssim \|g_2(t)\|_{L^\infty}\lesssim \varepsilon$ it follows that
	\begin{equation*}
	\int_{\alpha\geq \max(2^{-20},|y-a_1|)}\Big|\frac{[\widetilde{p}_1(t,y,\alpha)-\widetilde{p}_1(t,y,-\alpha)]\widetilde{p}_2(t,y,\alpha)}{\alpha}\Big|\,d\alpha\lesssim \varepsilon.
	\end{equation*}
    A similar argument follows for the second term on the right hand side of \eqref{lpo13.1}. We thus get
\begin{equation}\label{lpo14}
    \int_{|\alpha|\geq\max{(2^{-20},|y-a_2|)}}\Big|\frac{\widetilde{p}_1(t,y,\alpha)\widetilde{p}_2(t,y,\alpha)}{\alpha}\Big|d\alpha\lesssim\ep.
\end{equation}
    
	It follows from \eqref{lpo13} and \eqref{lpo14} that
	\begin{equation*}\label{lpo15}
	|V^1_2[g_1,g_2](t,y)|\lesssim \varepsilon.
	\end{equation*}
 
	{\bf{Step 3.}} We notice that for $t\leq 2^{-20}$
	\begin{equation}\label{lpo16}
	V[g_1,g_2](t,\widetilde{\mathfrak Q}(t,y))-V_2^1[g_1,g_2](t,y)=-\frac{1}{\pi}\int_{|\alpha|\in[t,2^{-20}]}\frac{\widetilde{p}_1(t,y,\alpha)\widetilde{p}_2(t,y,\alpha)}{\alpha}\,d\alpha.
	\end{equation}

	We would like to remove now the $y$-dependence from the functions $\widetilde{g}_j^\pm$. We define
	\begin{equation}\label{lpo26}
	\begin{split}
	r_j(t,y,\alpha):=g_j^*(t,0,\alpha)\varphi_{\leq -4}((y-a_j)/\alpha)+g_j(t,y-a_j)\varphi_{\leq -4}(\alpha/(y-a_j)).
	\end{split}
	\end{equation}
	Moreover, we define
	\begin{equation}\label{lpo27}
	\begin{split}
	&V_1[g_1,g_2](t,y):=-\frac{1}{\pi}\int_{|\alpha|\in [t,2^{-20}]}\frac{r_1(t,y,\alpha)r_2(t,y,\alpha)}{\alpha}\,d\alpha,\\
	&V_2^2[g_1,g_2](t,y):=V[g_1,g_2](t,\widetilde{\mathfrak Q}(t,y))-V_2^1[g_1,g_2](t,y)-V_1[g_1,g_2](t,y).
	\end{split}
	\end{equation}
	
	{\bf{Step 4.}} We would like to show that $V_2^2$ satisfies the $L^\infty$ bounds \eqref{lpo11.2}, which follows from 
 \eqref{lpo28}. Using the formulas \eqref{lpo16} and \eqref{lpo27} we have
	\begin{equation*}
	\begin{split}
	|V_2^2[g_1,g_2](t,y)|&\lesssim \int_{|\alpha|\in[t,2^{-20}]}\Big|\frac{\widetilde{p}_1(t,y,\alpha)\widetilde{p}_2(t,y,\alpha)}{\alpha}-\frac{r_1(t,y,\alpha)r_2(t,y,\alpha)}{\alpha}\Big|\,d\alpha\\
	&\lesssim\varepsilon\sum_{j\in\{1,2\}}\int_{|\alpha|\in[t,2^{-20}]}\frac{|\widetilde{p}_j(t,y,\alpha)-r_j(t,y,\alpha)|}{|\alpha|}\,d\alpha\\
	&\lesssim \varepsilon^2+\varepsilon\sum_{j\in\{1,2\}}\int_{|\alpha|\geq \max(t,|y-a_j|)}\frac{|\widetilde{g}_j(t,y,\alpha)-g_j^*(t,0,\alpha)|}{|\alpha|}\,d\alpha\\
	&\lesssim \varepsilon^2,
	\end{split}
	\end{equation*}
	where, from the second to the third lines, we pulled out finitely many dyadic pieces.
	
	{\bf{Step 5.}} We now prove the bounds \eqref{lpo11.1} on $V_1$. The bounds in the first line follow easily once we notice that the integrand $\frac{r_1(t,y,\alpha)r_2(t,y,\alpha)}{\alpha}$ is odd in $\alpha$ if $|y-a_j|\geq 2^{-10}$ for all $j\in\{1,2\}$, thus $V_1[g_1,g_2](t,y)$ vanishes in this case.
	
	To prove the derivative bounds we begin by rewriting
	\begin{equation*}
    V_1[g_1,g_2]=V_1^1[g_1,g_2]+V_1^2[g_1,g_2]+V_1^3[g_1,g_2]+V_1^4[g_1,g_2]
	\end{equation*}
	where
	\begin{equation*}\label{lpo29}
	\begin{split}
	&V_1^1[g_1,g_2](t,y):=-\frac{1}{\pi}\int_{|\alpha|\in [t,2^{-20}]}\frac{g_{1}^*(t,0,\alpha)\varphi_{\leq -4}((y-a_1)/\alpha)g_2^*(t,0,\alpha)\varphi_{\leq -4}((y-a_2)/\alpha)}{\alpha}\,d\alpha,\\
	&V_1^2[g_1,g_2](t,y):=-\frac{1}{\pi}\int_{|\alpha|\in [t,2^{-20}]}\frac{g_1^*(t,0,\alpha)\varphi_{\leq -4}((y-a_1)/\alpha)g_2(t,y-a_2)\varphi_{\leq -4}(\alpha/(y-a_2))}{\alpha}\,d\alpha,\\
	&V_1^3[g_1,g_2](t,y):=-\frac{1}{\pi}\int_{|\alpha|\in [t,2^{-20}]}\frac{g_1(t,y-a_1)\varphi_{\leq -4}(\alpha/(y-a_1))g_2^*(t,0,\alpha)\varphi_{\leq -4}((y-a_2)/\alpha)}{\alpha}\,d\alpha,\\
	&V_1^4[g_1,g_2](t,y):=-\frac{1}{\pi}\int_{|\alpha|\in [t,2^{-20}]}\frac{g_1(t,y-a_1)\varphi_{\leq -4}(\alpha/(y-a_1))g_2(t,y-a_2)\varphi_{\leq -4}(\alpha/(y-a_2))}{\alpha}\,d\alpha.
	\end{split}
	\end{equation*}
	To begin with, due to the oddness of $1/\alpha$, we have
	\begin{equation*}
	V_1^4[g_1,g_2](t,y)\equiv 0.
	\end{equation*}
	
	Without loss of generality, in proving the derivative bounds in \eqref{lpo11.1} we may assume that $|y-a_1|\leq |y-a_2|$. In this case we notice that $V_1^3[g_1,g_2](t,y)\equiv 0$ as well.
	
	We calculate
	\begin{equation*}
	\begin{split}
	\partial_y V_1^1[g_1,g_2](t,y)&=-\frac{1}{\pi}\int_{|\alpha|\in [t,2^{-20}]}\frac{g_1^*(t,0,\alpha)\frac{1}{\alpha}\varphi'_{\leq -4}((y-a_1)/\alpha)g_2^*(t,0,\alpha)\varphi_{\leq -4}((y-a_2)/\alpha)}{\alpha}\,d\alpha\\
 &\quad-\frac{1}{\pi}\int_{|\alpha|\in [t,2^{-20}]}\frac{g_1^*(t,0,\alpha)\varphi_{\leq -4}((y-a_1)/\alpha)g_2^*(t,0,\alpha)\frac{1}{\alpha}\varphi'_{\leq -4}((y-a_2)/\alpha)}{\alpha}\,d\alpha.
	\end{split}	
	\end{equation*}
	Since $\varphi'_{\leq -4}\big(\frac{y-a_i}{\alpha}\big)=0$ unless $\frac{|y-a_i|}{|\alpha|}\in[2^{-5},2^{-3}]$, and  $|\alpha|\geq t$, we easily get
	\begin{equation}\label{der1}
	|\partial_yV_1^1[g_1,g_2](t,y)|\lesssim\varepsilon^2
	\begin{cases}
	\frac{1}{|y-a_1|} & \text{ if }|y-a_1|,|y-a_2|\geq t,\\
	\frac{1}{|y-a_2|} & \text{ if } |y-a_1|\leq t\leq |y-a_2|,\\
	0 & \text{ if }|y-a_1|,|y-a_2|\leq t,
	\end{cases}
	\end{equation} 
	since we estimate $|g_j^*(t,0,\alpha)|\lesssim \varepsilon$ and integrate $\alpha$ only over one dyadic piece at a time, yielding integrals of order 1. Similarly, for $n\in\{2,3\}$ we get 
	\begin{equation}\label{der2}
	|\partial_y^nV_1^1[g_1,g_2](t,y)|\lesssim\varepsilon^2
	\begin{cases}
	\frac{1}{|y-a_1|^n} & \text{ if }|y-a_1|,|y-a_2|\geq t,\\
	\frac{1}{|y-a_2|^n} & \text{ if } |y-a_1|\leq t\leq |y-a_2|,\\
	0 & \text{ if }|y-a_1|,|y-a_2|\leq t.
	\end{cases}
	\end{equation}

	Taking the derivative in $y$ of $V_1^2$ yields
	\begin{equation*}
	\begin{split}
	\partial_yV_1^2[g_1,g_2](t,y)&=-\frac{1}{\pi}\int_{|\alpha|\in [t,2^{-20}]}\frac{g_1^*(t,0,\alpha)\frac{1}{\alpha}\varphi'_{\leq -4}((y-a_1)/\alpha)g_2(t,y-a_2)\varphi_{\leq -4}(\alpha/(y-a_2))}{\alpha}\,d\alpha
	\\
 &\quad-\frac{1}{\pi}\int_{|\alpha|\in [t,2^{-20}]}\!\!\!\!\!\!\!\!\!\frac{g_1^*(t,0,\alpha)\varphi_{\leq -4}((y-a_1)/\alpha)g_2(t,y-a_2)\frac{\alpha}{(y-a_2)^2}\varphi'_{\leq -4}(\alpha/(y-a_2))}{\alpha}\,d\alpha
	\\
 &\quad-\frac{1}{\pi}\int_{|\alpha|\in [t,2^{-20}]}\frac{g_1^*(t,0,\alpha)\varphi_{\leq -4}((y-a_1)/\alpha)g'_2(t,y-a_2)\varphi_{\leq -4}(\alpha/(y-a_2))}{\alpha}\,d\alpha.
	\end{split}	
	\end{equation*}
	The first two terms are treated as for $V_1^1$, where the integral in $\alpha$ is taken over each dyadic piece at a time. For the last term, use the fact that $|x\partial_{x}g(x)|\lesssim \varepsilon$ and we integrate $1/\alpha$ over the interval $[t,|y-a_2|]$, and hence we get
	\begin{equation}\label{der3}
	|\partial_yV_1^2[g_1,g_2](t,y)|\lesssim\varepsilon^2
	\begin{cases}
	\frac{1}{|y-a_1|}+\frac{1}{|y-a_2|}\log\big(\frac{2|y-a_2|}{t}\big) & \text{ if }|y-a_1|,|y-a_2|\geq t,\\
	\frac{1}{|y-a_2|}\log\big(\frac{2|y-a_2|}{t}\big) & \text{ if } |y-a_1|\leq t\leq |y-a_2|,\\
	0 & \text{ if }|y-a_1|,|y-a_2|\leq t.
	\end{cases}
	\end{equation} 
	Similarly, for $n\in\{2,3\}$ we get 
	\begin{equation}\label{der4}
	|\partial_y^nV_1^2[g_1,g_2](t,y)|\lesssim\varepsilon^2
	\begin{cases}
	\frac{1}{|y-a_1|^n}+\frac{1}{|y-a_2|^n}\log\big(\frac{y-a_2}{t}\big) & \text{ if }|y-a_1|,|y-a_2|\geq t,\\
	\frac{1}{|y-a_2|^n}\log\big(\frac{2|y-a_2|}{t}\big) & \text{ if } |y-a_1|\leq t\leq |y-a_2|,\\
	0 & \text{ if }|y-a_1|,|y-a_2|\leq t.
	\end{cases}
	\end{equation} 
	
	The desired bounds \eqref{lpo11.1} follow from \eqref{der1}--\eqref{der2} and \eqref{der3}--\eqref{der4} in the case $|y-a_1|\leq |y-a_2|$ (once we recall that $V_1^3[g_1,g_2](t,y)\equiv 0$ and $V_1^4[g_1,g_2](t,y)\equiv 0$ in this case.
\end{proof}

\begin{corollary}\label{lemma:L-inf2}
	With the assumptions and the notation of Lemma \ref{lemma:L-inf} above, let
	\begin{equation*}\label{lpo31}
	W[g_1,g_2](t,y):=\int_0^tV_1[g_1,g_2](s,y)\,ds.
	\end{equation*}
	Then, for $n\in\{1,2,3\}$,
	\begin{equation*}\label{lpo32}
	\begin{split}
	&|W[g_1,g_2](t,y)|\lesssim \varepsilon^2t\log (2/t)\sum_{j\in\{1,2\}}\mathbf{1}_{[a_j-2^{-10},a_j+2^{-10}]}(y),\\
	&|\partial_y^nW[g_1,g_2](t,y)|\lesssim\varepsilon^2\sum_{j\in\{1,2\},\,|y-a_j|\geq t}\frac{t}{|y-a_j|^n}\ln\Big(\frac{2|y-a_j|}{t}\Big)+\varepsilon^2\sum_{j\in\{1,2\},\,|y-a_j|\leq t}\frac{1}{|y-a_j|^{n-1}}.
	\end{split}
	\end{equation*}
 In particular,
 \begin{equation*}
     \|\partial_y W[g_1,g_2](t,y)\|_{L_x^4}\lesssim\ep t^{1/4}.
 \end{equation*}
\end{corollary}

\begin{proof} The bounds follow directly by integrating the bounds \eqref{lpo11.1} for $s\in[0,t]$. 
\end{proof}

The following lemma provides additional bounds on $V_1$ which will be very useful for constructing the solution in the next section. 
\begin{lemma}\label{lem:V1pk}
	Let $V_1[f_1,f_2](t,y)$ be defined as in \eqref{lpo27}, with $(f_1,f_2)\in Z_1\times Z_1$. Then, for $t<2^{-20}$,
	\begin{equation}\label{V1_bounds}
	\big|V_1[P_kf_1,f_2](t,y) \big|\lesssim(2^kt)^{-1/10}\|f_{1}\|_{Z_1}\|f_{2}\|_{Z_1}.
	\end{equation}
\end{lemma}	
\begin{proof}	We may assume that $\|f_{1}\|_{Z_1}=\|f_{2}\|_{Z_1}=1$ and that the function $\widehat{f_1}$ is supported in the set $\{\xi:|\xi|\in [2^{k-2},2^{k+2}]\}$.
	Recalling the definition of $V_1$ in \eqref{lpo27}, we have
	\begin{equation}\label{V1aux}
	\big|V_1[P_kf_1,f_2](t,y) \big|\lesssim\int_{|\alpha|\in [t,2^{-20}]}\frac{|r_1(t,y,\alpha)|}{|\alpha|}d\alpha,
	\end{equation}
	with 
	\begin{equation}\label{r_1}
	r_1(t,y,\alpha)=f_1^*(t,0,\alpha)\varphi_{\leq -4}\bigg(\frac{y-a_1}{\alpha}\bigg)+P_kf_1(t,y-a_1)\varphi_{\leq -4}\bigg(\frac{\alpha}{y-a_1}\bigg),
	\end{equation}
	where
	\begin{equation*}
	f_1^*(t,0,\alpha)=\frac{1}{\alpha}\int_{0}^\alpha P_kf_1(t,-\rho)\,d\rho.
	\end{equation*}

 We split the integral in \eqref{V1aux},
 	\begin{equation*}
	\int_{|\alpha|\in [t,2^{-20}]}\frac{|r_1(t,y,\alpha)|}{|\alpha|}d\alpha=I_1+I_2,
	\end{equation*}
 so that $|\alpha|>2^{-k}$ in $I_1$. We begin with $I_1$. We find that
	\begin{equation*}
	|f_1^*(t,0,\alpha)|\lesssim (2^k|\alpha|)^{-1}\|P_kf_1\|_{L^\infty}\lesssim(1+2^kt)^{-1/10}(2^k|\alpha|)^{-1}\|f_1\|_{Z_1}.
	\end{equation*}
	Moreover, setting $x=y-a_1$ we have
	\begin{equation*}
	P_kf_1(x)=\int_{\R}f_1(x-z)K_k(z)\,dz=\int_{\R}[f_1(x-z)-f_1(x)]K_k(z)\,dz
	\end{equation*}
	where $K_k(z)$, the inverse Fourier transform of the cutoff function, is a dilation of a Schwarz function. We thus get
	\begin{equation*}
	\begin{split}
	|P_kf_1(x)|&=\bigg|\int_{|z|\leq|x|/2}[f_1(x-z)-f_1(x)]K_k(z)\,dz+\int_{|z|\geq|x|/2}[f_1(x-z)-f_1(x)]K_k(z)\,dz \bigg|\\
	&\lesssim(1+2^k t)^{-1/10}\bigg(\int_{|z|\leq|x|/2}\frac{|z|}{|x|}|K_k(z)|\,dz+\int_{|z|\geq|x|/2}|K_k(z)|\,dz\bigg)\\
	&\lesssim (1+2^k t)^{-1/10}(2^k|x|)^{-1},
	\end{split}
	\end{equation*}
	where we used the fact that $f\in Z_1$ in the second line and that $K_k$ is rapidly decaying in the last line. Since the second term in \eqref{r_1} requires that $|x|>|\alpha|$, we get $$|r_1(t,y,\alpha)|\lesssim(1+2^kt)^{-1/10}(2^k|\alpha|)^{-1}\qquad\text{for }|\alpha|>2^{-k},$$ which upon integration yields $$I_1\lesssim(1+2^k t)^{-1/10}.$$
 For $|\alpha|<2^{-k}$, we simply integrate to get
	$$I_2\lesssim |\log(2^kt)|,$$
	thus concluding the proof of \eqref{V1_bounds}.
\end{proof}

\section{The full nonlinearity}\label{sect:full}

Before we proceed to run the fixed point argument to conclude the proof of our main result, we now show how the previous results can be extended to the full nonlinearity as defined in \eqref{nonlin2}. We recall the notation $\underline{\xi}\in\R^{2n+1}$ as $\underline{\xi}:=(\xi_1,\ldots,\xi_{2n+1})$ for scalars and similarly, for functions, $\underline{f}:=(f_1,f_2,\ldots,f_{2n+1})$. Moreover, we will use the symbol $\lesssim_n$ for whenever a term is bounded by a constant of size $C^n$.

We begin by reproving the localized $L^2$ estimates for the pseudoproduct of the full nonlinearity. We get the following lemma.
\begin{lemma}\label{L2_est_full}
	Let $\mathcal{T}_n(P_{k_1}f_1,P_{k_2}f_2,P_{k_3}f_3,\ldots,P_{k_{2n+1}}f_{2n+1})$ be defined as in \eqref{nonlin8.1}. Then we have
	\begin{equation*}\label{nonlin10}
	\begin{split}
		\|P_k\big[\mathcal{T}_n(P_{k_1}f_1,\ldots,P_{k_{2n+1}}f_{2n+1}) \big]\|_{L^2}\lesssim_n 2^k \|P_{k_1}f_1\|_{L^2}\prod_{\ell=2}^{2n+1}\|P_{k_\ell}f_\ell\|_{L^\infty}.
	\end{split}
	\end{equation*}
	Alternatively, we also have
	\begin{equation*}\label{nonlin11}
	\begin{split}
		\|P_k\big[\mathcal{T}_n(P_{k_1}f_1,\ldots,P_{k_{2n+1}}f_{2n+1}) \big]\|_{L^2}&\lesssim_n2^k2^{k/2} \|P_{k_1}f_1\|_{L^2}\| P_{k_2}f_2\|_{L^2}\prod_{\ell=3}^{2n+1}\|P_{k_\ell
  }f_\ell\|_{L^\infty}.
	\end{split}
	\end{equation*}
\end{lemma}
\begin{proof}
We only show here the case when $k_1$ is the highest frequency, as the others are easier. Then, without loss of generality, we assume the ordering $k_1\geq k_2\geq \ldots\geq k_{2n+1}$.
 Following the same ideas as in Lemma~\ref{lemma:TP}, in order to introduce the modified kernel, we split
	\begin{equation*}\label{nonlin13}
		K_{\underline{k}}(\underline{z})=K^{1,1}_{\underline{k}}(\underline{z})+K^{1,2}_{\underline{k}}(\underline{z})+K^{1,3}_{\underline{k}}(\underline{z}),
	\end{equation*}
	with
	\begin{equation*}\label{nonlin14}
		\begin{split}
		&K^{1,1}_{\underline{k}}(\underline{z})=\int_{\R}L_{k_1}(z_1,\alpha)\widetilde{L}_{k_2}(z_2,\alpha)L_{k_3}(z_3,\alpha)\ldots L_{k_{2n+1}}(z_{2n+1},\alpha)\,d\alpha,\\
		&K^{1,2}_{\underline{k}}(\underline{z})=\int_{|\alpha|\notin[2^{-k_1},2^{-k_2}]}L_{k_1}(z_1,\alpha)\frac{1}{\alpha}\min(1,2^{k_2}\alpha)\psi_{0}'(2^{k_2}z_2)L_{k_3}(z_3,\alpha)\ldots L_{k_{2n+1}}(z_{2n+1},\alpha)\,d\alpha,\\
		&K^{1,3}_{\underline{k}}(\underline{z})=\int_{|\alpha|\in[2^{-k_1},2^{-k_2}]}L_{k_1}(z_1,\alpha)2^{k_2}\psi_{0}'(2^{k_2}z_2)L_{k_3}(z_3,\alpha)\ldots L_{k_{2n+1}}(z_{2n+1},\alpha)\,d\alpha.
		\end{split}
		\end{equation*}
		Using Lemmas~\ref{lemma:K}-\ref{lemma:K2} we get
		\begin{equation*}\label{nonlin14.1}
		\|K^{1,1}_{\underline{k}}(\underline{z})\|_{L^1(\R^{2n+1})}+\|K^{1,2}_{\underline{k}}(\underline{z})\|_{L^1(\R^{2n+1})}\lesssim_n2^{-k_1}.
		\end{equation*}
		To estimate $\|K^{1,3}_{\underline{k}}(\underline{z})\|_{L^1(\R^{2n+1})}$, we now further split
		\begin{equation*}\label{nonlin12}
		K^{1,3}_{\underline{k}}(\underline{z})=K^2_{\underline{k}}(\underline{z})+K^3_{\underline{k}}(\underline{z})+\ldots+K^{2n+1}_{\underline{k}}(\underline{z}),
		\end{equation*} 
		where
	\begin{equation*}\label{nonlin15}
		\begin{split}
		K^2_{\underline{k}}(\underline{z})&=\int_{|\alpha|\in[2^{-k_1},2^{-k_2}]}L_{k_1}(z_1,\alpha)2^{k_2}\psi_{0}'(2^{k_2}z_2)\widetilde{L}_{k_3}(z_3,\alpha)\ldots L_{k_{2n+1}}(z_{2n+1},\alpha)\,d\alpha,\\
		K^3_{\underline{k}}(\underline{z})&=\int_{|\alpha|\in[2^{-k_1},2^{-k_2}]}L_{k_1}(z_1,\alpha)2^{k_2}\psi_{0}'(2^{k_2}z_2)2^{k_3}\psi_{0}'(2^{k_3}z_3)\widetilde{L}_{k_4}(z_4,\alpha)\ldots L_{k_{2n+1}}(z_{2n+1},\alpha)\,d\alpha,\\
		 & \vdots \\
		K^{2n}_{\underline{k}}(\underline{z})&=\int_{|\alpha|\in[2^{-k_1},2^{-k_2}]}L_{k_1}(z_1,\alpha)2^{k_2}\psi_{0}'(2^{k_2}z_2)2^{k_3}\psi_{0}'(2^{k_3}z_3)\ldots 2^{k_{2n}}\psi_{0}'(2^{k_{2n}}z_{2n})\widetilde{L}_{k_{2n+1}}(z_{2n+1},\alpha)\,d\alpha,\\
		K^{2n+1}_{\underline{k}}(\underline{z})&=\int_{|\alpha|\in[2^{-k_1},2^{-k_2}]}L_{k_1}(z_1,\alpha)2^{k_2}\psi_{0}'(2^{k_2}z_2)2^{k_3}\psi_{0}'(2^{k_3}z_3)\ldots 2^{k_{2n+1}}\psi_{0}'(2^{k_{2n+1}}z_{2n+1})\,d\alpha.
		\end{split}
	\end{equation*}
	From Lemmas~\ref{lemma:K}-\ref{lemma:K2} we get
	\begin{equation*}\label{nonlin15.1}
	\|K^{2}_{\underline{k}}(\underline{z})\|_{L^1(\R^{2n+1})}+\ldots+\|K^{2n}_{\underline{k}}(\underline{z})\|_{L^1(\R^{2n+1})}\lesssim_n2^{-k_1}.
	\end{equation*}
	It remains to study $K^{2n+1}$. By expressing $L_{k_1}$ as in \eqref{eq.4.20}, and inserting it into $K^{2n+1}$, we get
	\begin{equation}\label{nonlin16}
	\begin{split}
	K^{2n+1}_{\underline{k}}(\underline{z})=\prod_{i=2}^{2n+1}2^{k_i}\psi_{0}'(2^{k_i}z_i)\int_{|\alpha|\in[2^{-k_1},2^{-k_2}]}\frac{[-\psi_{0}(2^{k_1}(x-\alpha))]}{\alpha}\,d\alpha.
	\end{split}
	\end{equation}
	As in the proof of Lemma~\ref{lemma:TP}, upon reinserting \eqref{nonlin16} back into the pseudoproduct \eqref{nonlin8.1} and using the fact that
	\begin{equation*}\label{nonlin17}
		P_{k_i}f_i(x)=\int_{\R}2^{k_i}(P_{k_i}f_i)(x-y_i)\psi_{0}'(2^{k_i}y_i)\,dy_i,\qquad i\in\{1,2,\ldots,2n+1\}.
	\end{equation*}
	and the boundedness of the truncated Hilbert transform in $L^2$ yields the desired bounds.
\end{proof}
We now localize and decompose the nonlinearities. For any tuple $\underline{k}\in\mathbb{Z}^{2n+1}$, we let $\underline{k^*}\in\mathbb{Z}^{2n+1}$ denote its non-increasing rearrangement $k_1^*\geq k_2^*\geq\ldots\geq k_{2n+1}^*$. Given $k\in\mathbb{Z}$, we now define the two sets 
\begin{equation*}\label{nonlin18}
	\begin{split}
	&S_{k,1}^n:=\{\underline{k}\in\mathbb{Z}^{2n+1}:k_1^*\in[k-3n,k+3n]\text{ and } k_2^*,k_3^*,\ldots,k_{2n+1}^*\leq k-6n \},\\
	&S_{k,2}^n:=\{\underline{k}\in\mathbb{Z}^{2n+1}:|k_1^*- k_2^*|\leq10n \text{ where }k_1^*\geq k-3n\text{ and }k_{2}^*\geq k-5n \}.
	\end{split}
\end{equation*}
From \eqref{nonlin8.1}, we thus get
\begin{equation*}\label{nonlin19}
	P_k\big[\mathcal{T}_n(\underline{h}) \big](x)=G_{k,1}^n(x)+G_{k,2}^n(x)
\end{equation*}
where, for $l\in\{1,2 \}$,
\begin{equation*}\label{nonlin20}
G_{k,l}^n[\underline{h}](x):=\frac{(-1)^n}{\pi}\sum_{\underline{k}\in S_{k,l}}P_k\frac{d}{dx}\int_{\R^{2n+1}}\frac{d}{dy_1}\prod_{\ell=1}^{2n+1}P_{k_\ell}h_\ell(y_\ell)K_{\underline{k}}(x-\underline{y})\,d\underline{y}.
\end{equation*}
For the high-high-to-low interactions, from the Lemmas~\ref{lemma:K}-\ref{lemma:K2} and the Lemmas~\ref{L2_est_full} and~\ref{lem:g_h_Z} we get the following bounds.
\begin{lemma}\label{lem:hh_full}
	Assume that $f_1,f_2,\ldots,f_{2n+1}\in Z$, and define $h_j(t,x)=f_j(t,x-a_j+q(t,x))$ for some points $a_j\in\R$. Then for any $k\in\mathbb{Z}$ and $t\in[0,\infty)$ we have
	\begin{equation*}\label{nonlin21}
		\|G_{k,2}[\underline{h}](t)\|_{L^2}\lesssim_n2^{k/2}(1+2^kt)^{-2/10}\prod_{i=1}^{2n+1}\|f_{j_i}\|_{Z}.
	\end{equation*}
\end{lemma}

For the high-low-to-high interactions, we decompose $G_{k,1}^n$ as we did previously. We write the disjoint union $S_{k,1}=\bigcup_{i=1}^{2n+1}S_{k,1}^i$, where $$S_{k,1}^i:=\{\underline{k}\in S_{k,1}:k_i=\max(k_1,k_2,\ldots,k_{2n+1})\},$$ for all $i\in\{1,2,\ldots,2n+1\}$. As in the trilinear case, whenever the function with the highest frequency is not hit by the derivative, the following lemma follows immediately using Lemmas~\ref{lemma:K} and~\ref{lemma:K2}, and Lemmas~\ref{L2_est_full} and~\ref{lem:g_h_Z}.

\begin{lemma}\label{lem:der_full}
	Assume that $f_1,f_2,\ldots,f_{2n+1}\in Z$ and define $h_j(t,x):=f_j(t,x-a_j+q(t,x))$ for some points $a_j\in\R$. Then for $i\in\{2,3,\ldots,2n+1\}$, $k\in\Z$ and $t\in[0,\infty)$ we have
	\begin{equation*}\label{nonlin22}
		\|G_{k,1}^i(t)\|_{L^2}\lesssim_n2^{k/2}(1+2^kt)^{-2/10}\prod_{\ell=1}^{2n+1}\|f_{j_\ell}\|_Z.
	\end{equation*}
\end{lemma} 

Strikingly, we get the desired nonlinear bounds on $G_{k,1}^1$ provided at least one of the lower frequency functions is in the space $Z_2$. The proof follows a similar argument to that of Lemma~\ref{L2_est_full}.

\begin{lemma}\label{lem:mix_full}
	Assume that $(f_1,\ldots,f_{2n+1})\in Z\times\ldots \times Z$ and there is some $\ell\in\{2,\ldots,2n+1\}$ such that $f_\ell\in Z_2$. Define $h_j(t,x)=f_j(t,x-a_j+q(t,x))$ as before. Then, for any $k\in\Z$ and $t\in[0,\infty)$,
	\begin{equation*}\label{nonlin23}
		\|G_{k,1}^1\|_{L^2}\lesssim_n2^{k/2}\min\{1,(2^kt)^{-1/10}\}\|f_\ell\|_{Z_2}\prod_{p\neq\ell}\|f_{j_p}\|_{Z}.
	\end{equation*}
\end{lemma}
\begin{proof}
	We assume without loss of generality that $f_2\in Z_2$. If all lower frequency functions lie in the space $Z_2$, the bounds follow easily from Lemmas~\ref{lemma:K}-\ref{lemma:K2} and Lemmas~\ref{L2_est_full} and~\ref{lem:g_h_Z}. We therefore focus on the case when all other lower frequency functions lie in the space $Z_1$.
	
	We decompose the set of indices $\{3,\ldots,2n+1\}=A+B$ where for all $j\in A$ we have $k_j>k_2$ and for all $i\in B$ we have $k_i<k_2$. We will then consider the following sets.
	\begin{equation*}\label{nonlin23.1}
		\begin{split}
		&S_{k,a}:=\{\underline{k}\in\mathbb{Z}^{2n+1}:\forall (j,i)\in A\times B,|k_1-k|\leq3n, k_2\leq k-6n, k_j\in[k_2-3n,k-6n], B=\emptyset \}\\
		&S_{k,b}:=\{\underline{k}\in\mathbb{Z}^{2n+1}:\forall (j,i)\in A\times B,|k_1-k|\leq3n, k_2\leq k-6n, k_j\in[k_2-3n,k-6n], k_i\leq k-6n-1  \}\\
		&S_{k,c}:=\{\underline{k}\in\mathbb{Z}^{2n+1}:\forall (j,i)\in A\times B,|k_1-k|\leq3n, k_2\leq k-6n, k_i\leq k-6n-1, A=\emptyset \}
		\end{split}
	\end{equation*}
	
	As in the proof of Lemma~\ref{lemmaG1Z2}, we further decompose $G_{k,1}^1$ into $G_{k,1}^1:=G_{k,1}^{1,a}+G_{k,1}^{1,b}+G_{k,1}^{1,c}$ where for $j\neq i\in\{3,\ldots,2n+1\}$
	\begin{equation*}\label{nonlin24}
		G_{k,1}^{1,a}[\underline{h}](x):=\frac{(-1)^n}{\pi}\sum_{\underline{k}\in S_{k,a}}P_k\frac{d}{dx}\int_{\R^{2n+1}}\frac{d}{dy_1}\prod_{\ell=1}^{2n+1}P_{k_\ell}h_\ell(y_{\ell})K_{k_1,\ldots,k_{2n+1}}(x-\underline{y})\,d\underline{y},
	\end{equation*}
	\begin{equation*}\label{nonlin25}
	G_{k,1}^{1,b}[\underline{h}](x):=\frac{(-1)^n}{\pi}\sum_{\underline{k}\in S_{k,b}}P_k\frac{d}{dx}\int_{\R^{2n+1}}\frac{d}{dy_1}\prod_{\ell=1}^{2n+1}P_{k_\ell}h_\ell(y_\ell)K_{k_1,k_2,\ldots,k_j\ldots\leq k_2-3n-1}(x-\underline{y})\,d\underline{y},
	\end{equation*}
	\begin{equation*}\label{nonlin26}
	G_{k,1}^{1,c}[\underline{h}](x):=\frac{(-1)^n}{\pi}\sum_{\underline{k}\in S_{k,c}}P_k\frac{d}{dx}\int_{\R^{2n+1}}\frac{d}{dy_1}\prod_{\ell=1}^{2n+1}P_{k_\ell}h_\ell(y_\ell)K_{k_1,k_2,\leq k_2\ldots}(x-\underline{y})\,d\underline{y}.
	\end{equation*}
	The function $G_{k,1}^{1,a}$ can be estimated simply by using Lemmas~\ref{L2_est_full} and~\ref{lem:g_h_Z}. We now consider the function $G_{k,1}^{1,c}$. We decompose the kernel 
	\begin{equation*}\label{nonlin27}
		K_{k_1,k_2,\leq k_2\ldots}(\underline{x})=K_{k_1,k_2,\leq k_2\ldots}^{1,1}(\underline{x})+K_{k_1,k_2,\leq k_2\ldots}^{1,2}(\underline{x})+K_{k_1,k_2,\leq k_2\ldots}^{1,3}(\underline{x}),
	\end{equation*} 
	where
	\begin{equation*}\label{nonlin28}
	\begin{split}
	&K^{1,1}_{k_1,k_2,\leq k_2\ldots}(\underline{z})=\int_{\R}L_{k_1}(z_1,\alpha)\widetilde{L}_{k_2}(z_2,\alpha)L_{\leq k_2-3n-1}(z_3,\alpha)\ldots L_{\leq k_2-3n-1}(z_{2n+1},\alpha)\,d\alpha,\\
	&K^{1,2}_{k_1,k_2,\leq k_2\ldots}(\underline{z})=\int_{|\alpha|\notin[2^{-k_1},2^{-k_2}]}L_{k_1}(z_1,\alpha)\frac{1}{\alpha}\min(1,2^{k_2}\alpha)\psi_{0}'(2^{k_2}z_2)L_{\leq k_2-3n-1}(z_3,\alpha)\ldots L_{\leq k_2-3n-1}(z_{2n+1},\alpha)\,d\alpha,\\
	&K^{1,3}_{k_1,k_2,\leq k_2\ldots}(\underline{z})=\int_{|\alpha|\in[2^{-k_1},2^{-k_2}]}L_{k_1}(z_1,\alpha)2^{k_2}\psi_{0}'(2^{k_2}z_2)L_{\leq k_2-3n-1}(z_3,\alpha)\ldots L_{\leq k_2-3n-1}(z_{2n+1},\alpha)\,d\alpha.
	\end{split}
	\end{equation*}
	Using Lemmas~\ref{lemma:K}-\ref{lemma:K2} we get
	\begin{equation*}\label{nonlin29}
	\|K^{1,1}_{k_1,k_2,\leq k_2\ldots}(\underline{z})\|_{L^1(\R^{2n+1})}+\|K^{1,2}_{k_1,k_2,\leq k_2\ldots}(\underline{z})\|_{L^1(\R^{2n+1})}\lesssim_n2^{-k_1}.
	\end{equation*}
	To estimate $\|K^{1,3}_{k_1,k_2,\leq k_2\ldots}(\underline{z})\|_{L^1(\R^{2n+1})}$, we now further split
	\begin{equation*}\label{nonlin30}
	K^{1,3}_{k_1,k_2,\leq k_2\ldots}(\underline{z})=K^2_{k_1,k_2,\leq k_2\ldots}(\underline{z})+K^3_{k_1,k_2,\leq k_2\ldots}(\underline{z})+\ldots+K^{2n+1}_{k_1,k_2,\leq k_2\ldots}(\underline{z}),
	\end{equation*} 
	where
	\begin{equation*}\label{nonlin31}
	\begin{split}
	K^2_{k_1,k_2,\leq k_2\ldots}(\underline{z})&=\int_{|\alpha|\in[2^{-k_1},2^{-k_2}]}L_{k_1}(z_1,\alpha)2^{k_2}\psi_{0}'(2^{k_2}z_2)\widetilde{L}_{\leq k_2-3n-1}(z_3,\alpha)\ldots L_{\leq k_2-3n-1}(z_{2n+1},\alpha)\,d\alpha,\\
	K^3_{k_1,k_2,\leq k_2\ldots}(\underline{z})&=\int_{|\alpha|\in[2^{-k_1},2^{-k_2}]}L_{k_1}(z_1,\alpha)2^{k_2}\psi_{0}'(2^{k_2}z_2)2^{k_3}\psi_{0}'(2^{k_3}z_3)\widetilde{L}_{\leq k_2-3n-1}(z_4,\alpha)\ldots L_{\leq k_2-3n-1}(z_{2n+1},\alpha)\,d\alpha,\\
	&\vdots\\
	K^{2n}_{k_1,k_2,\leq k_2\ldots}(\underline{z})&=\int_{|\alpha|\in[2^{-k_1},2^{-k_2}]}L_{k_1}(z_1,\alpha)2^{k_2}\psi_{0}'(2^{k_2}z_2)2^{k_3}\psi_{0}'(2^{k_3}z_3)\ldots 2^{k_{2n}}\psi_{0}'(2^{k_{2n}}z_{2n})\widetilde{L}_{\leq k_2-3n-1}(z_{2n+1},\alpha)\,d\alpha,\\
	K^{2n+1}_{k_1,k_2,\leq k_2\ldots}(\underline{z})&=\int_{|\alpha|\in[2^{-k_1},2^{-k_2}]}L_{k_1}(z_1,\alpha)2^{k_2}\psi_{0}'(2^{k_2}z_2)2^{k_3}\psi_{0}'(2^{k_3}z_3)\ldots 2^{k_{2n+1}}\psi_{0}'(2^{k_{2n+1}}z_{2n+1})\,d\alpha.
	\end{split}
	\end{equation*}
	Using once more Lemmas~\ref{lemma:K}-\ref{lemma:K2} we get
	\begin{equation*}\label{nonlin32}
	\|K^{2}_{k_1,k_2,\leq k_2\ldots}(\underline{z})\|_{L^1(\R^{2n+1})}+\ldots+\|K^{2n}_{k_1,k_2,\leq k_2\ldots}(\underline{z})\|_{L^1(\R^{2n+1})}\lesssim_n2^{-k_1}.
	\end{equation*}
	It remains to study $K^{2n+1}$. By expressing $L_{k_1}$ as in \eqref{eq.4.20}, and inserting it into $K^{2n+1}$, we get
	\begin{equation*}\label{nonlin33}
	\begin{split}
	K^{2n+1}_{k_1,k_2,\leq k_2\ldots}(\underline{z})=\prod_{i=2}^{2n+1}2^{k_i}\psi_{0}'(2^{k_i}z_i)\int_{|\alpha|\in[2^{-k_1},2^{-k_2}]}\frac{[-\psi_{0}(2^{k_1}(x-\alpha))]}{\alpha}\,d\alpha.
	\end{split}
	\end{equation*}
	Using the fact that
	\begin{equation*}\label{nonlin34}
	P_{k_i}f_i(x)=\int_{\R}2^{k_i}(P_{k_i}f_i)(x-y_i)\psi_{0}'(2^{k_i}y_i)\,dy_i,\qquad i\in\{1,2,\ldots,2n+1\}.
	\end{equation*}
	and the boundedness of the truncated Hilbert transform in $L^2$, we obtain the desired bounds.
	
	We now consider the function $G_{k,1}^{1,b}$. In this case, we have some frequencies which are greater than $k_2$ and some which are less. For this function, we use a hybrid argument of those used for $G_{k,1}^{1,a}$ and $G_{k,1}^{1,c}$. That is, we treat the functions with frequencies greater than $k_2$ as though they were in the $Z_2$ space and handle the functions with frequencies less than $k_2$ as we did above for $G_{k,1}^{1,c}$.
\end{proof}

It remains to estimate the high-low-to-high term for $Z_1$ inputs. As for the trilinear setting, we further decompose $G_{k,1}^1$ as follows.
\begin{equation}\label{nonlin36}
	G_{k,1}^1:=G_{k,1}^{1,1}+G_{k,1}^{1,2}+G_{k,1}^{1,3}+G_{k,1}^{1,4}
\end{equation}
where
\begin{equation*}
	\begin{aligned}
	G_{k,1}^{1,1}[\underline{h}](x):=\frac{(-1)^n}{\pi}\sum_{k_1\in[k-3,k+3]}P_k\frac{d}{dx}\int_{|\alpha|\leq 2^{-k}}&\int_{\R^{2n+1}}\frac{d}{dx}P_{k_1}h_1(x-y_1)L_{k_1}(y_1,\alpha)\\
    &\qquad\times\prod_{\ell=2}^{2n+1}P_{\leq k-6n}h_\ell(x-y_\ell)L_{\leq k-6n}(y_\ell,\alpha)\,d\alpha d\underline{y},
 	\end{aligned}
\end{equation*}
\begin{equation*}
	\begin{aligned}
	G_{k,1}^{1,2}[\underline{h}](x):=\frac{(-1)^n}{\pi}\sum_{k_1\in[k-3,k+3]}P_k\int_{|\alpha|\geq 2^{-k}}&\int_{\R^{2n+1}}\frac{d}{dx}P_{k_1}h_1(x-y_1)L_{k_1}(y_1,\alpha)\frac{d}{dx}\\
    &\qquad\times\prod_{\ell=2}^{2n+1}P_{\leq k-6n}h_\ell(x-y_\ell)L_{\leq k-6n}(y_\ell,\alpha)\,d\alpha d\underline{y},
 	\end{aligned}
\end{equation*}
\begin{equation*}
	\begin{aligned}
     G_{k,1}^{1,3}[\underline{h}](x):=\frac{(-1)^{n+1}}{\pi}\!\!\!\!\sum_{k_1\in[k-3,k+3]}\!\!\!P_k\int_{|\alpha|\geq 2^{-k}}&\int_{\R^{2n}}\frac{d}{dx}P_{k_1}h_1(x-\alpha)\\
    &\quad\times\prod_{\ell=2}^{2n+1}P_{\leq k-6n}h_\ell(x-y_\ell)\frac{L_{\leq k-6n}(y_\ell,\alpha)}{\alpha}\,d\alpha dy_2\ldots dy_{2n+1},
  	\end{aligned}
\end{equation*}
 \begin{equation*}
	\begin{aligned}
	G_{k,1}^{1,4}[\underline{h}](x):=\frac{(-1)^n}{\pi}\sum_{k_1\in[k-3,k+3]}&P_k\Big\{\frac{d}{dx}P_{k_1}h_1(x)\\
    &\times\int_{|\alpha|\geq 2^{-k}}\int_{\R^{2n+1}}\prod_{\ell=2}^{2n+1}P_{\leq k-6n}h_\ell(x-y_\ell)\frac{L_{\leq k-6n}(y_\ell,\alpha)}{\alpha}\,d\alpha dy_2\ldots dy_{2n+1}\Big\}.
	\end{aligned}
\end{equation*}
Using Lemmas~\ref{L2_est_full} and~\ref{lem:g_h_Z}, we obtain the following.
\begin{lemma}\label{lem:hl_full}
	Assume that $f_1\in Z$, $f_2,\ldots,f_{2n+1}\in Z_1$ and define $h_j(t,x):=f_j(t,x-a_j+q(t,x))$ for some points $a_j\in\R$. Then
	\begin{equation*}\label{nonlin37}
		\|G_{k,1}^{1,1}(t)\|_{L^2}+\|G_{k,1}^{1,2}(t)\|_{L^2}+\|G_{k,1}^{1,3}(t)\|_{L^2}\lesssim_n 2^{k/2}\min\{1,(2^kt)^{-1/10}\}\|f_1\|_{Z}\prod_{\ell=1}^{2n}\|f_{j_\ell}\|_{Z_1}.
	\end{equation*}
\end{lemma}
To understand the term $G_{k,1}^{1,4}$ we define, for any functions $g_1,\ldots g_{2n}\in Z_1$ and any base points $a_1,\ldots, a_{2n}\in\R$ the following multilinear expression
\begin{equation}\label{nonlin38}
	V[g_1,\ldots,g_{2n}](t,x):=\frac{(-1)^n}{\pi}\int_{|\alpha|\geq t}\prod_{\ell=1}^{2n}\frac{p_\ell(t,x,\alpha)}{\alpha}d\alpha,
\end{equation}
where the $p_\ell$ are defined as in \eqref{pbigDef}. As we did in the trilinear setting, by exploiting the structure of the $Z_1$ functions in Lemma~\ref{lemma:error}, we get the following lemma.
\begin{lemma}\label{lem:G4_full}
	Assume that $f_1\in Z$, $f_2,\ldots,f_{2n+1}\in Z_1$, and define $h_j(t,x):=f_j(t,x-a_j+q(t,x))$ for some points $a_j\in\R$. Then with $G_{k,1}^{1,4}$ defined as in \eqref{nonlin36}, we have
	\begin{equation*}\label{nonlin39}
		\Big\|G_{k,1}^{1,4}(t,x)-P_k\Big\{\frac{d}{dx}h_1(t,x)\cdot P_{\leq k-4n}V[f_2,\ldots,f_{2n+1}](t,x)\Big\}\Big\|_{L^2_x}\lesssim_n2^{k/2}(2^kt)^{-1/10}\|f_1\|_Z\prod_{\ell=2}^{2n+1}\|f_{j_\ell}\|_{Z_1}.
	\end{equation*}
\end{lemma}
Finally, we need to check that the multilinear expression $V$ satisfies all the necessary bounds. Assume that $q$ satisfies the bounds \eqref{q_bounds} and, in addition,
\begin{equation}\label{nonlin41}
\|\partial_xq(t,.)\|_{L^4_x}\lesssim\varepsilon t^{1/4}\qquad\text{ for any }t\in[0,T].
\end{equation}
Further assume that the functions $g_1,\ldots,g_{2n}\in Z_1$ satisfy the bounds
\begin{equation*}\label{nonlin42}
\sum_{\ell=1}^{2n}\|g_\ell\|_{Z_1}\lesssim_n \varepsilon^{n},\qquad \sum_{\ell=1}^{2n}\|g_\ell\|_{L^\infty_tL^2_x}\lesssim_n 1.
\end{equation*}
Finally, assume that $a_1,\ldots,a_{2n}\in\R$ are $2n$ points, and define the velocity field $V[g_1,\ldots,g_{2n}]$ as in \eqref{nonlin38}.
Then the following lemma is a straightforward generalization to $2n$ functions of Lemma~\ref{lemma:L-inf}.
\begin{lemma}\label{lem:V_full}
	With the assumptions above and $\widetilde{Q}$ defined as in \eqref{changeV}, we can decompose
	\begin{equation*}
	V[g_1,\ldots,g_{2n}](t,\widetilde{Q}(t,y))=V_1[g_1,\ldots,g_{2n}](t,y)+V_2[g_1,\ldots,g_{2n}](t,y),
	\end{equation*}
	where
	\begin{equation}\label{nonlin40}
	\begin{split}
	&|V_1[g_1,\ldots,g_{2n}](t,y)|\lesssim_n \varepsilon^{2n}\log (2/t)\mathbf{1}_{[0,2^{-20}]}(t)\sum_{j\in\{1,\ldots,2n\}}\mathbf{1}_{[a_j-2^{-10},a_j+2^{-10}]}(y),\\
	&|\partial_y^mV_1[g_1,\ldots,g_{2n}](t,y)|\lesssim_n\varepsilon^{2n}\sum_{j\in\{1,\ldots,2n\},\,|y-a_j|\geq t}\frac{1}{|y-a_j|^n}\ln\Big(\frac{2|y-a_j|}{t}\Big),
	\end{split}
	\end{equation}
	for $m\in\{1,2,3\}$, and
	\begin{equation}\label{nonlin45}
	|V_2[g_1,\ldots,g_{2n}](t,y)|\lesssim_n \varepsilon^{2n}.
	\end{equation}
	The sum in the second line of \eqref{nonlin40} is taken over all the indices $j$ with the property that $|y-a_j|\geq t$ (so it vanishes if $|y-a_j|\leq t$ for all $j\in\{1,\ldots,2n\}$).
\end{lemma}

By integrating the bounds \eqref{nonlin40} for $s\in[0,t]$, an analogue of Corollary~\ref{lemma:L-inf2} for the full nonlinearity is now immediate.
\begin{corollary}\label{cor:full}
	With the assumptions and the notation of Lemma \ref{lem:V_full} above, let
	\begin{equation*}
	W[g_2,\ldots,g_{2n+1}](t,y):=\int_0^tV_1[g_2,\ldots,g_{2n+1}](s,y)\,ds.
	\end{equation*}
	Then, for $m\in\{1,2,3\}$,
	\begin{equation*}
	\begin{split}
	&|W[g_2,\ldots,g_{2n+1}](t,y)|\lesssim \varepsilon^2t\log (2/t)\sum_{j\in\mathcal{J}}\mathbf{1}_{[a_j-2^{-10},a_j+2^{-10}]}(y),\\
	&|\partial_y^mW[g_2,\ldots,g_{2n+1}](t,y)|\lesssim\varepsilon^2\sum_{j\in\mathcal{J},\,|y-a_j|\geq t}\frac{t}{|y-a_j|^m}\ln\Big(\frac{2|y-a_j|}{t}\Big)+\varepsilon^2\sum_{j\in\mathcal{J},\,|y-a_j|\leq t}\frac{1}{|y-a_j|^{m-1}}.
	\end{split}
	\end{equation*}
  In particular,
 \begin{equation*}
     \|\partial_y W[g_{2},\ldots,g_{2n+1}](t,y)\|_{L_x^4}\lesssim\ep t^{1/4}.
 \end{equation*}
\end{corollary}

Finally, for the purposes of the fixed point argument in the next section, we provide the extension of Lemma~\ref{lem:V1pk} to the full nonlinearity.
\begin{lemma}\label{lem:V1pk_full}
	Let $V_1[f_2,\ldots,f_{2n+1}](t,y)$ be defined as
 \begin{equation*}\label{V1_full}
     V_1[g_2,\ldots,g_{2n+1}](t,y):=-\frac{1}{\pi}\int_{|\alpha|\in[t,2^{-20}]}\frac{\prod_{\ell=2}^{2n+1}r_{\ell}(t,y,\alpha)}{\alpha}d\alpha,
 \end{equation*} 
 where $r_\ell$ is defined as in \eqref{lpo26}, with $(f_2,\ldots,f_{2n+1})\in Z_1\times\ldots\times Z_1$. Then, for $t<2^{-20}$,
	\begin{equation*}\label{V1_bounds_full}
	\big|V_1[P_kf_2,\ldots,f_{2n+1}](t,y) \big|\lesssim(2^kt)^{-1/10}\prod_{\ell=2}^{2n+1}\|f_{\ell}\|_{Z_1}.
	\end{equation*}
\end{lemma}	

	\section{Proof of the main result}\label{CloseProof}

	\subsection{Determining $\tilde{q}$ from the free evolution}
	We define the free evolution for initial data $g_{0,j}(x)\in Z_1$ for all $j\in\mathcal{J}=\{1,\ldots,M\}$ as
	\begin{equation*}
		g_j^{(0)}(t,x):=e^{-t|\nabla|}g_{0,j}(x).
	\end{equation*}
	Moreover, we define
	\begin{equation*}\label{free_V1}
		V_1^*[g_{{j,2}}^{(0)},\ldots,g_{{j,2n+1}}^{(0)}](t,x):=-\frac{1}{\pi}\int_{|\alpha|\in [t,2^{-20}]}\frac{\prod_{i=2}^{2n+1}r_{j,i}^{(0)}(t,x,\alpha)}{\alpha}d\alpha,
	\end{equation*}
	where we define, as in \eqref{lpo25}-\eqref{lpo26}, for $j\in\mathcal{J}$,
	\begin{equation*}\label{lpo261}
	\begin{split}
	r_{j}^{(0)}(t,x,\alpha):=g_j^{*(0)}(t,0,\alpha)\varphi_{\leq -4}((x-a_{j})/\alpha)+g_{j}^{(0)}(t,x-a_{j})\varphi_{\leq -4}(\alpha/(x-a_j)),
	\end{split}
	\end{equation*}
	where we recall that
\begin{equation*}\label{lpo251}
\begin{aligned}
g_j^{*(0)}(t,0,\alpha):=\frac{1}{\alpha}\int_{0}^\alpha g_{j}^{(0)}(t,-\rho)\,d\rho.
\end{aligned}
\end{equation*}
We can now define for $j\in\mathcal{J}$
\begin{equation}\label{real_qtilde}
	\tilde{q}(t,x):=\sum_{j\in\mathcal{J}}\int_{0}^{t}V_1^*[g_{{j,2}}^{(0)},\ldots, g_{{j,2n+1}}^{(0)}](s,x)\,ds.
\end{equation}
From Corollary~\ref{cor:full}, we see that $\tilde{q}$ satisfies the bounds we assumed in \eqref{q_bounds}-\eqref{qtilde_bounds} and \eqref{nonlin41}.

\subsection{Construction of the solution}
	
	We can set up a fixed point argument using equation \eqref{ga_eq} and Duhamel's formula. More precisely, we define 
	\begin{equation}\label{duhamel}
	\begin{aligned}
	g_j^{(m+1)}(t,x):=e^{-t|\nabla|}g_{0,j}(x)+\int_0^t e^{-(t-s)|\nabla|}F_j^{(m)}(s,x)ds,
	\end{aligned}
	\end{equation}
	where
	\begin{equation}\label{fn1}
	\begin{aligned}
	F_j^{(m)}(t,x+a_j)&=\partial_{x}h_j^{(m)}(t,\tilde{Q}(t,x+a_j))\partial_t\tilde{q}(t,x+\!a_j)+\mathcal{N}_{h_j}^{(m)}(t,\tilde{Q}(t,x+a_j))\\
	&\quad-|\nabla|g_j^{(m)}(t,x)\partial_{x}q(t,\tilde{Q}(t,x+a_j))+E_j^{(m)}(t,\tilde{Q}(t,x+a_j)),
	\end{aligned}
	\end{equation}
 and $E$ was defined in Lemma \eqref{lem:ga_eq}.
 The goal is to show that
	\begin{equation*}\label{constr1}
		\|g_j^{(m+1)}(t,x)-g_j^{(m)}(t,x)\|_{Z_2}\lesssim\ep^2\|g_j^{(m)}(t,x)-g_j^{(m-1)}(t,x)\|_{Z_2},
	\end{equation*}
	for all $m\in\mathbb{N}$. We begin by proving that for all $m\in\mathbb{N}$, the functions $g_j^{(m)}$ are small in the $Z$ norm.
	
	\begin{lemma}\label{lem:g_j_small}
		For all $g_j^{(m)}\in Z$, $j\in\mathcal{J}$ and $m\in\mathbb{N}$, we have 
		\begin{equation}\label{small_g}
			\|g_j^{(m)}-g_{j}^{(0)}\|_{Z_2}\lesssim\ep.
		\end{equation}
	\end{lemma} 
	\begin{proof}
		We argue by induction: by hypothesis, \eqref{small_g} clearly holds for the free evolution $g_j^{(0)}$. That is, there exists some constant $C_0$ such that $\|g_j^{(0)}\|_Z\leq C_0\ep$. We now assume that \eqref{small_g} holds for some given $m$. This implies that $\|g_j^{(m)}\|_{Z}\leq2C_0\ep$. To show that the inequality holds for $m+1$, we note from Duhamel's formula \eqref{duhamel}, that
		\begin{equation*}\label{claim9}
			\begin{split}
			\|g_j^{(m+1)}-g_j^{(0)}\|_{Z_2}&\leq \bigg\|\int_{0}^{t}e^{-(t-s)|\nabla|}F_j^{(m)}(s,x)\,ds\bigg\|_{Z_2}.
			\end{split}
		\end{equation*}
	We split the nonlinearity as in Sections~\ref{sec:G1G2} and~\ref{sect:full} to get
	\begin{equation}\label{fn}
	\begin{split}
F_j^{(m)}(t,x)&=\sum_{n\in\mathbb N}\sum_{k\in\mathbb{Z}}\,\sum_{j_2,\ldots,j_{2n+1}\in\mathcal{J}}\Big(G_{k,2}[h_{1,j}^{(m)},h_{2,j_2}^{(m)},\ldots,h_{2n+1,j_{2n+1}}^{(m)}](t,\tilde{Q}(t,x+a_j))\\&\quad+\sum_{i\in\{1,2,3,4\}}G_{k,1}^i[h_{1,j}^{(m)},h_{2,j_2}^{(m)},\ldots,h_{2n+1,j_{2n+1}}^{(m)}](t,\tilde{Q}(t,x+a_j))\Big)\\
&\quad-|\nabla|g_j^{(m)}(t,x)\partial_{x}q(t,\tilde{Q}(t,x+a_j))+\text{E}_j^{(m)}(t,\tilde{Q}(t,x+a_j)).
	\end{split}		
	\end{equation}
	The terms in the second line of \eqref{fn} are dealt with in Lemmas~\ref{lem:ga_eq} and \ref{lem:err2}. Moreover, $G_{k,2}$ and $G_{k,1}^i$ for $i\in\{1,2,3\}$ are bounded in a straightforward way using Lemmas~\ref{lem:hh_full}-- \ref{lem:hl_full}, combined with Lemma~\ref{lem:g_h} to take into account the change of variables. 
	
	It remains to consider $G_{k,1}^4$. From \eqref{fn1} and \eqref{real_qtilde}, we see that we need to find bounds on $$G_{k,1}^4-P_k\Big\{\frac{d}{dx}h_{1,j}^{(m)}\cdot\sum_{j_2,\ldots,j_{2n+1}\in\mathcal{J}} V_1^*[h_{2,j_2}^{(0)},\ldots,h_{2n+1,j_{2n+1}}^{(0)}]\Big\}.$$ We begin by defining the renormalized free evolution by
	\begin{equation*}
		h^{(0)}(t,x):=\sum_{j\in\mathcal{J}}g_j^{(0)}(t,x-a_j+q(t,x)),
	\end{equation*}
	and rewrite
	\begin{equation}\label{G4}
		\begin{split}
		G_{k,1}^{1,4}[h_{1,j}^{(m)},h_2^{(m)},\ldots,h_{2n+1}^{(m)}](t,x)&=G_{k,1}^{1,4}[h_{1,j}^{(m)},h_2^{(m)},\ldots,h_{2n+1}^{(m)}](t,x)-G_{k,1}^{1,4}[h_{1,j}^{(m)},h_2^{(0)},\ldots,h_{2n+1}^{(0)}](t,x)\\&\qquad+G_{k,1}^{1,4}[h_{1,j}^{(m)},h_2^{(0)},\ldots,h_{2n+1}^{(0)}](t,x)\\&=
		G_{k,1}^{1,4}[h_{1,j}^{(m)},h_2^{(m)}-h_2^{(0)},\ldots,h_{2n+1}^{(m)}](t,x)+\ldots+\\&\qquad+G_{k,1}^{1,4}[h_{1,j}^{(m)},h_2^{(0)},\ldots,h_{2n+1}^{(m)}-h_{2n+1}^{(0)}](t,x)\\&\qquad+G_{k,1}^{1,4}[h_{1,j}^{(m)},h_2^{(0)},\ldots,h_{2n+1}^{(0)}](t,x).
		\end{split}
	\end{equation}
	We notice from Lemma~\ref{lem:g_h_Z} that $h_i^{(m)}-h_i^{(0)}$ for $i\in\{2,\ldots,2n+1\}$ are (renormalized) elements in $Z_2$, and therefore the first $2n$ terms in the second equality of \eqref{G4} can be bounded by Lemma~\ref{lem:mix_full}. For the last term, we rewrite 
	\begin{equation}\label{G42}
		\begin{split}
	G_{k,1}^{1,4}[h_{1,j}^{(m)},h_2^{(0)},\ldots,h_{2n+1}^{(0)}](t,x)=&G_{k,1}^{1,4}[h_{1,j}^{(m)},h_2^{(0)},\ldots,h_{2n+1}^{(0)}](t,x)\\&-\sum_{j_2,\ldots,j_{2n+1}\in\mathcal{J}}P_{k}\Big\{\frac{d}{dx}h_{1,j}^{(m)}(t,x)\cdot P_{\leq k-4}V[g_{2,j_2}^{(0)},\ldots,g_{2n+1,j_{2n+1}}^{(0)}](t,x)\Big\}\\&+
	\sum_{j_2,\ldots,j_{2n+1}\in\mathcal{J}}P_{k}\Big\{\frac{d}{dx}h_{1,j}^{(m)}(t,x)\cdot P_{\leq k-4}V_2[g_{2,j_2}^{(0)},\ldots,g_{2n+1,j_{2n+1}}^{(0)}](t,x)\Big\}\\&+
	\sum_{j_2,\ldots,j_{2n+1}\in\mathcal{J}}P_{k}\Big\{\frac{d}{dx}h_{1,j}^{(m)}(t,x)\cdot P_{\leq k-4}V_1[g_{2,j_2}^{(0)},\ldots,g_{2n+1,j_{2n+1}}^{(0)}](t,x)\Big\},
		\end{split}
	\end{equation}
	where $V_1$ and $V_2$ are defined as in Lemma~\ref{lemma:L-inf}. The first term in \eqref{G42} can be bounded using Lemma~\ref{lem:G4_full}. Using \eqref{nonlin45} in Lemma~\ref{lem:V_full}, we estimate
	\begin{equation*}
    \begin{split}
		\bigg\|\sum_{j_2,\ldots,j_{2n+1}\in\{1,\ldots,M\}}P_{k}\Big\{\frac{d}{dx}h_{1,j}^{(m)}(t,x)\cdot P_{\leq  k-4}&V_2[g_{2,j_2}^{(0)},\ldots,g_{2n+1,j_{2n+1}}^{(0)}](t,x)\Big\}\bigg\|_{L^2}\\&\lesssim_n (1+2^kt)^{-1/10}2^{k/2}\|g_{1,j}^{(m)}\|_{Z}\prod_{i=2}^{2n+1}\sum_{j_i\in\mathcal{J}}\|g_{j_i}^{(m)}\|_{Z_1}.
    \end{split}
	\end{equation*}
	Finally, we notice that in the last term of \eqref{G42} can be rewritten as 
	\begin{equation}\label{G43}
	\begin{split}
	&\sum_{j_2,\ldots,j_{2n+1}\in\mathcal{J}}P_{k}\Big\{\frac{d}{dx}h_{1,j}^{(m)}(t,x)\cdot P_{\leq k-4}V_1[g_{2,j_2}^{(0)},\ldots,g_{2n+1,j_{2n+1}}^{(0)}](t,x)\Big\}=\\&\qquad
	\sum_{j_2,\ldots,j_{2n+1}\in\mathcal{J}}P_{k}\Big\{\frac{d}{dx}h_{1,j}^{(m)}(t,x)\cdot V_1^*[g_{2,j_2}^{(0)},\ldots,g_{2n+1,j_{2n+1}}^{(0)}](t,x)\Big\}\\&\quad-
	\sum_{j_2,\ldots,j_{2n+1}\in\mathcal{J}}P_{k}\Big\{\frac{d}{dx}h_{1,j}^{(m)}(t,x)\cdot P_{\geq  k-3}V_1[g_{2,j_2}^{(0)},\ldots,g_{2n+1,j_{2n+1}}^{(0)}](t,x)\Big\}.
	\end{split}
	\end{equation}
	Let us now consider the last term in \eqref{G43}, remarking that it is a high-high-to-low interaction. Applying Lemma~\ref{lem:V1pk_full} then yields  
	\begin{equation*}
    \begin{split}
		\bigg\|\sum_{j_2,\ldots,j_{2n+1}\in\mathcal{J}}P_{k}\Big\{\frac{d}{dx}h_{1,j}^{(m)}(t,x)\cdot P_{\geq  k-3}&V_1[g_{2,j_2}^{(0)},\ldots,g_{2n+1,j_{2n+1}}^{(0)}](t,x)\Big\}\bigg\|_{L^2}\\&\lesssim_n2^{k/2}(2^kt)^{-1/10}\|g_{1,j}^{(m)}\|_Z\prod_{i=2}^{2n+1}\sum_{j_i\in\mathcal{J}}\|g_{j_i}^{(m)}\|_{Z_1}.
    \end{split}
	\end{equation*}
	
	Combining all the above yields 
	\begin{equation*}
	\begin{split}
	\bigg\|G_{k,1}^{1,4}(t,x)-\sum_{j_2,\ldots,j_{2n+1}\in\mathcal{J}}P_{k}\Big\{\frac{d}{dx}h_{1,j}^{(m)}(t,x)\cdot &V_1^*[g_{2,j_2}^{(0)},\ldots,g_{2n+1,j_{2n+1}}^{(0)}](t,x)\Big\}\bigg\|_{L_x^2}\\&\lesssim_n2^{k/2}(2^kt)^{-1/10}\|g_{j,1}^{(m)}\|_{Z}\prod_{i=2}^{2n+1}\sum_{j_i\in\mathcal{J}}\|g_{j_i}^{(m)}\|_{Z_1}.
	\end{split}
	\end{equation*}
	Combining all the estimates and using \eqref{N1} from Lemma~\ref{lem:g_h_Z}, we have that
	\begin{equation}\label{Fka_bound}
	\begin{aligned}
	\|P_{k'} F_j^{(m)}\|_{L^2}&\lesssim \sum_{n\in\mathbb N}C^n2^{k/2}(2^k t)^{-1/10}\|g_{j}^{(m)}\|_Z\prod_{i=2}^{2n+1}\sum_{j_i\in\mathcal{J}}\|g_{j_i}^{(m)}\|_{Z}+\varepsilon2^{k/2}(2^kt)^{-1/10}\|g_j\|_{Z}
	\end{aligned}
	\end{equation}
    where we picked up the constant $C^n$ from the $\lesssim_n$.
    Using the induction assumption $\|g_j^{(m)}\|_Z\leq2C_0\ep$, we recover a factor of $\ep^{2n+1}$ in the first term on the left hand side, which, by choosing $\ep$ to be sufficiently small, will trump the constant $C^n$. From the sum over $n$ we pick the biggest case, being $n=1$, and we can thus rewrite \eqref{Fka_bound} as
    \begin{equation*}\label{Fka_bound2}
   \|P_{k'} F_j^{(m)}\|_{L^2}\lesssim \ep^2 2^{k/2}(2^k t)^{-1/10}.    
    \end{equation*}
	We now claim that 
	\begin{equation}\label{claim10}
	\|\mathcal{C}(F_j^{(m)})\|_{Z_2}\lesssim\ep^2,
	\end{equation}
	where 
	\begin{equation*}
	\mathcal{C}(F_j^{(m)})(t,x):=\int_0^t e^{-(t-s)|\nabla|}F_j^{(m)}(s,x)\,ds.
	\end{equation*}
	
	We assume that $t\in(0,\infty)$ and $k\in\mathbb{Z}$ are fixed. From the estimate \eqref{Fka_bound} we immediately get that
	\begin{align*}		
	2^{k/2}(1+2^k t)^{2/10}&(2^k t)^{-1/10}\|P_k\mathcal{C}(F_j^{(m)})(t,x)\|_{L^2}\\
	&\lesssim\ep^2\int_0^t2^{k/2}(1+2^k t )^{2/10}(2^k t)^{-1/10}e^{-(t-s)2^k/8}\|P_k F_{j}^{(m)}(s,x)\|_{L^2}\,ds\\
	&\lesssim\ep^2\int_0^t2^k\frac{(1+2^k t)^{2/10}}{(2^k s)^{1/10}}\frac{(2^k t)^{-1/10}}{(1+(t-s)2^k)^{8}}\,ds\\
	&\lesssim \ep^2\min\{1,(2^kt)^{8/10}\}.
	\end{align*}
	This concludes the proof of the claim \eqref{claim10}, thus concluding the proof.
	\end{proof}
	
	\begin{lemma}
		For all $g_j^{(m)}\in Z$, $j\in\mathcal{J}$ and $m\in\mathbb{N},$ we have
		\begin{equation}\label{constr2}
			\|g_j^{(m+1)}(t,x)-g_j^{(m)}(t,x)\|_{Z_2}\lesssim\ep\|g_j^{(m)}(t,x)-g_j^{(m-1)}(t,x)\|_{Z_2}.
		\end{equation}
	\end{lemma}
	\begin{proof}
		Again, we argue by induction. By hypothesis and by Lemma~\ref{lem:g_j_small}, \eqref{constr2} clearly holds for the free evolution $g_j^{(0)}$. We now assume that \eqref{constr2} holds for some given $m$. That is, we have
		\begin{equation*}\label{constr3}
		\|g_j^{(m)}(t,x)-g_j^{(m-1)}(t,x)\|_{Z_2}\lesssim\ep\|g_j^{(m-1)}(t,x)-g_j^{(m-2)}(t,x)\|_{Z_2}.
		\end{equation*}
		Using Duhamel's formula \eqref{duhamel}, we have
		\begin{equation*}\label{constr4}
		\|g_j^{(m+1)}-g_j^{(m)}\|_{Z_2}\leq \bigg\|\int_{0}^{t}e^{-(t-s)|\nabla|}\big[F_j^{(m)}(s,x)-F_j^{(m-1)}(s,x) \big]\,ds\bigg\|_{Z_2}.
		\end{equation*}
		We begin by considering the nonlinearity $F_j^{(m)}-F_j^{(m-1)}$ which we split as in Sections~\ref{sec:G1G2} and~\ref{sect:full} to get
		\begin{equation*}\label{constr5}
			\begin{split}
			(F_j^{(m)}-F_j^{(m-1)})(t,x)=\sum_{n\in\mathbb N}\sum_{k\in\mathbb{Z}}\,\sum_{j_2,\ldots,j_{2n+1}\in\mathcal{J}}&\Big(G_{k,2}[h_{1,j}^{(m)},h_{2,j_2}^{(m)},\ldots h_{2n+1,j_{2n+1}}^{(m)}](t,\tilde{Q}(t,x+a_j))\\&-G_{k,2}[h_{1,j}^{(m-1)},h_{2,j_2}^{(m-1)},\ldots,h_{2n+1,j_{2n+1}}^{(m-1)}](t,\tilde{Q}(t,x+a_j))\\&+\sum_{i=1}^4\Big[G_{k,1}^i[h_{1,j}^{(m)},h_{2,j_2}^{(m)},\ldots h_{2n+1,j_{2n+1}}^{(m)}](t,\tilde{Q}(t,x+a_j))\\&-G_{k,1}^i[h_{1,j}^{(m-1)},h_{2,j_2}^{(m-1)},\ldots, h_{2n+1,j_{2n+1}}^{(m-1)}](t,\tilde{Q}(t,x+a_j))\Big]\Big)\\
			&-|\nabla|g_j^{(m)}(t,x)\partial_{x}q(t,\tilde{Q}(t,x+a_j))-|\nabla|g_j^{(m-1)}(t,x)\partial_{x}q(t,\tilde{Q}(t,x+a_j))\\&+\text{E}_j^{(m)}(t,\tilde{Q}(t,x+a_j))-\text{E}_j^{(m-1)}(t,\tilde{Q}(t,x+a_j)).
			\end{split}
		\end{equation*}
	We remark that all the difference of the form $g_j^{(m)}-g_j^{(m-1)}$ are elements in the $Z_2$ space. From Lemmas~\ref{lem:ga_eq} and \ref{lem:err2}, we easily get
	\begin{equation*}
		\begin{split}
		&\|P_k\big[E_j^{(m)}-E_j^{(m-1)}\big](t,\tilde{Q}(t,x+a_j))\|_{L^2}\lesssim\ep2^{k/2}\min\{1,(2^kt)^{-1/10}\}\|g_j^{(m)}-g_j^{(m-1)}\|_{Z_2},\\
		&\|P_k\big[|\nabla|(g_j^{(m)}-g_j^{(m-1)})(t,x)\partial_xq(t,\tilde{Q}(t,x+a_j)) \big]\|_{L^2}\lesssim\ep2^{k/2}(2^kt)^{-1/10}\|g_j^{(m)}-g_j^{(m-1)}\|_{Z_2}
		\end{split}
	\end{equation*}	
	Using the fact that a difference of multilinear products can be rewritten as
	\begin{equation*}
		G(f_1,\ldots,f_{2n+1})-G(f_1^*,\ldots,f_{2n+1}^*)=G(f_1-f_1^*,\ldots,f_{2n+1})+\ldots+G(f_1^*,\ldots,f_{2n+1}-f_{2n+1}^*)
	\end{equation*}
	we rewrite the high-high-to-low interaction difference as
	\begin{equation}\label{G2}
		\begin{split}
		&\sum_{j_2,\ldots,j_{2n+1}\in\mathcal{J}}\|\big(G_{k,2}[h_{1,j}^{(m)},h_{2,j_2}^{(m)},\ldots, h_{2n+1,j_{2n+1}}^{(m)}]-G_{k,2}[h_{1,j}^{(m-1)},h_{2,j_2}^{(m-1)},\ldots,h_{2n+1,j_{2n+1}}^{(m-1)}]\big)(t,\widetilde Q(t,x+a_j))\|_{L^2}\\&\lesssim_n\quad\sum_{j_2,\ldots,j_{2n+1}\in\mathcal{J}}\|G_{k,2}[h_{1,j}^{(m)}-h_{1,j}^{(m-1)},h_{2,j_2}^{(m)},\ldots,h_{2n+1,j_{2n+1}}^{(m)}](t,x)\|_{L^2}+\ldots+\\&\quad+\sum_{j_2,\ldots,j_{2n+1}\in\mathcal{J}}\|G_{k,2}[h_{1,j}^{(m-1)},h_{2,a}^{(m-1)},\ldots,h_{2n+1,j_{2n+1}}^{(m)}-h_{2n+1,j_{2n+1}}^{(m-1)}](t,x)\|_{L^2}\\
		&\lesssim_n2^{k/2}(1+2^kt)^{-2/10}\sum_{j_2,\ldots,j_{2n+1}\in\mathcal{J}}\Big[\|g_{1,j}^{(m)}-g_{1,j}^{(m-1)}\|_{Z_2}\prod_{i=2}^{2n+1}\|g_{i,j_i}^{(m)}\|_{Z_2}+\ldots+\\ &\qquad\qquad\qquad\qquad\qquad\qquad\qquad\qquad+\|g_{2n+1,j_{2n+1}}^{(m)}-g_{2n+1,j_{2n+1}}^{(m-1)}\|_{Z_2}\|g_{1,j}^{(m-1)}\|_{Z_2}\prod_{i=2}^{2n-1}\|g_{i,j_i}^{(m-1)}\|_{Z_2} \Big]\\
		&\lesssim_n\ep^{2n}2^{k/2}(1+2^kt)^{-2/10}\|g_{j}^{(m)}-g_{j}^{(m-1)}\|_{Z_2}
		\end{split}
	\end{equation}	
	where we used Lemma~\ref{lem:g_h} to take into account the change of variables, the estimate in the third inequality follows from Lemma~\ref{lem:hh_full} and the estimate in the last line follows from Lemma~\ref{lem:g_j_small}. The estimates for $G_{k,1}^{i}$ with $i\in\{1,2,3\}$ follow similarly, using Lemmas~\ref{lem:der_full}-\ref{lem:hl_full}. It remains to consider $G_{k,1}^4$. We rewrite
	\begin{equation}\label{G44}
		\begin{split}
		&\sum_{j_2,\ldots,j_{2n+1}\in\mathcal{J}}G_{k,1}^{1,4}[h_{1,j}^{(m)},h_{2,j_2}^{(m)},\ldots,h_{2n+1,j_{2n+1}}^{(m)}]-\sum_{j_2,\ldots,j_{2n+1}\in\{1,\ldots,M\}}G_{k,1}^{1,4}[h_{1,j}^{(m-1)},h_{2,j_2}^{(m-1)},\ldots,h_{2n+1,j_{2n+1}}^{(m-1)}]\\&=\sum_{j_2,\ldots,j_{2n+1}\in\mathcal{J}}\Big\{G_{k,1}^{1,4}[h_{1,j}^{(m)}-h_{1,j}^{(m-1)},h_{2,j_2}^{(0)},\ldots,h_{2n+1,j_{2n+1}}^{(0)}]\\&\qquad\qquad\qquad\qquad+G_{k,1}^{1,4}[h_{1,j}^{(m-1)},h_{2,j_2}^{(m)}-h_{2,j_2}^{(m-1)},\ldots,h_{2n+1,j_{2n+1}}^{(m)}]+\ldots+\\&\qquad\qquad\qquad\qquad+G_{k,1}^{1,4}[h_{1,j}^{(m-1)},h_{2,j_2}^{(m-1)},\ldots,h_{2n+1,j_{2n+1}}^{(m)}-h_{2n+1,j_{2n+1}}^{(m-1)}]\\&\qquad\qquad\qquad\qquad+G_{k,1}^{1,4}[h_{1,j}^{(m)}-h_{1,j}^{(m-1)},h_{2,j_2}^{(m)}-h_{2,j_2}^{(0)},\ldots, h_{2n+1,j_{2n+1}}^{(m)}]+\ldots+\\&\qquad\qquad\qquad\qquad+G_{k,1}^{1,4}[h_{1,j}^{(m)}-h_{1,j}^{(m-1)},h_{2,j_2}^{(0)},\ldots,h_{2n+1,j_{2n+1}}^{(m)}-h_{2n+1,j_{2n+1}}^{(0)}]\Big\}.
		\end{split}
	\end{equation}
	We can now further decompose and bound \eqref{G44} similarly as in Lemma~\ref{lem:g_j_small} to get bounds as in \eqref{G2}.
	
	Combining all estimates and using \eqref{N1} and \eqref{small_g}, we get
	\begin{equation*}
		\|P_{k'}F_j^{(m)}-P_{k'}F_j^{(m-1)}\|_{L^2}\lesssim\sum_{n\in\mathbb N}C^n\ep^{2n}2^{k/2}(2^kt)^{-1/10}\|g_j^{(m)}-g_j^{(m-1)}\|_{Z_2}+\varepsilon2^{k/2}(2^kt)^{-1/10}\|g_j^{(m)}-g_j^{(m-1)}\|_{Z},
	\end{equation*}
    where here again, we get the factor of $C^n$ from the $\lesssim_n$ which gets trumped by the factor of $\ep^{2n+1}$ we pick up from Lemma~\ref{lem:g_j_small} and the induction assumption.
    
	Arguing exactly as in the end of the proof of Lemma~\ref{lem:g_j_small}, we then obtain
	\begin{equation*}
		\|\mathcal{C}(F_j^{(m)}-F_j^{(m-1)})\|_{Z_2}\lesssim\ep\|g_j^{(m)}-g_j^{(m-1)}\|_{Z_2},
	\end{equation*}
	where here
	\begin{equation*}
		\mathcal{C}(F_j^{(m)}-F_j^{(m-1)})(t,x):=\int_{0}^te^{-(t-s)|\nabla|}\big[F_j^{(m)}-F_j^{(m-1)}\big](s,x)\,ds,
	\end{equation*}
	thus concluding the proof of the lemma.
	\end{proof}
	We have now shown that $g_j^{(m)}-g_j^{(0)}$ is a Cauchy sequence in $Z_2$,  thus concluding the construction of the solution $g_j-g_j^{(0)}$ in $Z_2$.

\section*{Acknowledgements}

E.G-J. and J.G.-S. were partially supported by the AGAUR project 2021-SGR-0087 (Catalunya), by MICINN (Spain) research grant number PID2021-125021NA-I00, and by the European Research Council through ERC-StG-852741-CAPA.
S.V.H is partially supported by the National Science Foundation through the award DMS-2102961. B.P. is supported by NSF grant DMS-2154162. This work is supported by the Spanish State Research Agency, through the
Severo Ochoa and Mar\'ia de Maeztu Program for Centers and Units of Excellence in R\&D (CEX2020-001084-M), and by the European Union’s Horizon 2020 research and innovation program under
the Marie Sklodowska-Curie grant agreement CAMINFLOW No 101031111.
 
\bibliographystyle{abbrv}
\bibliography{references}	

\def\cprime{$'$}
\begin{thebibliography}{10}

\bibitem{Abels-Matioc:muskat-lp-subcritical}
H.~Abels and B.-V. Matioc.
\newblock Well-posedness of the {M}uskat problem in subcritical
  {$L_p$}-{S}obolev spaces.
\newblock {\em European J. Appl. Math.}, 33(2):224--266, 2022.

\bibitem{Agrawal-Patel-Wu:corners-one-phase-muskat}
S.~Agrawal, N.~Patel, and S.~Wu.
\newblock Rigidity of acute angled corners for one phase {M}uskat interfaces.
\newblock {\em Adv. Math.}, 412:Paper No. 108801, 71, 2023.

\bibitem{Alazard-Lazar:paralinearization-muskat}
T.~Alazard and O.~Lazar.
\newblock Paralinearization of the {M}uskat equation and application to the
  {C}auchy problem.
\newblock {\em Arch. Ration. Mech. Anal.}, 237(2):545--583, 2020.

\bibitem{Alazard-QHNguyen:cauchy-muskat-ii-critical}
T.~Alazard and Q.-H. Nguyen.
\newblock On the {C}auchy problem for the {M}uskat equation. {II}: {C}ritical
  initial data.
\newblock {\em Ann. PDE}, 7(1):Paper No. 7, 25, 2021.

\bibitem{Alazard-QHNguyen:muskat-non-lipschitz}
T.~Alazard and Q.-H. Nguyen.
\newblock On the {C}auchy problem for the {M}uskat equation with
  non-{L}ipschitz initial data.
\newblock {\em Comm. Partial Differential Equations}, 46(11):2171--2212, 2021.

\bibitem{Alazard-QHNguyen:muskat-3d}
T.~Alazard and Q.-H. Nguyen.
\newblock Quasilinearization of the 3{D} {M}uskat equation, and applications to
  the critical {C}auchy problem.
\newblock {\em Adv. Math.}, 399:Paper No. 108278, 52, 2022.

\bibitem{Alazard-QHNguyen:muskat-endpoint}
T.~Alazard and Q.-H. Nguyen.
\newblock Endpoint {S}obolev {T}heory for the {M}uskat {E}quation.
\newblock {\em Comm. Math. Phys.}, 397(3):1043--1102, 2023.

\bibitem{Ambrose:well-posedness-hele-shaw}
D.~M. Ambrose.
\newblock Well-posedness of two-phase {H}ele-{S}haw flow without surface
  tension.
\newblock {\em European J. Appl. Math.}, 15(5):597--607, 2004.

\bibitem{Bazaliy-Vasylyeva:two-phase-hele-shaw}
B.~V. Bazaliy and N.~Vasylyeva.
\newblock The two-phase {H}ele-{S}haw problem with a nonregular initial
  interface and without surface tension.
\newblock {\em J. Math. Phys. Anal. Geom.}, 10(1):3--43, 152, 155, 2014.

\bibitem{Cameron:global-wellposedness-muskat-slope-less-1}
S.~Cameron.
\newblock Global well-posedness for the two-dimensional {M}uskat problem with
  slope less than 1.
\newblock {\em Anal. PDE}, 12(4):997--1022, 2019.

\bibitem{Cameron:global-muskat-3D}
S.~Cameron.
\newblock Global wellposedness for the 3{D} {M}uskat problem with medium size
  slope.
\newblock {\em Preprint, arXiv:2002.00508}, 2020.

\bibitem{Castro-Cordoba-Fefferman-Gancedo:breakdown-muskat}
{\'A}.~Castro, D.~C{\'o}rdoba, C.~Fefferman, and F.~Gancedo.
\newblock Breakdown of {S}moothness for the {M}uskat {P}roblem.
\newblock {\em Arch. Ration. Mech. Anal.}, 208(3):805--909, 2013.

\bibitem{Castro-Cordoba-Fefferman-Gancedo:splash-singularities-muskat}
A.~Castro, D.~C\'ordoba, C.~Fefferman, and F.~Gancedo.
\newblock Splash singularities for the one-phase {M}uskat problem in stable
  regimes.
\newblock {\em Arch. Ration. Mech. Anal.}, 222(1):213--243, 2016.

\bibitem{Castro-Cordoba-Fefferman-Gancedo-LopezFernandez:rayleigh-taylor-breakdown}
{\'A}.~Castro, D.~C{\'o}rdoba, C.~Fefferman, F.~Gancedo, and
  M.~L{\'o}pez-Fern{\'a}ndez.
\newblock Rayleigh-{T}aylor breakdown for the {M}uskat problem with
  applications to water waves.
\newblock {\em Ann. of Math. (2)}, 175:909--948, 2012.

\bibitem{Chen-Nguyen-Xu:muskat-c1}
K.~Chen, Q.-H. Nguyen, and Y.~Xu.
\newblock The {M}uskat problem with {$C^1$} data.
\newblock {\em Trans. Amer. Math. Soc.}, 375(5):3039--3060, 2022.

\bibitem{Cheng-GraneroBelinchon-Shkoller:well-posedness-h2-muskat}
C.~H.~A. Cheng, R.~Granero-Belinch\'on, and S.~Shkoller.
\newblock Well-posedness of the {M}uskat problem with ${H}^{2}$ initial data.
\newblock {\em Advances in Mathematics}, 286:32 -- 104, 2016.

\bibitem{Choi-Jerison-Kim:one-phase-hele-shaw-lipschitz}
S.~Choi, D.~Jerison, and I.~Kim.
\newblock Regularity for the one-phase {H}ele-{S}haw problem from a {L}ipschitz
  initial surface.
\newblock {\em Amer. J. Math.}, 129(2):527--582, 2007.

\bibitem{Choi-Kim:waiting-time-hele-shaw}
S.~Choi and I.~Kim.
\newblock Waiting time phenomena of the {H}ele-{S}haw and the {S}tefan problem.
\newblock {\em Indiana Univ. Math. J.}, 55(2):525--551, 2006.

\bibitem{Constantin-Cordoba-Gancedo-RodriguezPiazza-Strain:muskat-global-2d-3d}
P.~Constantin, D.~C\'ordoba, F.~Gancedo, L.~Rodr\'iguez-Piazza, and R.~M.
  Strain.
\newblock On the {M}uskat problem: global in time results in 2{D} and 3{D}.
\newblock {\em Amer. J. Math.}, 138(6):1455--1494, 2016.

\bibitem{Constantin-Cordoba-Gancedo-Strain:global-existence-muskat}
P.~Constantin, D.~C\'{o}rdoba, F.~Gancedo, and R.~M. Strain.
\newblock On the global existence for the {M}uskat problem.
\newblock {\em J. Eur. Math. Soc. (JEMS)}, 15(1):201--227, 2013.

\bibitem{Constantin-Gancedo-Shvydkoy-Vicol:global-regularity-muskat-finite-slope}
P.~Constantin, F.~Gancedo, R.~Shvydkoy, and V.~Vicol.
\newblock Global regularity for 2{D} {M}uskat equations with finite slope.
\newblock {\em Ann. Inst. H. Poincar\'e Anal. Non Lin\'eaire},
  34(4):1041--1074, 2017.

\bibitem{Cordoba-Cordoba-Gancedo:interface-heleshaw-muskat}
A.~C{\'o}rdoba, D.~C{\'o}rdoba, and F.~Gancedo.
\newblock Interface evolution: the {H}ele-{S}haw and {M}uskat problems.
\newblock {\em Ann. of Math. (2)}, 173(1):477--542, 2011.

\bibitem{Cordoba-Cordoba-Gancedo:muskat-3d}
A.~C\'{o}rdoba, D.~C\'{o}rdoba, and F.~Gancedo.
\newblock Porous media: the {M}uskat problem in three dimensions.
\newblock {\em Anal. PDE}, 6(2):447--497, 2013.

\bibitem{Cordoba-Gancedo:contour-dynamics-3d-porous-medium}
D.~C{\'o}rdoba and F.~Gancedo.
\newblock Contour dynamics of incompressible 3-{D} fluids in a porous medium
  with different densities.
\newblock {\em Comm. Math. Phys.}, 273(2):445--471, 2007.

\bibitem{Cordoba-GomezSerrano-Zlatos:stability-shifting-muskat}
D.~C{\'o}rdoba, J.~G{\'o}mez-Serrano, and A.~Zlato{\v s}.
\newblock A note on stability shifting for the {M}uskat problem.
\newblock {\em Philosophical Transactions of the Royal Society of London A:
  Mathematical, Physical and Engineering Sciences}, 373(2050):20140278, 10,
  2015.

\bibitem{Cordoba-GomezSerrano-Zlatos:stability-shifting-muskat-II}
D.~C\'ordoba, J.~G\'omez-Serrano, and A.~Zlato{\v s}.
\newblock A note on stability shifting for the {M}uskat problem, {II}: {F}rom
  stable to unstable and back to stable.
\newblock {\em Anal. PDE}, 10(2):367--378, 2017.

\bibitem{Cordoba-Lazar:global-wellposedness-muskat-H32}
D.~C\'{o}rdoba and O.~Lazar.
\newblock Global well-posedness for the 2{D} stable {M}uskat problem in
  {$H^{3/2}$}.
\newblock {\em Ann. Sci. \'{E}c. Norm. Sup\'{e}r. (4)}, 54(5):1315--1351, 2021.

\bibitem{Deng-Lei-Lin:2d-muskat-monotone-data}
F.~Deng, Z.~Lei, and F.~Lin.
\newblock On the two-dimensional {M}uskat problem with monotone large initial
  data.
\newblock {\em Comm. Pure Appl. Math.}, 70(6):1115--1145, 2017.

\bibitem{Dong-Gancedo-Nguyen:one-phase-muskat}
H.~Dong, F.~Gancedo, and H.~Q. Nguyen.
\newblock Global well-posedness for the one-phase {M}uskat problem.
\newblock {\em Comm. Pure Appl. Math, to appear (arXiv:2103.02656)}, 2021.

\bibitem{Gancedo:survey-muskat}
F.~Gancedo.
\newblock A survey for the {M}uskat problem and a new estimate.
\newblock {\em SeMA J.}, 74(1):21--35, 2017.

\bibitem{Gancedo-GarciaJuarez-Patel-Strain:muskat-viscosity-global}
F.~Gancedo, E.~Garc\'{\i}a-Ju\'{a}rez, N.~Patel, and R.~M. Strain.
\newblock On the {M}uskat problem with viscosity jump: global in time results.
\newblock {\em Adv. Math.}, 345:552--597, 2019.

\bibitem{Gancedo-Lazar:global-wellposedness-muskat-3d}
F.~Gancedo and O.~Lazar.
\newblock Global well-posedness for the three dimensional {M}uskat problem in
  the critical {S}obolev space.
\newblock {\em Arch. Ration. Mech. Anal.}, 246(1):141--207, 2022.

\bibitem{GarciaJuarez-GomezSerrano-Nguyen-Pausader:self-similar-muskat}
E.~Garc\'{\i}a-Ju\'{a}rez, J.~G\'{o}mez-Serrano, H.~Q. Nguyen, and B.~Pausader.
\newblock Self-similar solutions for the {M}uskat equation.
\newblock {\em Adv. Math.}, 399:Paper No. 108294, 30, 2022.

\bibitem{GraneroBelinchon-Lazar:growth-muskat}
R.~Granero-Belinch\'{o}n and O.~Lazar.
\newblock Growth in the {M}uskat problem.
\newblock {\em Math. Model. Nat. Phenom.}, 15:Paper No. 7, 23, 2020.

\bibitem{Matioc:viscous-displacement-muskat}
B.-V. Matioc.
\newblock Viscous displacement in porous media: the {M}uskat problem in 2{D}.
\newblock {\em Trans. Amer. Math. Soc.}, 370(10):7511--7556, 2018.

\bibitem{Matioc:local-existence-muskat-hs}
B.-V. Matioc.
\newblock The {M}uskat problem in two dimensions: equivalence of formulations,
  well-posedness, and regularity results.
\newblock {\em Anal. PDE}, 12(2):281--332, 2019.

\bibitem{Muskat:porous-media}
M.~Muskat.
\newblock The flow of fluids through porous media.
\newblock {\em Journal of Applied Physics}, 8(4):274--282, 1937.

\bibitem{Nguyen:global-solutions-muskat-besov}
H.~Q. Nguyen.
\newblock Global solutions for the {M}uskat problem in the scaling invariant
  {B}esov space {$\dot {B}^1_{\infty,1}$}.
\newblock {\em Adv. Math.}, 394:Paper No. 108122, 28, 2022.

\bibitem{Nguyen-Pausader:paradifferential-muskat}
H.~Q. Nguyen and B.~Pausader.
\newblock A paradifferential approach for well-posedness of the {M}uskat
  problem.
\newblock {\em Arch. Ration. Mech. Anal.}, 237(1):35--100, 2020.

\bibitem{Nguyen-Tice:one-phase-muskat-traveling}
H.~Q. Nguyen and I.~Tice.
\newblock Traveling wave solutions to the one-phase {M}uskat problem: existence
  and stability.
\newblock {\em Preprint, arXiv:2211.06286}, 2022.

\bibitem{Shi:regularity-muskat}
J.~Shi.
\newblock Regularity of {S}olutions to the {M}uskat {E}quation.
\newblock {\em Arch. Ration. Mech. Anal.}, 247(3):Paper No. 36, 2023.

\bibitem{Siegel-Caflisch-Howison:global-existence-muskat}
M.~Siegel, R.~E. Caflisch, and S.~Howison.
\newblock Global existence, singular solutions, and ill-posedness for the
  {M}uskat problem.
\newblock {\em Comm. Pure Appl. Math.}, 57(10):1374--1411, 2004.

\bibitem{Yi:local-muskat}
F.~Yi.
\newblock Local classical solution of {M}uskat free boundary problem.
\newblock {\em J. Partial Differential Equations}, 9(1):84--96, 1996.

\bibitem{Yi:global-muskat}
F.~Yi.
\newblock Global classical solution of {M}uskat free boundary problem.
\newblock {\em J. Math. Anal. Appl.}, 288(2):442--461, 2003.

\end{thebibliography}

\vskip 1cm

\begin{tabular}{l}
  \textbf{Eduardo Garc\'ia-Ju\'arez} \\
  {Departamento de An\'alisis Matem\'atico} \\
  {Universidad de Sevilla} \\
  {C/ Tarfia s/n, Campus Reina Mercedes} \\
  {41012, Sevilla, Spain} \\
  {Email: egarcia12@us.es} \\ \\
  \textbf{Javier G\'omez-Serrano}\\
  {Department of Mathematics} \\
  {Brown University} \\
  {Kassar House, 151 Thayer St.} \\
  {Providence, RI 02912, USA} \\
  {Email: javier\_gomez\_serrano@brown.edu} \\ \\
  \textbf{Susanna V. Haziot}\\
  {Department of Mathematics} \\
  {Brown University} \\
  {Kassar House, 151 Thayer St.} \\
  {Providence, RI 02912, USA} \\
  {Email: susanna\_haziot@brown.edu} \\ \\
  \textbf{Beno\^it Pausader}\\
  {Department of Mathematics} \\
  {Brown University} \\
  {Kassar House, 151 Thayer St.} \\
  {Providence, RI 02912, USA} \\
  {Email: benoit\_pausader@brown.edu} \\ \\
\end{tabular}

\end{document}